\documentclass[11pt]{article}
\usepackage[total={7in, 8in}]{geometry}
\usepackage{graphicx}
\usepackage[margin=0.5in]{caption}
\usepackage{comment}
\usepackage{amsmath, amsthm, latexsym, amssymb, color,cite,enumerate, physics, framed}
\usepackage{caption,subcaption,verbatim, empheq, cancel}
\usepackage[inline]{enumitem}
\usepackage{mathtools}
\usepackage{appendix}
\usepackage[T1]{fontenc}
\usepackage[utf8]{inputenc}
\usepackage{authblk}

\usepackage{etoolbox}
\newcommand{\addQEDstyle}[2]{\AtBeginEnvironment{#1}{\pushQED{\qed}
\renewcommand{\qedsymbol}{#2}}\AtEndEnvironment{#1}{\popQED}}

\newtheorem{theorem}{Theorem}[section]
\newtheorem{lemma}[theorem]{Lemma}
\newtheorem{definition}[theorem]{Definition}
\newtheorem{corollary}[theorem]{Corollary}
\newtheorem{proposition}[theorem]{Proposition}

\newtheoremstyle{example}
{5pt}
{5pt}
{}
{}
{\bf}
{:}
{.5em}
{}

\theoremstyle{example}
\newtheorem{example}{Example}
\addQEDstyle{example}{$\triangle$}

\newtheorem{remark}{Remark}

\newcommand*{\myproofname}{Proof}
\newenvironment{subproof}[1][\myproofname]{\begin{proof}[#1]\renewcommand*{\qedsymbol}{$\mathbin{/\mkern-6mu/}$}}{\end{proof}}
\newenvironment{exproof}[1][\myproofname]{\begin{proof}[#1]\renewcommand*{\qedsymbol}{$\square$}}{\end{proof}}
\renewcommand\Re{\operatorname{Re}}

\newcommand\End{\operatorname{End}} 
\newcommand\Gl{\operatorname{Gl}} 
\newcommand\OdR{\mbox{O}(\mathbb{R}^d)} 
\newcommand\Sym{\operatorname{Sym}}
\newcommand\Exp{\operatorname{Exp}}
\newcommand\diag{\operatorname{diag}}

\newcommand\Spec{\operatorname{Spec}}
\renewcommand\det{\operatorname{det}}

\usepackage{hyperref}
\usepackage{tensor}
\usepackage{xcolor}
\hypersetup{
	colorlinks,
	linkcolor={black!50!black},
	citecolor={blue!50!black},
	urlcolor={blue!80!black}
}

\title{On-diagonal asymptotics for heat kernels of a class of inhomogeneous partial differential operators}

\author[1]{Evan Randles\thanks{Corresponding author: evan.randles@colby.edu}\thanks{Permanent Address: Department of Mathematics, Colby College, Waterville, Maine 04901, USA}}
\author[1]{Laurent Saloff-Coste}
\affil[1]{Department of Mathematics, Cornell University, Ithaca, New York 14853-4201, USA}

\date{}
\begin{document}
\maketitle

\begin{abstract}
We consider certain constant-coefficient differential operators on $\mathbb{R}^d$ with positive-definite symbols. Each such operator $\Lambda$ with symbol $P$ defines a semigroup $e^{-t\Lambda}$ , $t>0$ , admitting a convolution kernel $H^t_P$ for which the large-time behavior of $H_P^t(0)$ cannot be deduced by basic scaling arguments. The simplest example has symbol $P(\xi)=(\eta+\zeta^2)^2+\eta^4$ , $\xi=(\eta,\zeta)\in \mathbb{R}^2$ . We devise a method to establish large-time asymptotics of $H^t_P(0)$ for several classes of examples of this type and we show that these asymptotics are preserved by perturbations by certain higher-order differential operators. For the $P$ just given, it turns out that $H^t_P(0)\sim c_Pt^{-5/8}$ as $t\to\infty$ . We show how such results are relevant to understand the convolution powers of certain complex functions on $\mathbb Z^d$ . Our work represents a first basic step towards a good understanding of the semigroups associated with these operators. Obtaining meaningful off-diagonal upper bounds for $H_P^t$ remains an interesting challenge.
\end{abstract}


\noindent{\small\bf 2020 Mathematics Subject Classification:} Primary 35K08, 35K25; Secondary 47D06, 42B99.\\

\noindent{\small\bf Keywords:} On-diagonal heat kernel asymptotics, inhomogeneous partial differential operators, convolution powers

\section{Introduction}

On $\mathbb{R}^2$, consider the constant-coefficient partial differential operator
\begin{equation*}
\Lambda=\partial_{x_1}^4+\partial_{x_2}^4+2i\partial_{x_1}\partial_{x_2}^2-\partial_{x_1}^2
\end{equation*}
and its symbol
\begin{equation*}
P(\xi)=(\eta+\zeta^2)^2+\eta^4
\end{equation*}
defined for $\xi=(\eta,\zeta)\in\mathbb{R}^2$. It is evident that $\Lambda$ is a non-negative, symmetric, and fourth-order elliptic operator. Thus, when defined initially on the set of compactly supported smooth functions, $C_c^\infty(\mathbb{R}^2)$, $\Lambda$ extends uniquely to a non-negative self-adjoint operator on $L^2(\mathbb{R}^2)$ which, by an abuse of notation, we denote by $\Lambda$. Via the spectral calculus or the Hille-Yosida construction, $-\Lambda$ generates a continuous one-parameter semigroup of contractions on $L^2(\mathbb{R}^2)$ which is denoted by $\{e^{-t\Lambda}\}$ and called the \textit{heat semigroup associated to $\Lambda$}. Thanks to the Fourier transform\footnote{On $\mathbb{R}^d$, we shall take the Fourier transform $\mathcal{F}$ and inverse Fourier transform $\mathcal{F}^{-1}$ to be given by $\mathcal{F}(f)(\xi)=\widehat{f}(\xi)=\int_{\mathbb{R}^d}f(x)e^{ix\cdot\xi}\,dx$ for $f\in L^2(\mathbb{R}^d)\cap L^1(\mathbb{R}^d)$ and $\mathcal{F}^{-1}(g)(x)=\check{g}(x)=\frac{1}{(2\pi)^d}\int_{\mathbb{R}^d}g(\xi)e^{-ix\cdot\xi}\,d\xi$ for $g\in L^2(\mathbb{R}^d)\cap L^1(\mathbb{R}^d)$, respectively.}, this semigroup has the integral representation
\begin{equation}\label{eq:HeatKernelRep}
(e^{-t\Lambda}f)(x)=\int_{\mathbb{R}^2}H_P^t(x-y)f(y)\,dy
\end{equation}
for each $f\in L^2(\mathbb{R}^2)$ where $H_P=H_P^{(\cdot)}(\cdot)$ is called the \textit{heat kernel associated to $\Lambda$} and is given by
\begin{equation}\label{eq:HeatKernelIntro}
H_P^t(x)=\frac{1}{(2\pi)^2}\int_{\mathbb{R}^2}e^{-tP(\xi)}e^{-ix\cdot\xi}\,d\xi
\end{equation}
for $t>0$ and $x\in\mathbb{R}^2$. For its central role in the analysis surrounding $\Lambda$, including its spectral theory, associated Sobolev inequalities, and properties of the semigroup $\{e^{-t\Lambda}\}$, we are interested in the behavior and properties of the heat kernel $H_P$. As we demonstrate below, this curiosity is further spurred by the appearance of $H_P$ as a scaled limit of convolution powers of complex-valued functions on $\mathbb{Z}^2$, just as the Gaussian density appears as the scaled limit in the local (central) limit theorem \cite{RSC17}.

Consider the function $\phi:\mathbb{Z}^2\to\mathbb{C}$ defined by 
\begin{equation}\label{eq:IntroPhi}
\phi=\frac{1}{3840000}\left(\frac{1}{12}\phi_1+\frac{1}{8}\phi_2+\frac{1}{3}\phi_3\right)
\end{equation}
where
\begin{equation*}
\phi_1(x)=
\begin{cases}
41375061 & (x_1,x_2)=(0,0)\\
1080000\pm 969232 i & (x_1,x_2)=\pm(1,0)\\
    -165072 &(x_1,x_2)=\pm(2,0)\\
    72000\pm 9024 & (x_1,x_2)=\mp(3,0)\\
    -38256 &(x_1,x_2)=\pm (4,0)\\
    0 & \mbox{else}
\end{cases},
\hspace{1cm}
\phi_2(x)=
\begin{cases}
1228800 & (x_1,x_2)=(0,\pm 1)\\
    -286328 & (x_1,x_2)=(0,\pm 2)\\
    -9524 & (x_1,x_2)=(0,\pm 4)\\
    2232 &(x_1,x_2)=(0,\pm 6)\\
    -179 & (x_1,x_2)=(0,\pm 8)\\
    0 & \mbox{else}
\end{cases}
\end{equation*}
and
\begin{equation*}
    \phi_3(x)=
    \begin{cases}
    \pm115200i & (x_1,x_2)=(\mp 1,1), (\mp 1,-1)\\
    \pm6939 i &(x_1,x_2)=(\mp 1,2), (\mp 1,-2)\\
    216 &(x_1,x_2)=(\pm 2,2),(\pm2 ,-2)\\
    1128 i &(x_1,x_2)=(\pm 3,2),(\pm 3,-2)\\
    \pm1062 i &(x_1,x_2)=(\pm 1,4),(\pm 1,-4)\\
    -54 &(x_1,x_2)=(\pm 2, 4),(\pm2,-4)\\
    \pm 77 i &(x_1,x_2)=(\mp 1,6),(\mp 1,-6)\\
    0 &\mbox{else}
    \end{cases}
\end{equation*}
for $x=(x_1,x_2)\in\mathbb{Z}^2$.
With this function, we define its iterated convolution powers $\phi^{(n)}:\mathbb{Z}^2\to\mathbb{C}$ by putting $\phi^{(1)}=\phi$ and, for $n\geq 2$, 
\begin{equation*}
    \phi^{(n)}(x)=\sum_{y\in\mathbb{Z}^2}\phi^{(n-1)}(x-y)\phi(y)
\end{equation*}
for $x\in\mathbb{Z}^2$. Motivated by applications to data smoothing and partial differential equations, we are interested in the asymptotic behavior of $\phi^{(n)}$ as $n\to\infty$. Given the nature of convolution, it is reasonable to expect that the mass of $\phi^{(n)}$ ``spreads out'' on $\mathbb{Z}^2$ as $n$ increases, however, exactly how it does this is not a priori clear. In Section \ref{sec:ConvPower}, we show that, for large $n$, $\phi^{(n)}$ is well approximated by the heat kernel $H_P$ evaluated at $t=n/100$. This so-called local limit theorem is illustrated in Figure \ref{fig:Intro} and it motivates us to understand $H_P^t(\cdot)$ for large $t$.

\begin{figure}[h!]
\begin{center}
\resizebox{\textwidth}{!}{
	    \begin{subfigure}[5cm]{0.5\textwidth}
		\includegraphics[width=\textwidth]{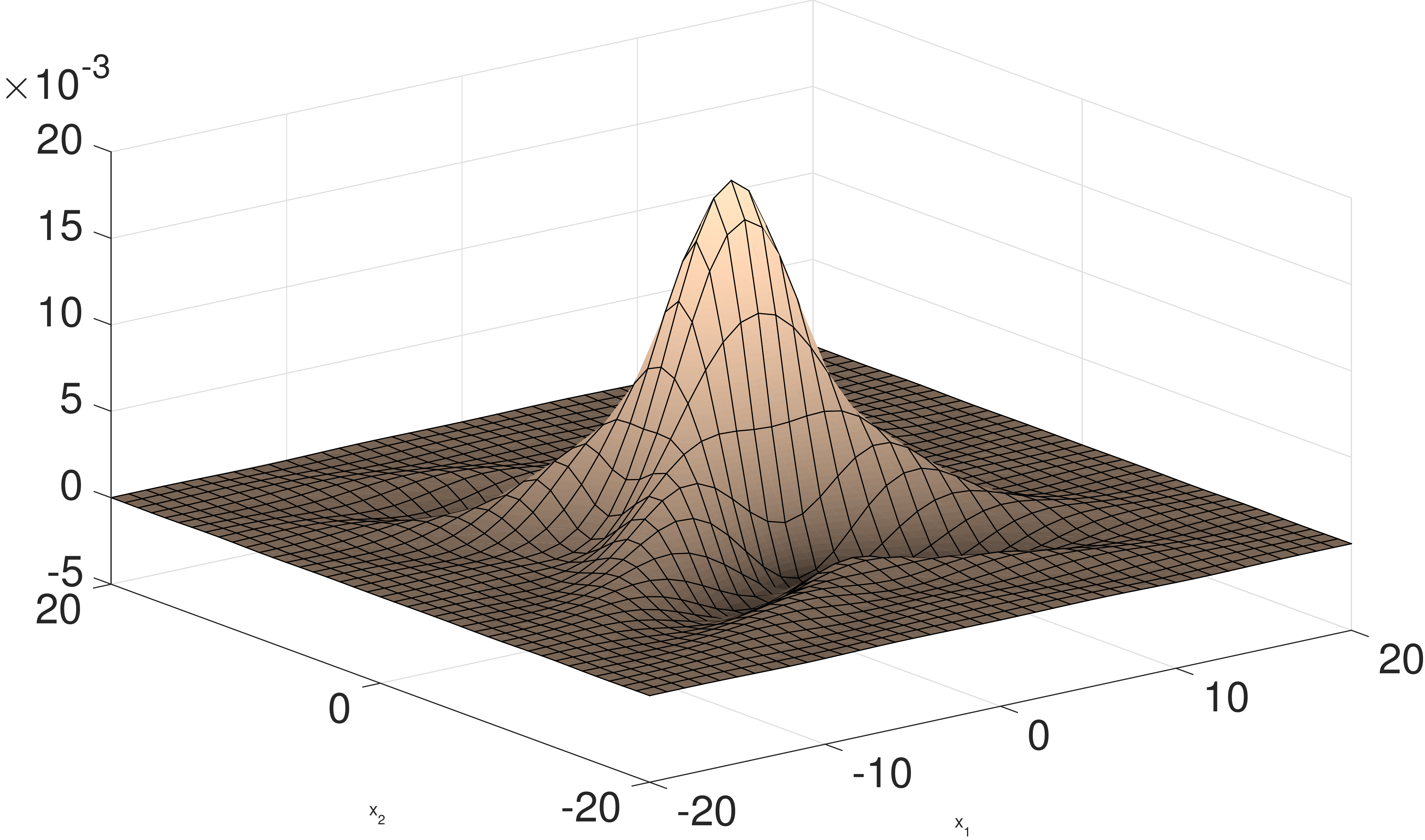}
		\caption{$\Re(\phi^{(n)})$ for $n=1000$}
		\label{fig:IntroPhi}
	    \end{subfigure}
\begin{subfigure}[6cm]{0.5\textwidth}
		\includegraphics[width=\textwidth]{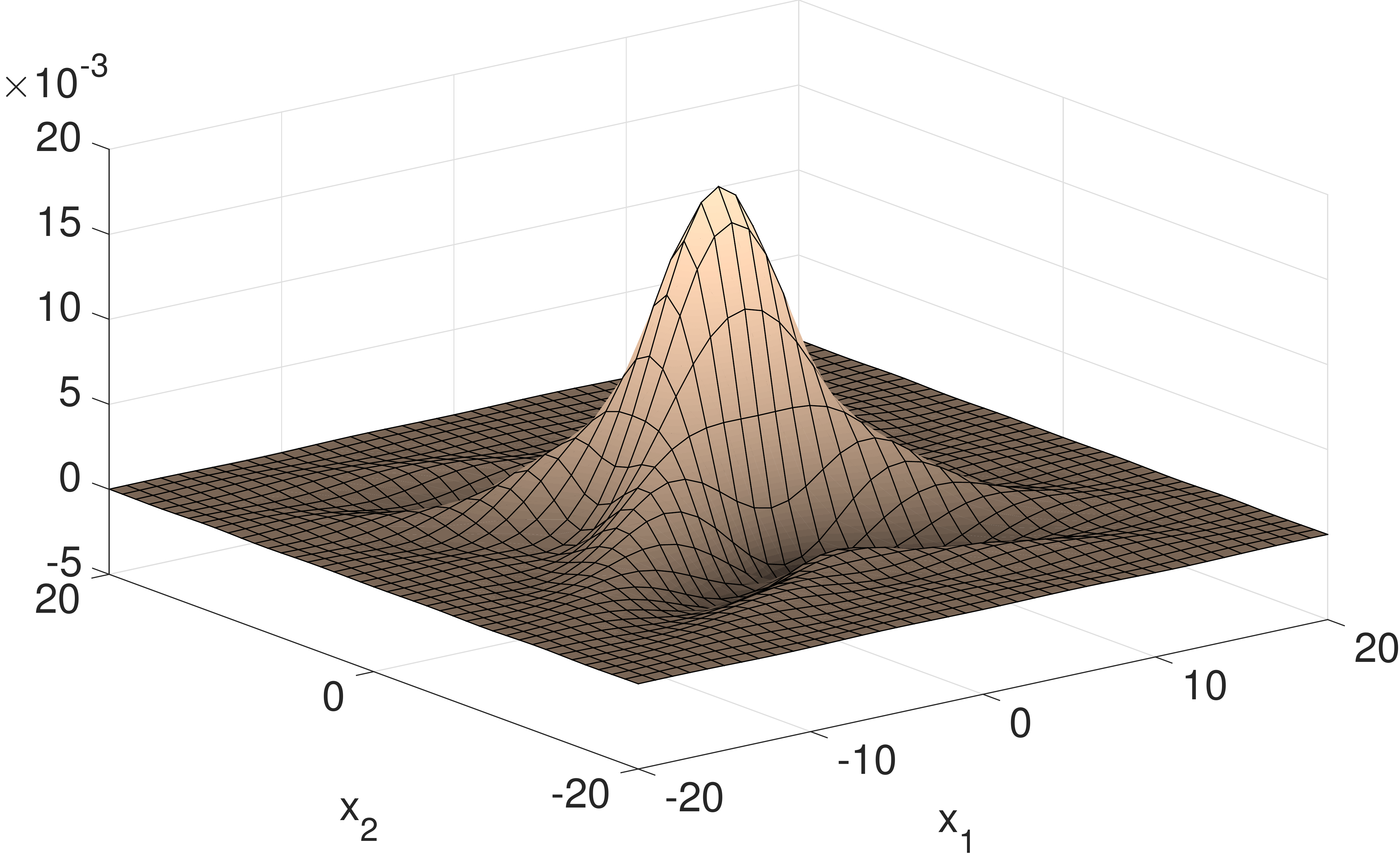}
		\caption{$\Re(H_P^{t})$ for $t=n/100$}
		\label{fig:IntroAttractor}
	    \end{subfigure}}
\caption{The graphs of $\Re(\phi^{(n)})$ and $\Re(H_{P})$ for $n=1000$.}
\label{fig:Intro}
\end{center}
\end{figure}

In light of the fast growth of $P(\xi)$ as $\xi\to\infty$, it is not difficult to show that, for each $t>0$, $x\mapsto H_P^t(x)$ is a Schwartz function. Given that $\Lambda$ is a fourth-order (uniformly) elliptic operator, $H_P$ is known to satisfy the following estimate: There are positive constants $A, B$, and $C$ for which
\begin{equation}\label{eq:OffDiagIntro}
\abs{H_P^t(x)}\leq \frac{A}{t^{1/2}}\exp\left(-Bt\abs{\frac{x}{t}}^{4/3}+Ct\right)
\end{equation}
for $t>0$ and $x\in\mathbb{R}^2$ \cite{Ei69,Da95}. Using this so-called off-diagonal estimate, it is easily verified that the semigroup $\{e^{-t\Lambda}\}$, initially defined on $L^2(\mathbb{R}^2)$, extends uniquely to a continuous semigroup $\{e^{-t\Lambda_p}\}$ on  $L^p(\mathbb{R}^2)$ for all $1\leq p\leq\infty$. In fact, this estimate also guarantees that the spectrum of $\Lambda=\Lambda_p$ independent of $p$, c.f., \cite[Theorem 20]{Da95}. 

For the present article, we shall focus our attention on the so-called on-diagonal behavior of $H_P^t(x)$. That is, we are interested in the behavior of
\begin{equation*}
    \varphi(t):=H_P^t(0)=\sup_{x\in\mathbb{R}^2}\abs{H_P^t(x)}
\end{equation*}
defined for $t>0$. Given that $P$ is non-negative, $\varphi$ is evidently a non-increasing function on $(0,\infty)$. In view of the representation \eqref{eq:HeatKernelRep}, we observe immediately that
\begin{equation*}
    \|e^{-t\Lambda}\|_{1\to\infty}=\sup_{x\in\mathbb{R}^d}\abs{H_P^t(x)}=\varphi(t)
\end{equation*}
for $t>0$, i.e., $e^{-t\Lambda}$ is a bounded operator from $L^1(\mathbb{R}^2)$ to $L^\infty(\mathbb{R}^2)$ with operator norm precisely $\varphi(t)$. By virtue of Plancherel's theorem,
\begin{equation*}
    \|H_P^t\|_2^2=\frac{1}{(2\pi)^2}\|e^{-tP}\|_2^2=\frac{1}{(2\pi)^2}\int_{\mathbb{R}^2}e^{-2tP(\xi)}\,d\xi=H_P^{2t}(0)=\varphi(2t)
\end{equation*}
for $t>0$. From this it follows that $e^{-t\Lambda}$ is a bounded operator from $L^2(\mathbb{R}^2)$ to $L^\infty(\mathbb{R}^2)$ with
\begin{equation*}
    \|e^{-t\Lambda}\|_{2\to\infty}=\|H_P^t\|_2=\sqrt{\varphi(2t)}
\end{equation*}
for $t>0$. In other words, the semigroup $\{e^{-t\Lambda}\}$ is ultracontractive with bound characterized by $\varphi$. It is well known that the ultracontractivity of the semigroup $\{e^{-t\Lambda}\}$ can be used to establish various Sobolev-type inequalities associated to $\Lambda$, e.g., Sobolev, Nash, Gagliardo-Nirenberg \cite{CKS87,Co96,Da89,RSC20,SC09,Va85}. To do this, however, it is necessary to have a good understanding of the function $\varphi(t)$. Further, given our motivation to understand the convolution powers of the example $\phi$ given above, we are especially interested in understanding the function $\varphi(t)$ for large values of $t$ and, in fact, we will find that $\|\phi^{(n)}\|_\infty\asymp\varphi(n)$ for\footnote{Here and in what follows, for real-valued functions $f$ and $g$ defined, at least, on a non-empty set $X$, we shall write $f(x)\asymp g(x)$ for $x\in X$ to mean that there are positive constants $C$ and $C'$ for which $Cg(x)\leq f(x)\leq C'g(x)$ for all $x\in X$.} $n\in\mathbb{N}_+:=\{1,2,\dots\}$.

Analyzing $\varphi$ in small time is straightforward. Using \eqref{eq:OffDiagIntro}, we see that
\begin{equation}\label{eq:IntroSmallTimeEst}
\varphi(t)\leq \frac{C}{t^{1/2}}
\end{equation}
for all $0<t\leq 1$ where $C$ is a positive constant. We can, of course, do better by employing the following elementary scaling argument: Upon noting that
\begin{equation}\label{eq:IntroScaledLimit}
\lim_{t\to 0}tP(t^{-1/4}\xi)=\lim_{t\to 0}tP(t^{-1/4}\eta,t^{-1/4}\zeta)=\eta^4+\zeta^4
\end{equation}
for each $\xi=(\eta,\zeta)\in\mathbb{R}^2$, the change of variables $\xi\mapsto t^{-1/4}\xi=(t^{-1/4}\eta,t^{-1/4}\zeta)$ in \eqref{eq:HeatKernelIntro} yields
\begin{equation}\label{eq:IntroSmallTimeTrueAsymptotic}
    \lim_{t\to 0}t^{1/2}\varphi(t)=\frac{1}{(2\pi)^2}\int_{\mathbb{R}^2}e^{-(\eta^4+\zeta^4)}\,d\eta\,d\zeta=\frac{1}{\pi^2}\Gamma(5/4)^2\approx 0.0832420
\end{equation}
where $\Gamma$ is Euler's Gamma function. Consequently, $\varphi(t)\asymp t^{-1/2}$ for $0<t\leq 1$. It is noteworthy that scaled limit \eqref{eq:IntroScaledLimit} at $t=0$ ``picks out'' the homogeneous fourth-order polynomial $P_0(\xi)=\eta^4+\zeta^4$ which is precisely the principal symbol of $\Lambda$. Though not directly related, we refer the reader to the works of Evgrafov and Postnikov \cite{EP70}, and Tintarev \cite{Ti82} who established short-time off-diagonal asymptotics similar to the right hand side in \eqref{eq:OffDiagIntro} for heat kernels of higher-order elliptic operators. See also the related works of Barbatis and Davies \cite{BD96,Ba98,Ba01,Da95a}.

Following Davies \cite{Da95a} and driven by the motivations previously discussed, we seek to understand the behavior of $\varphi(t)$ in large time. For this goal, it is clear that the estimate \eqref{eq:OffDiagIntro} is useless for, in contrast to $\varphi(t)$, the function $t\mapsto t^{-1/2}\exp(Ct)$ is increasing for $t>1/2C$. In looking back through the scaling argument above, we wonder: Perhaps there is a rescaling of the symbol $P$ that will yield an asymptotic (or simply a useful estimate) for $\varphi$ in large time. In fact, this type of approach was taken in \cite{Da95a} to characterize the large-time on-diagonal behavior of the heat kernel associated to the fourth-order elliptic operator $f\mapsto \frac{d^4f}{dx^4}-\frac{d^2f}{dx^2}$ in one spatial dimension. We note that Davies' results addressed a question posed by M. van den Berg concerning the necessity of the term $Ct$ in off-diagonal estimates of the form \eqref{eq:OffDiagIntro}. Given real numbers $\alpha$ and $\beta$, we make the change of variables $\xi=(\eta,\zeta)\mapsto (t^{-\alpha}\eta,t^{-\beta}\zeta)$ in the integral \eqref{eq:HeatKernelIntro} to see that
\begin{equation*}
    \varphi(t)=\frac{t^{-(\alpha+\beta})}{(2\pi)^2}\int_{\mathbb{R}^2}e^{-tP(t^{-\alpha}\eta,t^{-\beta}\zeta)}\,d\xi
\end{equation*}
for $t>0$. With the aim of mimicking our small-time scaling argument, we seek values of $\alpha$ and $\beta$ for which $tP(t^{-\alpha}\eta,t^{-\beta}\zeta)$ is sufficiently well behaved as $t\to\infty$. Upon noting that
\begin{equation*}
    tP(t^{-\alpha}\eta,t^{-\beta}\zeta)=t^{1-2\alpha}\eta^2+t^{1-\alpha-2\beta}2\eta\zeta^2+t^{1-4\alpha}\eta^4+t^{1-4\beta}\zeta^4
\end{equation*}
for $t>0$ and $(\eta,\zeta)\in\mathbb{R}^2$ and considering all possibilities for $\alpha$ and $\beta$, we find that
\begin{equation*}
   \lim_{t\to\infty}tP(t^{-\alpha}\eta,t^{-\beta}\zeta)=\begin{cases}
    (\eta+\zeta^2)^2 &\alpha=1/2,\,\beta=1/4\\
    \eta^2 &\alpha=1/2,\,\beta>1/4\\
    \zeta^4 &\alpha>1/2, \beta=1/4\\
    0\,\,\mbox{or}\,\,\infty &\mbox{otherwise}
    \end{cases}
\end{equation*}
for almost every $(\eta,\zeta)\in\mathbb{R}^2$. Given that none of the functions $(\eta,\zeta)\mapsto e^{-(\eta+\zeta^2)^2}$, $(\eta,\zeta)\mapsto e^{-\eta^2}$, or $(\eta,\zeta)\mapsto e^{-\zeta^4}$ are integrable, it follows that
\begin{equation*}
    \lim_{t\to \infty}\int_{\mathbb{R}^2}e^{-tP(t^{-\alpha}\eta,t^{-\beta}\zeta)}\,d\xi
\end{equation*}
is either $0$ or $\infty$ for all possible cases of $\alpha$ and $\beta$. Consequently, the argument we used to establish the small-time asymptotics for $\varphi$ is not helpful to us. In fact, it can be shown that no ``reasonable'' scaling, which is linear in $\xi$, can be used to deduce the asymptotic behavior of $\varphi(t)$ as $t\to \infty$. These observations are tied, in some sense, to the absence of a well-behaved lower-order component of $P$ characterizing the behavior of $\varphi(t)$ in large time just as the principal symbol $\eta^4+\zeta^4$ does for small time. Without a tractable scaling method for large time, the nature of the decay of $\varphi(t)$ as $t\to\infty$ -- be it exponential, polynomial, or otherwise -- is not a priori clear. Our main theorem, Theorem \ref{thm:OnDiagonal}, yields the (to us) surprising conclusion that $\varphi(t)\asymp t^{-5/8}$ for $t\geq 1$. By an application of Theorem \ref{thm:TrueAsymptotic}, we are able to obtain the ``true asymptotic'',
\begin{equation*}
\lim_{t\to\infty}t^{5/8}\varphi(t)=\frac{1}{2\pi^{3/2}}\Gamma(9/8)\approx 0.0845624.
\end{equation*}

Taking this example as motivation, we introduce and study a class of symbols on $\mathbb{R}^d$ for which it is possible to establish on-diagonal heat kernel asymptotics in large (and small) time. For these examples, the small-time behavior is often characterized by a principal homogeneous term and we focus on the interesting question of large-time behavior. In particular, for positive integers $a$ and $b$ with $d=a+b$, we consider a polynomial of the form
\begin{equation*}
    P(\xi)=P_1(\eta+Q(\zeta))+P_2(\eta)
\end{equation*}
for $\xi=(\eta,\zeta)\in\mathbb{R}^d=\mathbb{R}^a\times\mathbb{R}^b$ where $P_1$ and $P_2$ are real-valued positive homogeneous polynomials in sense of \cite{RSC17} and $Q:\mathbb{R}^b\to\mathbb{R}^a$ is a so-called multivariate nondegenerate homogeneous polynomial. The notions of positive homogeneous and multivariate nondegenerate homogeneous are presented in Section \ref{sec:Hom}; we remark that the prototypical example of a positive homogeneous polynomial is a positive-definite and homogeneous semi-elliptic polynomial \cite{Ho83,RSC17,RSC17a,RSC20}. The homogeneous structure (and order) of $P_1$ and $P_2$ need not coincide and so $P$ is generally inhomogeneous and further, as seen in our motivating example, no rescaling of $P$ in large time yields a tractable homogeneous term.  To the constant-coefficient operator $\Lambda$ on $\mathbb{R}^d$ with symbol $P$, provided that $P$ grows sufficiently fast as $\xi\to\infty$ (which will always be the case for us), there corresponds a heat semigroup $\{e^{-t\Lambda}\}$ with heat kernel given by
\begin{equation*}
    H_P^t(x)=\frac{1}{(2\pi)^d}\int_{\mathbb{R}^d}e^{-tP(\xi)}e^{-ix\cdot\xi}\,d\xi
\end{equation*}
for $t>0$ and $x\in\mathbb{R}^d$ and, with this, we define $\varphi(t)=H_P^t(0)$ for $t>0$. Under certain hypotheses concerning $P_1$, $P_2$, and $Q$, our main theorem (Theorem \ref{thm:OnDiagonal}) gives positive numbers $\mu_0$ and $\mu_\infty$ for which
\begin{equation*}
    \varphi(t)\asymp \begin{cases}
    t^{-\mu_0} & 0<t\leq 1\\
    t^{-\mu_\infty} & t\geq 1
    \end{cases}
\end{equation*}
for $t>0$; in particular, our result describes the elusive behavior of $\varphi(t)$ in large time. Under one addition hypothesis, we also show that the limits
\begin{equation*}
\lim_{t\to0}t^{\mu_0}\varphi(t)\hspace{1cm}\mbox{and}\hspace{1cm}\lim_{t\to\infty}t^{\mu_\infty}\varphi(t)
\end{equation*}
exist and are positive computable numbers depending only on $P_1$, $P_2$ and $Q$; this is Theorem \ref{thm:TrueAsymptotic}. We note that Theorems \ref{thm:OnDiagonal} and \ref{thm:TrueAsymptotic} are stated in terms of positive homogeneous functions $P_1$ and $P_2$ (in the sense of \cite{BR22}) and a multivariate nondegenerate homogeneous function $Q$ (introduced in Section \ref{sec:Hom}) and, correspondingly, $P$ need not be a polynomial nor $H_P$ correspond to a constant-coefficient partial differential operator. Following Section \ref{sec:OnDiagonalAsymptotics}, we treat a perturbation theory in which $P$ is replaced by $P+R$ where $R(\xi)=o(P(\xi))$ as\footnote{This ``little-o'' notation means that, for each $\epsilon>0$, there is an open neighborhood $\mathcal{O}\subseteq\mathbb{R}^d$ of $0$ for which $\abs{R(\xi)}\leq \epsilon P(\xi)$ whenever $\xi\in\mathcal{O}$.} $\xi\to 0$ and, under certain conditions, we show that $H_{P+R}^t(0)\asymp \varphi(t)\asymp t^{-\mu_\infty}$ for $t\geq 1$; this is Theorem \ref{thm:Perturbation}. Further, viewing it essentially as a perturbation problem, we then apply our methods to the related problem of determining the asymptotic behavior of the convolution powers of certain complex-valued functions $\phi$ on $\mathbb{Z}^d$. Our results in this direction, Theorem \ref{thm:LLT} and Corollary \ref{cor:ConvSupNorm}, describe the asymptotic behavior of $\phi^{(n)}$ in form of local limit theorems and sup-norm asymptotics. Our theory provides an inhomogeneous counterpart to the homogeneous theory developed in \cite{RSC17}, \cite{BR22}, and \cite{Ra22}. Specifically, Theorem \ref{thm:LLT} can be compared to Theorem 1.6 in \cite{RSC17} and Theorems 1.9 and 3.8 of \cite{Ra22} and Corollary \ref{cor:ConvSupNorm} can be compared to Theorem 4.1 of \cite{RSC17}, Theorem 3.2 of \cite{BR22}, and Theorem 3.1 of \cite{Ra22}. Applying our results to the $\phi$ discussed in this introduction, we find that $\|\phi^{(n)}\|_\infty\asymp n^{-5/8}$ for $n\in\mathbb{N}_+$ and obtain the value of the (existent) limit, $\lim_{n\to\infty}n^{5/8}\|\phi^{(n)}\|_\infty$.

The entire theory developed in this article has a parallel version where the large-time behavior of $\varphi(t)$ is easy to compute while the small-time behavior is unclear. In fact, a key to our proof of Theorem \ref{thm:OnDiagonal} is that the heat kernel $H_P$ associated to $\Lambda$ agrees on the diagonal with the heat kernel $H_{\widetilde{P}}$ associated to a ``dual'' constant-coefficient partial differential operator $\widetilde{\Lambda}$, i.e.,
\begin{equation*}
    \varphi(t)=H_P^t(0)=H_{\widetilde{P}}^t(0)=\widetilde{\varphi}(t)
\end{equation*}
for all $t>0$. The utility of this correspondence is that the large-time behavior of $\widetilde{\varphi}$ is easily and directly computed. For the fourth-order elliptic operator $\Lambda$ considered in this introduction,
\begin{equation*}
    \widetilde{\Lambda}=-\partial_{x_1}^2+\partial_{x_1}^4+4i\partial_{x_1}^3\partial_{x_2}^2-6\partial_{x_1}^2\partial_{x_2}^4-4i\partial_{x_1}\partial_{x_2}^6+\partial_{x_2}^8
\end{equation*}
which is not an elliptic operator, nor is it semi-elliptic or even hypoelliptic. Consequently, the heat kernel corresponding to $\widetilde{\Lambda}$, especially for small time, is not well understood. Akin to the fact that $\Lambda$ has elliptic principal part $\Lambda_0=\partial_{x_1}^4+\partial_{x_2}^4$ which determines the small-time decay of $\varphi(t)$, $\widetilde{\Lambda}$ has ``lowest-order'' component which is well behaved and determines the behavior of $\widetilde{\varphi}(t)$ in large time. This component, which could be called the principal symbol at infinity, is the operator
\begin{equation*}
    \widetilde{\Lambda}_\infty=-\partial_{x_1}^2+\partial_{x_2}^8
\end{equation*}
which is semi-elliptic and determines the $5/8=1/2+1/8$ exponent of polynomial decay of $\widetilde{\varphi}=\varphi$ for large time. In view of this, the authors see evidence for a useful notion of large-time semi-ellipticity for operators and a theory surrounding it which is akin (and perhaps dual to) the standard theory of elliptic/semi-elliptic operators. In Section \ref{sec:Discussion}, we discuss this and future directions of the theory presented in this article.

\section{Homogeneous Functions}\label{sec:Hom}

In this section, we give a brief account of the theory of positive homogeneous functions (presented more fully in \cite{BR22}) and introduce a multivariate generalization of such functions, which we will call nondegenerate multivariate homogeneous functions. For a positive integer $d$, we shall denote by $\End(\mathbb{R}^d)$ the set of linear endomorphisms of $\mathbb{R}^d$ and by $\Gl(\mathbb{R}^d)$ the corresponding subset of automorphisms. We shall take $\End(\mathbb{R}^d)$ to be equipped with the operator norm $\|\cdot\|$ inherited from the standard Euclidean norm $\abs{\cdot}$ on $\mathbb{R}^d$. For a given $E\in\End(\mathbb{R}^d)$, we define $T:(0,\infty)\to\Gl(\mathbb{R}^d)$ by
\begin{equation*}
T_t=t^E=\exp((\ln t) E)=\sum_{k=0}^\infty \frac{(\ln t)^k}{k!}E^k
\end{equation*}
for $t>0$. It is straightforward to verify that $T$ is a Lie group homomorphism from the set of positive real numbers under multiplication into $\Gl(\mathbb{R}^d)$. The collection $\{T_t:t>0\}=\{t^E:t>0\}$, which we view both as a set and as a subgroup of $\Gl(\mathbb{R}^d)$, is called \textit{the continuous one-parameter group generated by $E$}; it will usually be written $\{t^E\}_{t>0}$ or simply $\{t^E\}$. It is a standard fact that every continuous one-parameter (sub)group of $\{T_t\}\subseteq \Gl(\mathbb{R}^d)$ is of this form, i.e., is generated by some $E\in\End(\mathbb{R}^d)$. An account of the theory of continuous one-parameter groups can be found in \cite{EN00}. Two notions of particular interest for us are captured by the following definition; the first of which is equivalent to the so-called Lyapunov stability of the one-parameter additive group $\mathbb{R}\ni t\to T_{e^{-t}}=\exp(-tE)$ \cite{EN00}.

\begin{definition}
Let $\{T_t\}\subseteq\Gl(\mathbb{R}^d)$ be a continuous one-parameter group.
\begin{enumerate}
\item We say that $\{T_t\}$ is contracting if
\begin{equation*}
\lim_{t\to 0}\|T_t\|=0.
\end{equation*}
\item We say that $\{T_t\}$ is non-expanding if 
\begin{equation*}
\sup_{0<t\leq 1}\|T_t\|<\infty.
\end{equation*}
\end{enumerate}
\end{definition}

Thanks to the continuity of $T:(0,\infty)\to\Gl(\mathbb{R}^d)$, every contracting group is non-expanding.
\begin{proposition}
Let $\{T_t\}$ be a continuous one-parameter group generated by $E\in\End(\mathbb{R}^d)$. If $\{T_t\}$ is contracting, then $\tr E>0$. If $\{T_t\}$ is non-expanding, then $\tr E\geq 0$.
\end{proposition}
\begin{proof}
By virtue of the continuity of the determinant map and the fact that $\det(t^E)=t^{\tr E}$, we have
\begin{equation*}
\lim_{t\to 0}t^{\tr E}=\lim_{t\to 0}\det(t^E)=\det(0)=0
\end{equation*}
provided $\{t^E\}$ is contracting. In this case, it follows that $\tr E>0$. If $\{t^E\}$ is expanding with $M=\sup_{0<t\leq 1}\|t^E\|$, we have
\begin{equation*}
\sup_{0<t\leq 1}t^{\tr E}=\sup_{0<t\leq 1}\det(t^E)\leq\sup_{\|A\|\leq M}\abs{\det(A)}<\infty
\end{equation*}
because $\{A\in\End(\mathbb{R}^d):\|A\|\leq M\}$ is a compact set. Thus $\tr E\geq 0$.
\end{proof}

Given a function $P:\mathbb{R}^d\to\mathbb{R}$ and $E\in\End(\mathbb{R}^d)$, we shall say that \textit{}$P$ is homogeneous with respect to $E$ if
\begin{equation*}
tP(\xi)=P(t^E\xi)
\end{equation*}
for all $t>0$ and $\xi\in\mathbb{R}^d$; in this case we say that \textit{$E$ is a member of the exponent set of $P$} and write $E\in\Exp(P)$. Central to the definition of positive homogeneous function given below is the following characterization taken from \cite{BR22}.

\begin{proposition}\label{prop:PosHomChar}
Let $P:\mathbb{R}^d\to\mathbb{R}$ be continuous, positive-definite (in the sense that $P\geq 0$ and $P(\xi)=0$ only when $\xi= 0$), and have $\Exp(P)\neq \varnothing$. Then the following are equivalent:
\begin{enumerate}
\item The so-called unital level set of $P$,
\begin{equation*}
S_P:=\{\xi\in\mathbb{R}^d:\abs{P(\xi)}=P(\xi)=1\},
\end{equation*}
is compact.
\item There is a positive number $M$ for which $P(\xi)>1$ whenever $\abs{\xi}>M$.
\item For each $E\in\Exp(P)$, $\{t^E\}$ is contracting.
\item There exists $E\in\Exp(P)$ for which $\{t^E\}$ is contracting.
\item We have
\begin{equation*}
\lim_{\abs{\xi}\to\infty}P(\xi)=\infty.
\end{equation*}
\end{enumerate}
\end{proposition}

\begin{definition}
Let $P:\mathbb{R}^d\to\mathbb{R}$ be continuous, positive definite and have $\Exp(P)\neq \varnothing$. If any one
(and hence all) of the equivalent conditions in Proposition \ref{prop:PosHomChar} are fulfilled, we say that $P$ is positive homogeneous. We will also say that $P$ is a positive homogeneous function on $\mathbb{R}^d$.
\end{definition}

The following proposition amasses some basic facts about positive homogeneous functions and its proof can be found in Section 2 of \cite{BR22}. 

\begin{proposition}\label{prop:HomOrder}
Let $P$ be a positive homogeneous function and denote by $\Sym(P)$ the set of $O\in\End(\mathbb{R}^d)$ for which $P(O\xi)=P(\xi)$ for all $\xi\in\mathbb{R}^d$. We have
\begin{enumerate}
\item $\Sym(P)$ is a compact subgroup of $\Gl(\mathbb{R}^d)$.
\item For each $E,\widetilde{E}\in\Exp(P)$, we have
\begin{equation*}
\tr E=\tr\widetilde{E}>0.
\end{equation*}
\end{enumerate}
\end{proposition}
In view of the preceding proposition, we define \textit{the homogeneous order of $P$} to be the unique positive number $\mu_P$ for which
\begin{equation*}
\mu_P=\tr E
\end{equation*}
for all $E\in\Exp(P)$.

\begin{example}
For any $\alpha>0$, the map $\xi\mapsto \abs{\xi}^\alpha$ is a positive homogeneous function on $\mathbb{R}^d$. Indeed, it is evident that it is continuous and positive-definite and its unital level set is precisely the unit sphere $S_{\abs{\cdot}^\alpha}=\mathbb{S}_d$ in $\mathbb{R}^d$. Further,
\begin{equation*}
\Exp(\abs{\cdot}^\alpha)=\frac{1}{\alpha}I+\mathfrak{o}_d
\end{equation*}
where $I$ is the identity map on $\mathbb{R}^d$ and $\mathfrak{o}_d$ is the Lie algebra of the orthogonal group $\OdR$ and is characterized by the set of skew symmetric matrices.
\end{example}
\begin{example}
Given a $d$-tuple of positive even integers $\mathbf{m}=(m_1,m_2,\dots,m_d)$, we consider a polynomial of the form
\begin{equation}\label{eq:SemiElliptic}
P(\xi)=\sum_{|\alpha:\mathbf{m}|=1}a_\alpha\xi^{\alpha}
\end{equation}
for $\xi=(\xi_1,\xi_2,\dots,\xi_d)\in\mathbb{R}^d$ where, for each multi-index $\alpha=(\alpha_1,\alpha_2,\dots,\alpha_d)\in\mathbb{N}^d$, $\xi^{\alpha}:=\xi_1^{\alpha_1}\xi_2^{\alpha_2}\cdots\xi_d^{\alpha_d}$ and
\begin{equation*}
|\alpha:\mathbf{m}|:=\sum_{k=1}^d\frac{\alpha_k}{m_k}.
\end{equation*}
A polynomial of the form \eqref{eq:SemiElliptic} is said to be \textit{semi-elliptic} provided $P(\xi)$ vanishes only at $\xi=0$. Appearing in L. H\"{o}rmander's treatise on linear partial differential operators \cite{Ho83}, semi-elliptic polynomials are the symbols of a class of hypoelliptic partial differential operators, called semi-elliptic operators. For a semi-elliptic polynomial $P(\xi)$ of the form \eqref{eq:SemiElliptic}, its corresponding semi-elliptic operator is the constant-coefficient linear partial differential operator given by
\begin{equation*}
P(D)=\sum_{|\alpha:\mathbf{m}|=1}a_\alpha D^\alpha
\end{equation*}
where we have written $D=(-i\partial_{x_1},-i\partial_{x_2},\dots,-i\partial_{x_d})$ and, for each multi-index $\alpha=(\alpha_1,\alpha_2,\dots,\alpha_d)\in\mathbb{N}^d$, $D^\alpha=(-i\partial_{x_1})^{\alpha_1}(-i\partial_{x_2})^{\alpha_2}\cdots(-i\partial_{x_d})^{\alpha_d}$. For a polynomial $P(\xi)$ of the form \eqref{eq:SemiElliptic}, observe that, for $E\in \End(\mathbb{R}^d)$ with standard matrix representation $\diag(1/m_1,1/m_2,\dots,1/m_d)$, 
\begin{equation*}
P(t^E\xi)=\sum_{|\alpha:\mathbf{m}|=1}a_\alpha (t^{1/m_1}\xi_1)^{\alpha_1}(t^{1/m_2}\xi_2)^{\alpha_2}\cdots(t^{1/m_d}\xi_d)^{\alpha_d}=\sum_{|\alpha:\mathbf{m}|=1}a_\alpha t^{|\alpha:\mathbf{m}|}\xi^\alpha=tP(\xi)
\end{equation*}
for all $t>0$ and $\xi=(\xi_1,\xi_2,\dots,\xi_d)\in\mathbb{R}^d$. Thus $E\in\Exp(P)$ and, because $\{t^E\}$ is clearly contracting, Proposition \ref{prop:PosHomChar} guarantees that $P$ is positive homogeneous whenever it is positive-definite. Thus, whenever a semi-elliptic polynomial $P$ of the form \eqref{eq:SemiElliptic} is positive-definite, it is positive homogeneous with homogeneous order
\begin{equation*}
\mu_P=\tr E=|\mathbf{1}:\mathbf{m}|=\frac{1}{m_1}+\frac{1}{m_2}+\cdots+\frac{1}{m_d}.
\end{equation*}
For two concrete examples, consider
\begin{equation*}
P_1(\xi)=\xi_1^4+\xi_2^6\hspace{1cm}\mbox{and}\hspace{1cm}P_2(\xi)=\xi_1^2+\xi_1\xi^2+\xi_2^4
\end{equation*}
defined for $\xi=(\xi_1,\xi_2)\in\mathbb{R}^2$. These are positive homogeneous semi-elliptic polynomials on $\mathbb{R}^2$ with homogeneous order $\mu_{P_1}=1/4+1/6=5/12$ and $\mu_{P_2}=1/2+1/4=3/4$, respectively.
\end{example}

Before we conclude our treatment of positive homogeneous functions and turn our attention to multivariate homogeneous functions, we present the following lemma which is used several times throughout the course of this paper. Its proof makes use of the generalized polar-coordinate integration formula developed and presented in Theorem 1.4 of \cite{BR22}.

\begin{lemma}\label{lem:ExpIntegrability}
Let $P$ be a positive homogeneous function on $\mathbb{R}^d$ with relatively compact unit ball $B_P=\{\xi\in\mathbb{R}^d:P(\xi)<1\}$ and homogeneous order $\mu_P>0$. Then, for each $\epsilon>0$, 
\begin{equation*}
\int_{\mathbb{R}^d}e^{-\epsilon P(\xi)}\,d\xi=m(B_P)\frac{\Gamma(\mu_P+1)}{\epsilon^{\mu_P}}
\end{equation*}
where $m(B_P)$ denoted the Lebesgue measure of $B_P$. In particular, for each $\epsilon>0$, $ \exp(-\epsilon P)\in L^1(\mathbb{R}^d)$.
\end{lemma}
\begin{proof}
By an appeal to Theorem 1.5 of \cite{BR22}, we obtain a Borel measure $\sigma=\sigma_P$ on $S=S_P$ for which $\sigma(S)=\mu_P\cdot m(B_P)$ and
\begin{equation*}
\int_{\mathbb{R}^d}e^{-\epsilon P(\xi)}\,d\xi=\int_{S}\int_0^\infty e^{-\epsilon P(t^E\eta)}t^{\mu_P-1}\,dt\,\sigma(d\eta)=\int_S\int_0^\infty e^{-\epsilon t}t^{\mu_P-1}\,dt\,\sigma(d\eta)
\end{equation*}
for any $E\in\Exp(P)$. Thus
\begin{equation*}
\int_{\mathbb{R}^d}e^{-\epsilon P(\xi)}\,d\xi=\sigma(S)\int_0^\infty e^{-\epsilon t}t^{\mu_P-1}\,dt=m(B_P)\frac{\mu_P}{\epsilon^{\mu_P}}\int_0^\infty e^{-s}s^{\mu_P-1}\,ds=m(B_P)\frac{\Gamma(\mu_P+1)}{\epsilon^{\mu_P}}
\end{equation*}
where we have used the Laplace-transform representation of Euler's Gamma function and the property that $\mu_P\cdot\Gamma(\mu_P)=\Gamma(\mu_P+1)$. 
\end{proof}

We now introduce a generalization of the notion of positive homogeneous function. Given non-negative integers $a$ and $b$ and a function $Q:\mathbb{R}^{b}\to\mathbb{R}^{a}$, we say that $Q$ is \textit{nondegenerate} if $Q(\zeta)\neq 0$ whenever $\zeta\neq 0$. Given a pair $(E,E')\in\End(\mathbb{R}^{a})\times\End(\mathbb{R}^{{b}})$, we say that \textit{$Q$ is homogeneous with respect to the pair $(E,E')$} provided that
\begin{equation*}
t^EQ(\zeta)=Q(t^{E'}\zeta)
\end{equation*}
for all $t>0$ and $\zeta\in\mathbb{R}^{{b}}$. Akin to Proposition \ref{prop:PosHomChar}, we have the following:
\begin{proposition}\label{prop:NMHChar}
Let $Q:\mathbb{R}^{{b}}\to\mathbb{R}^{a}$ be continuous, nondegenerate, and homogeneous with respect to some pair $(E,E')\in\End(\mathbb{R}^{a})\times\End(\mathbb{R}^{{b}})$ for which $\{t^E\}_{t>0}$ is contracting. Then, the following are equivalent:
\begin{enumerate}
\item\label{item:NMHChar_CompComp} For any compact set $K\subseteq\mathbb{R}^{a}$, $Q^{-1}(K)\subseteq\mathbb{R}^{{b}}$ is compact. 
\item\label{item:NMHChar_SComp} The set
\begin{equation*}
S_Q=\{\zeta\in\mathbb{R}^{{b}}:\abs{Q(\zeta)}=1\}
\end{equation*}
is compact.
\item\label{item:NMHChar_BiggerThanOne} There is a number $M>0$ for which $\abs{Q(\zeta)}>1$ for all $|\zeta|>M$.
\item\label{item:NMHChar_AnyContracting} For each pair $(E,E')\in\End(\mathbb{R}^{a})\times\End(\mathbb{R}^{{b}})$ for which $Q$ is homogeneous and $\{t^E\}$ is contracting, $\{t^{E'}\}$ must also be contracting.
\item\label{item:NMHChar_OneContracting} There exists a pair $(E,E')\in\End(\mathbb{R}^{a})\times\End(\mathbb{R}^{{b}})$ for which $Q$ is homogeneous and $\{t^E\}$ and $\{t^{E'}\}$ are both contracting.
\item\label{item:NMHChar_InfLimit} We have
\begin{equation*}
\lim_{\abs{\zeta}\to\infty}\abs{Q(\zeta)}=\infty.
\end{equation*}
\end{enumerate}
\end{proposition}
\begin{proof}
\begin{subproof}[(\ref{item:NMHChar_CompComp}$\,\Rightarrow$\ref{item:NMHChar_SComp})]
Since $S_Q=Q^{-1}(\mathbb{S}_a)$ where $\mathbb{S}_a$ is the unit sphere in $\mathbb{R}^{a}$, this is immediate.
\end{subproof}
\begin{subproof}[(\ref{item:NMHChar_SComp}$\,\Rightarrow$\ref{item:NMHChar_BiggerThanOne})]
Assuming that $S_Q$ is compact, let $M>0$ be such that $\abs{Q(\zeta)}\neq 1$ for all $\abs{\zeta}>M$. Denote by $\overline{\mathbb{B}_M}$ the closed ball in $\mathbb{R}^{{b}}$ with center $0$ and radius $M$ and by $\mathcal{O}_M=\mathbb{R}^{{b}}\setminus\overline{\mathbb{B}_M}$ its complement. Let us first treat the situation in which ${b}>1$. In this case, the fact that $\mathcal{O}_M$ is path connected and $Q(\zeta)$ is continuous ensures that there cannot be two elements $\zeta_1$ and $\zeta_2$ in $\mathcal{O}_M$ with $\abs{Q(\zeta_1)}<1<\abs{Q(\zeta_2)}$ for otherwise the intermediate value theorem would imply that $\abs{Q(\zeta_3)}=1$ for some $\zeta_3\in\mathcal{O}_M$, an impossibility. Thus, to prove the statement, we simply need to rule out the possibility that $\abs{Q(\zeta)}<1$ for all $\zeta\in\mathcal{O}_M$. Let us assume, to reach a contradiction, that $\abs{Q(\zeta)}<1$ for all $\zeta\in\mathcal{O}_M$. Let $(E,E')\in \End(\mathbb{R}^{a})\times\End(\mathbb{R}^{{b}})$ be a pair for which $Q$ is homogeneous and $\{t^E\}$ is contracting and let $\zeta_0$ be a non-zero element in $\mathbb{R}^{{b}}$. The nondegenerateness of $Q$ and the fact that $\{t^E\}$ is contracting guarantees that
\begin{equation*}
\infty=\lim_{t\to\infty}\abs{t^EQ(\zeta_0)}=\lim_{t\to\infty}\abs{Q(t^{E'}\zeta_0)}.
\end{equation*}
In view of our hypothesis it follows that, for all sufficiently large $t$, $t^{E'}\zeta_0\in \overline{\mathbb{B}_M}$. Of course, this implies that $\zeta\mapsto \abs{Q(\zeta)}$ is unbounded on the compact set $\overline{\mathbb{B}_M}$ and this is impossible for we know that $Q$ is continuous. 

In the case that ${b}=1$, we first argue that $Q(\zeta)>1$ for all $\zeta>M>0$. Of course, since $(M,\infty)$ is connected, an argument analogous to that above for ${b}>1$ guarantees that it suffices to rule out the case that $Q(\zeta)<1$ for all $\zeta>M$. We therefore assume, to reach a contradiction, that $Q(\zeta)<1$ for all $\zeta>M$ and select a pair $(E,E')\in\End(\mathbb{R}^{a})\times\End(\mathbb{R}^{{b}})$ for which $Q$ is homogeneous and $\{t^E\}$ contracting. Due to the simplicity of $\End(\mathbb{R}^{{b}})=\End(\mathbb{R})$, $E'=\alpha' I$ for some $\alpha'\in\mathbb{R}$ and so $t^{E'}\zeta=t^{\alpha'}\zeta$ for all $t>0$ and $\zeta\in\mathbb{R}=\mathbb{R}^{{b}}$. Using the fact that $\{t^E\}$ is contracting, we have
\begin{equation*}
\infty=\lim_{t\to\infty}\abs{t^{E}Q(1)}=\lim_{t\to\infty}\abs{Q(t^{\alpha'})}
\end{equation*}
and so, in view of our supposition, it follows that $0<t^{\alpha'}\leq M$ for all sufficiently large $t$ (which, at the same time guarantees that $\alpha'< 0$). This, however, contradicts that fact that $Q$ is continuous at $0$. Hence, $Q(\zeta)>1$ for all $\zeta>M$. A similar argument shows that $Q(\zeta)>1$ for all $\zeta<-M$. Thus $Q(\zeta)>1$ for all $\abs{\zeta}>M$, as was asserted.
\end{subproof}

\begin{subproof}[(\ref{item:NMHChar_BiggerThanOne}$\,\Rightarrow$\ref{item:NMHChar_AnyContracting})]
We prove the contrapositive statement. Let us assume that there is a pair $(E,E')\in\End(\mathbb{R}^{a})\times\End(\mathbb{R}^{{b}})$ for which $Q$ is homogeneous and $\{t^E\}$ is contracting, but $\{t^{E'}\}$ is not contracting. It follows that, for some $\zeta\in\mathbb{R}^{{b}}$ and $\epsilon>0$, there is a sequence $t_k\to 0$ for which
\begin{equation*}
\abs{t_k^{E'}\zeta}\geq \epsilon
\end{equation*}
for all $k$. In the case that $\{t_k^{E'}\zeta\}$ is bounded, we assume without loss of generality (by passing to a subsequence, if needed) that $\lim_{k\to\infty}t_k^{E'}\zeta=\zeta_0$ where $\abs{\zeta_0}\geq \epsilon$. Consequently,
\begin{equation*}
Q(\zeta_0)=\lim_{k\to\infty}Q(t_k^{E'}\zeta)=\lim_{k\to\infty}t_k^{E}Q(\zeta)=0
\end{equation*}
since $t_k\to 0$ and $\{t^{E}\}$ is contracting. Since $\zeta_0\neq 0$, this is impossible for we know that $Q$ is nondegenerate. Thus, we must conclude that $\{t_k^{E'}\zeta\}$ is unbounded. Without loss of generality, we may assume (by passing to a subsequence, if necessary) that $\zeta_k=t_k^{E'}\zeta\to \infty$ as $k\to\infty$. Consequently, there is a sequence $\zeta_k\to \infty$ for which
\begin{equation*}
\lim_{k\to\infty}Q(\zeta_k)=\lim_{k\to\infty}t_k^{E}Q(\zeta)=0.
\end{equation*}
This shows that Item \ref{item:NMHChar_BiggerThanOne} cannot hold, as was asserted.
\end{subproof}

\begin{subproof}[(\ref{item:NMHChar_AnyContracting}$\,\Rightarrow$\ref{item:NMHChar_OneContracting})]
In view of the hypotheses, this is immediate.
\end{subproof}
\begin{subproof}[(\ref{item:NMHChar_OneContracting}$\,\Rightarrow$\ref{item:NMHChar_InfLimit})]
We fix a pair $(E,E')\in\End(\mathbb{R}^{a})\times\End(\mathbb{R}^{{b}})$ for which $Q$ is homogeneous and both $\{t^E\}$ and $\{t^{E'}\}$ are contracting. Let $\{\zeta_k\}\subseteq\mathbb{R}^{{b}}$ be a sequence with $\zeta_k\to \infty$ as $k\to\infty$. Since $\{t^{E'}\}$ is contracting and in view of Proposition A.5 of \cite{BR22}, we can write $\zeta_k=t_k^{E'}\eta_k$ where $\abs{\eta_k}=1$ for each $k$ and $t_k\to\infty$. We claim that
\begin{equation*}
\lim_{k\to\infty}\abs{t_k^EQ(\eta_k)}=\infty.
\end{equation*}
To see this, we assume, to reach a contradiction, that a subsequence has the property that $\abs{t_{k_j}^EQ(\eta_{k_j})}\leq M$ for some $M$. In this case, we see that
\begin{equation*}
\abs{Q(\eta_{k_j})}=\abs{(1/t_{k_j})^Et_{k_j}^EQ(\eta_{k_j)}}\leq M\|(1/t_{k_j})^E\|
\end{equation*}
for all $j$ and, since $1/t_{k_j}\to 0$ as $j\to\infty$, the fact that $\{t^E\}$ is contracting implies that 
\begin{equation*}
0=\inf_{\abs{\zeta}=1}\abs{Q(\zeta)}.
\end{equation*}
This is, however, impossible because $Q$ is continuous and nonvanishing on the compact unit sphere in $\mathbb{R}^{{b}}$. We have therefore substantiated our claim and so it follows that
\begin{equation*}
\lim_{k\to\infty}\abs{Q(\zeta_k)}=\lim_{k\to\infty}\abs{Q(t_k^{E'}\eta_k)}=\lim_{k\to\infty}\abs{t_k^EQ(\eta_k)}=\infty,
\end{equation*}
as desired. 
\end{subproof}
\begin{subproof}[(\ref{item:NMHChar_InfLimit}$\,\Rightarrow$\ref{item:NMHChar_CompComp})]
Let $K\subseteq\mathbb{R}^{a}$ be compact. The fact that $Q$ is continuous ensures that $Q^{-1}(K)$ is necessarily closed. If $Q^{-1}(K)$ were unbounded, we could find a sequence $\zeta_k \in Q^{-1}(K)$ for which $\abs{\zeta_k}\to \infty$ yet $Q(\zeta_k)\in K$ for all $k$. Of course, this is impossible in light of our assumption. Hence $Q^{-1}(K)$ must be bounded and therefore compact in view of the Heine-Borel theorem.
\end{subproof}
\end{proof}

\begin{definition}
Let $Q:\mathbb{R}^{{b}}\to\mathbb{R}^{a}$ be continuous, nondegenerate and homogeneous with respect to a pair $(E,E')\in\End(\mathbb{R}^{a})\times\End(\mathbb{R}^{{b}})$ for which $\{t^E\}$ is contracting. If any (hence, every) of the equivalent  conditions listed in Proposition \ref{prop:NMHChar} are satisfied, we say that $Q$ is nondegenerate multivariate homogeneous.
\end{definition}

\begin{example}
For a given positive integer $\alpha$, consider $Q:\mathbb{R}\to\mathbb{R}$ defined by $Q(\zeta)=\zeta^\alpha$. It is clear that $Q(\zeta)$ is continuous and nondegenerate. We observe further that
\begin{equation*}
Q(t^I\zeta)=Q(t\zeta)=t^\alpha\zeta^{\alpha}=t^{\alpha I}\zeta^{\alpha}
\end{equation*}
for all $t>0$ and $\zeta\in\mathbb{R}$ where $I$ is the identity transformation on $\mathbb{R}$. Thus $Q$ is homogeneous with respect to $(\alpha I, I)$ and, further, it is clear that $\{t^{\alpha I}\}=\{t^\alpha\}$ and $\{t^I\}=\{t\}$ are both contracting. Thus, by virtue of Proposition \ref{prop:NMHChar}, we conclude that $Q$ is nondegenerate multivariate homogeneous. 
\end{example}

\begin{example}\label{ex:CanonicalNMH}
Given positive integers $a$ and $b$, let $\sigma_1,\sigma_2,\dots,\sigma_a$ be a collection of positive integers for which, as sets, $\{1,2,\dots,{b}\}=\{\sigma_1,\sigma_2,\dots,\sigma_a\}$ and, given positive integers $\alpha_1,\alpha_2,\dots,\alpha_a$, define $Q:\mathbb{R}^{{b}}\to\mathbb{R}^a$ by
\begin{equation}\label{eq:CanonicalNMH}
Q(\zeta)=Q(\zeta_1,\zeta_2,\dots,\zeta_{{b}})=(\zeta_{\sigma_1}^{\alpha_1},\zeta_{\sigma_2}^{\alpha_2},\dots,\zeta_{\sigma_a}^{\alpha_a})
\end{equation}
for $\zeta=(\zeta_1,\zeta_2,\dots,\zeta_{{b}})\in\mathbb{R}^{{b}}$. We claim that $Q$ is nondegenerate multivariate homogeneous. To see this, we observe first that $Q$ is clearly continuous (in fact, $Q\in C^\infty$) and $Q(\zeta)=0$ if and only if $\zeta_{\sigma_j}=0$ for all $j=1,2,\dots,d$. In view of the condition that $\{\sigma_1,\sigma_2,\dots,\sigma_d\}=\{1,2,\dots,{b}\}$, we conclude that $Q(\zeta)=0$ if and only if $\zeta=(\zeta_1,\zeta_2,\dots,\zeta_{{b}})=0$ and so $Q$ is nondegenerate. Observe that
\begin{equation*}
Q(t^I\zeta)=Q(t\zeta)=(t^{\alpha_1}\zeta_{\sigma_1}^{\alpha_1},t^{\alpha_2}\zeta_{\sigma_2}^{\alpha_2},\dots,t^{\alpha_a}\zeta_{\sigma_a}^{\alpha_a})=t^EQ(\zeta)
\end{equation*}
where $I$ is the identity on $\mathbb{R}^{{b}}$ and $E\in\End(\mathbb{R}^{a})$ has standard representation $\diag(\alpha_1,\alpha_2,\dots,\alpha_a)$. Consequently, $Q$ is homogeneous with respect to $(E,I)$ and, since $\{t^E\}$ and $\{t^I\}$ are contracting, we conclude that $Q$ is nondegenerate multivariate homogeneous. For a concrete example, consider $Q:\mathbb{R}^2\to\mathbb{R}^3$ defined by
\begin{equation*}
Q(\zeta_1,\zeta_2)=(\zeta_2,\zeta_1^4,\zeta_2^3)
\end{equation*}
for $\zeta=(\zeta_1,\zeta_2)\in\mathbb{R}^2$. This is a nondegenerate multivariate homogeneous function of the above form with $\sigma_1=\sigma_3=2$, $\sigma_2=1$, $\alpha_1=1$, $\alpha_2=4$, and $\alpha_3=3$. 
\end{example}

We remark that, for any $Q$ of the form \eqref{eq:CanonicalNMH}, we have $b\leq a$. The following example generalizes that above and allows for $b>a$. In addition, the example shows that all positive homogeneous functions are nondegenerate multivariate homogeneous.

\begin{example}\label{ex:AGeneralQ}
For positive integers $a$ and $b$, let $q_1(\zeta),q_2(\zeta),\dots,q_a(\zeta)$ be a collection of continuous real-valued functions on $\mathbb{R}^{{b}}$ which satisfy the following conditions:
\begin{enumerate}
\item\label{cond:AGeneralQ1} If $\zeta\neq 0$, then $q_k(\zeta)\neq 0$ for at least one $k=1,2,\dots,a$.
\item\label{cond:AGeneralQ2} There exists $E'\in\End(\mathbb{R}^{{b}})$ for which $\{t^{E'}\}$ is contracting and $E'\in \Exp(q_k)$ for all $k=1,2,\dots,a$. 
\end{enumerate}
Given such a collection, let $\alpha_1,\alpha_2,\dots,\alpha_a$ be positive integers and define
$Q:\mathbb{R}^{{b}}\to\mathbb{R}^{a}$ by
\begin{equation}\label{eq:AGeneralQ}
Q(\zeta)=(q_1(\zeta)^{\alpha_1},q_2(\zeta)^{\alpha_2},\dots,q_a(\zeta)^{\alpha_a})
\end{equation}
for $\zeta\in\mathbb{R}^{{b}}$. We claim that $Q$ is nondegenerate multivariate homogeneous. Indeed, $Q$ is continuous and nondegenerate in view of Condition \ref{cond:AGeneralQ1}. Upon taking $E'\in\End(\mathbb{R}^{{b}})$ satisfying Condition \ref{cond:AGeneralQ2}, we observe that $Q$ is homogeneous with respect to the pair $(E,E')$ where $E\in\End(\mathbb{R}^d)$ has standard representation $\diag(\alpha_1,\alpha_2,\dots,\alpha_a)$. Since $\{t^E\}$ and $\{t^{E'}\}$ are contracting, we conclude that $Q$ is nondegenerate multivariate homogeneous in view of Proposition \ref{prop:NMHChar}. 

In the case that $a=1$, $\alpha_1=1$, and $q_1(\zeta)=P(\zeta)$ is a positive homogeneous function, the above conditions are automatically satisfied for any $E'\in\Exp(P)$. From this we conclude that positive homogeneous functions are nondegenerate multivariate homogeneous. We remark that, for any $E'\in\Exp(P)$, $Q(\zeta)=q_1(\zeta)=P(\zeta)$ is homogeneous with respect to the pair $(I,E')$ where $I$ is the identity on $\mathbb{R}$.

For a concrete example of a multivariate homogeneous function of the form \eqref{eq:AGeneralQ} (and for which $1<a<{b}$), consider $Q:\mathbb{R}^3\to\mathbb{R}^2$ defined by
\begin{equation*}
Q(\zeta)=(\zeta_1^2+\zeta_2^4,(\zeta_1+\zeta_3^3)^5)
\end{equation*}
for $\zeta=(\zeta_1,\zeta_2,\zeta_3)\in\mathbb{R}^3$. This can be written equivalently as
\begin{equation*}
Q(\zeta)=\left(\left(\zeta_1^2+\zeta_2^4)^{1/2}\right)^2,\left(\zeta_1+\zeta_3^3\right)^5\right)
\end{equation*}
for $\zeta=(\zeta_1,\zeta_2,\zeta_3)\in\mathbb{R}^3$. This is clearly of the form \eqref{eq:AGeneralQ} with $q_1(\zeta)=(\zeta_1^2+\zeta_2^4)^{1/2}$, $q_2(\zeta)=\zeta_1+\zeta_3^3$, $\alpha_1=2$ and $\alpha_2=5$. In this case, it is easy to see that $q_1(\zeta)$ and $q_2(\zeta)$ satisfy Condition \ref{cond:AGeneralQ1}. Further, observe that, for $E'\in\End(\mathbb{R}^3)$ with standard representation $\diag(1,1/2,1/3)$, we have $E'\in \Exp(q_1)\cap\Exp(q_2)$. Since $\{t^{E'}\}$ is clearly contracting, we may conclude that $Q$ is nondegenerate multivariate homogeneous and homogeneous with respect to $(E,E')$ where $E'$ is that above and $E\in\End(\mathbb{R}^2)$ has standard representation $\diag(2,5)$. Of course, we can confirm the homogeneity of $Q$ with respect to the pair $(E,E')$ directly: For $t>0$ and $\zeta=(\zeta_1,\zeta_2,\zeta_3)$, 
\begin{eqnarray*}
Q(t^{E'}\zeta)&=&Q(t^1\zeta_1,t^{1/2}\zeta_2,t^{1/3}\zeta_3)\\
&=&\left((t\zeta_1)^2+(t^{1/2}\zeta_2)^4,\left((t\zeta_1)+(t^{1/3}\zeta_3)^3\right)^5\right)\\
&=&\left(t^2\left(\zeta_1^2+\zeta_2^4\right),t^5\left(\zeta_1+\zeta_3^3\right)^5\right)\\
&=&t^{E}Q(\zeta).
\end{eqnarray*}
\end{example}

\begin{proposition}\label{prop:CompIsPosHom}
Suppose that $P$ is a positive homogeneous function on $\mathbb{R}^{a}$, $Q:\mathbb{R}^{{b}}\to\mathbb{R}^{a}$ is a nondegenerate multivariate homogeneous function and, for some $E\in\Exp(P)$ and $E'\in\End(\mathbb{R}^{{b}})$, $Q$ is homogeneous with respect to the pair $(E,E')$. Then $P\circ Q$ and $P\circ (-Q)$ are positive homogeneous functions on $\mathbb{R}^{{b}}$ whose exponent sets contain $E'$. In particular, $\mu_{P\circ Q}=\mu_{P\circ(-Q)}=\tr E'$.
\end{proposition}
\begin{proof}
It is clear that $Q$ is nondegenerate multivariate homogeneous if and only if $-Q$ is nondegenerate multivariate homogeneous and $Q$ is homogeneous with respect to a pair $(E,E')$ if and only if $-Q$ is. Thus, to prove the proposition, it suffices to prove that $P\circ Q$ is positive homogeneous with $E'\in\Exp(P\circ Q)$. Because $P$ is continuous and positive-definite and $Q$ is continuous and nondegenerate, it evident that  $P\circ Q$ is continuous and positive-definite. Given the pair $(E,E')$ as in the statement of the proposition, we observe that
\begin{equation*}
(P\circ Q)(t^{E'}\zeta)=P(t^EQ(\zeta))=tP(Q(\zeta))=t(P\circ Q)(\zeta)
\end{equation*}
for all $t>0$ and $\zeta\in\mathbb{R}^{{b}}$. Thus, $E'\in\Exp(P\circ Q)$. Finally, since $\{t^E\}$ is contracting thanks to Proposition \ref{prop:PosHomChar}, $\{t^{E'}\}$ must also be contracting in view or Proposition \ref{prop:NMHChar} and so we conclude that $P\circ Q$ is positive homogeneous in light of Proposition \ref{prop:PosHomChar}.
\end{proof}

\begin{example}
Consider the nondegenerate multivariate homogeneous function $Q:\mathbb{R}^3\to\mathbb{R}^2$ from the previous example given by
\begin{equation*}
Q(\zeta)=(\zeta_1^2+\zeta_2^4,(\zeta_1+\zeta_3^3)^5)
\end{equation*}
for $\zeta=(\zeta_1,\zeta_2,\zeta_3)\in\mathbb{R}^3$. Also, consider the positive homogeneous function $P:\mathbb{R}^2\to\mathbb{R}$ defined by
\begin{equation*}
P(\xi)=P(\xi_1,\xi_2)=\xi_1^{5}+\xi_2^{2}
\end{equation*}
for $\xi=(\xi_1,\xi_2)\in\mathbb{R}^2$. We see that $P$ is homogeneous with respect to $E\in\End(\mathbb{R}^2)$ with standard representation $\diag(1/5,1/2)$. Consider also $E'\in\End(\mathbb{R}^3)$ with standard representation $\diag(1/10, 1/20, 1/30)$ and observe that
\begin{eqnarray*}
Q(t^{E'}\zeta)&=&Q(t^{1/10}\zeta_1,t^{1/20}\zeta_2,t^{1/30}\zeta_3)\\
&=&\left(t^{1/5}(\zeta_1^2+\zeta_2^4),t^{1/2}(\zeta_1+\zeta_3^3)^5\right)\\
&=&t^EQ(\zeta)
\end{eqnarray*}
for $t>0$ and $\zeta=(\zeta_1,\zeta_2,\zeta_3)\in\mathbb{R}^3$. Thus, by the preceding proposition, we have that $P\circ Q$ is positive homogeneous with $\mu_{P\circ Q}=\tr E'=11/60$. Of course, this can be verified directly by simplification of $P\circ Q$. We have
\begin{equation*}
(P\circ Q)(\zeta)=(\zeta_1^2+\zeta_2^4)^5+(\zeta_1+\zeta_3^3)^{10}
\end{equation*}
which is clearly positive homogeneous with $E'\in\Exp(P\circ Q)$.
\end{example}

\section{On-Diagonal Asymptotics}\label{sec:OnDiagonalAsymptotics}

\begin{theorem}\label{thm:OnDiagonal}
Given positive integers ${a}$ and ${b}$, let $P_1$ and $P_2$ be positive homogeneous functions on $\mathbb{R}^{{a}}$ with homogeneous orders $\mu_{P_1}$ and $\mu_{P_2}$, respectively, and let $Q:\mathbb{R}^{{b}}\to\mathbb{R}^{{a}}$ be a $C^1$ function which is nondegenerate multivariate homogeneous. Set $d={a}+{b}$ and define $P:\mathbb{R}^{d}\to\mathbb{R}$ by
\begin{equation*}
P(\xi)=P(\eta,\zeta)=P_1(\eta+Q(\zeta))+P_2(\eta)
\end{equation*}
for $\xi=(\eta,\zeta)\in\mathbb{R}^{{a}}\times\mathbb{R}^{{b}}=\mathbb{R}^d$. Suppose that there exist $E_1\in \Exp(P_1)$, $E_2\in\Exp(P_2)$, and $F_1,F_2\in\End(\mathbb{R}^{{b}})$ for which the following conditions hold:

\begin{enumerate}
\item\label{cond:HomPair} For $k=1,2$,  $Q$ is homogeneous with respect to the pair $(E_k,F_k)$.
\item\label{cond:ComutingAndNonExpanding} We have $[E_1,E_2]=E_1E_2-E_2E_1=0$ and $\{t^{E_1-E_2}\}$ is non-expanding.
\end{enumerate}
In particular, $P_1\circ Q$ and $P_2\circ (-Q)$ are positive homogeneous on $\mathbb{R}^{{b}}$ (by virtue of Proposition \ref{prop:CompIsPosHom}) with homogeneous orders $\mu_{P_1\circ Q}$ and $\mu_{P_2\circ (-Q)}$, respectively. Then the heat kernel
\begin{equation*}
H_P^t(x)=\frac{1}{(2\pi)^{d}}\int_{\mathbb{R}^{d}}e^{-tP(\xi)}e^{-ix\cdot\xi}\,d\xi
\end{equation*}
exists for each $t>0$ and $x\in\mathbb{R}^{d}$. Upon setting $\varphi(t):=H_P^t(0)$ for $t>0$, we have the following on-diagonal asymptotics:
\begin{equation*}
\varphi(t)\asymp \begin{cases}
t^{-\mu_0} & 0<t\leq 1\\
t^{-\mu_\infty} & 1\leq t<\infty
\end{cases}
\end{equation*}
for $t>0$ where $\mu_0=\mu_{P_2}+\mu_{P_1\circ Q}$ and $\mu_\infty=\mu_{P_1}+\mu_{P_2\circ(-Q)}$.
\end{theorem}
\begin{remark}\label{rmk:OrderTrace}
In view of Propositions \ref{prop:HomOrder} and \ref{prop:CompIsPosHom}, $\mu_{P_1}=\tr E_1$, $\mu_{P_2}=\tr E_2$, $\mu_{P_1\circ Q}=\tr F_1$, and $\mu_{P_2\circ (-Q)}=\tr F_2$ where $E_1,E_2,F_1,F_2$ are those given in the hypotheses of Theorem \ref{thm:OnDiagonal} (or any which satisfy Conditions \ref{cond:HomPair} and \ref{cond:ComutingAndNonExpanding}). In these terms, the asymptotics for $\varphi(t)$ can be equivalently written
\begin{equation*}
\varphi(t)\asymp \begin{cases}
t^{-\left(\tr E_2+\tr F_1\right)} & 0<t\leq 1\\
t^{-\left(\tr E_1+\tr F_2\right)} & 1\leq t<\infty
\end{cases}
\end{equation*}
for $t>0$.
\end{remark}
\begin{remark}
Given our hypothesis that $Q$ is nondegenerate multivariate homogeneous, when we ask that $Q$ be homogeneous with respect to the pair $(E_1,F_1)$ for some $E_1\in\Exp(P_1)$ and $F_1\in\End(\mathbb{R}^{{b}})$, we are ensuring that $\{t^{F_1}\}$ is contracting by virtue of Proposition \ref{prop:NMHChar} since it is known that $\{t^{E_1}\}$ is contracting thanks to Proposition \ref{prop:PosHomChar}. For the same reason, $\{t^{F_2}\}$ must also be contracting. Of course, these observations also follow from Proposition \ref{prop:CompIsPosHom} since $F_1\in\Exp(P_1\circ Q)$ and $F_2\in\Exp(P_2\circ(-Q))$.

In this direction, if Condition \ref{cond:HomPair} were adjusted to include the hypothesis that, for $F_1$ (or $F_2$), $\{t^{F_1}\}$ (or $\{t^{F_2}\}$) is contracting, then the initial hypothesis could be weakened to ask only that $Q:\mathbb{R}^{{a}}\to\mathbb{R}^d$ be nondegenerate and $C^1$. In this case, the modified Condition \ref{cond:HomPair} would ensure that $Q$ were nondegenerate multivariate homogeneous by virtue of Item \ref{item:NMHChar_OneContracting} of Proposition \ref{prop:NMHChar}.
\end{remark}

\begin{remark}
If Condition \ref{cond:ComutingAndNonExpanding} were replaced by the stronger condition that $[E_1,E_2]=E_1E_2-E_2E_1=0$ and $\{t^{E_1-E_2}\}$ is contracting, then one finds that the small and large-time asymptotics $t^{-\mu_0}$ and $t^{-\mu_\infty}$ are ``true asymptotics'' in the sense that the limits $\lim_{t\to 0}t^{\mu_0}\varphi(t)$ and $\lim_{t\to\infty}t^{\mu_\infty}\varphi(t)$ both exist and are positive numbers. This result is presented in Theorem \ref{thm:TrueAsymptotic}. The theorem, in fact, gives us the precise value of these limits in terms of $P_1$, $P_2$, and $Q$. 
\end{remark}

\begin{example}\label{ex:OnDiagonalMotivatingExample}
Let $q$ and $l$ be even positive integers with $q\leq l$ and let $p\in\mathbb{N}_+$. We consider $P:\mathbb{R}^2\to\mathbb{R}$ defined by
\begin{equation*}
P(\eta,\zeta)=(\eta+\zeta^p)^q+\eta^l
\end{equation*}
for $(\eta,\zeta)\in\mathbb{R}^2$ and the corresponding heat kernel $H_P$ defined by
\begin{equation*}
H_P^t(x,y)=\frac{1}{(2\pi)^2}\int_{\mathbb{R}^2}e^{-tP(\eta,\zeta)}e^{-i(\eta,\zeta)\cdot(x,y)}\,d\zeta\,d\eta
\end{equation*}
for $t>0$ and $(x,y)\in\mathbb{R}^2$. In this case, we can write
\begin{equation*}
P(\eta,\zeta)=P_1(\eta+Q(\zeta))+P_2(\eta)
\end{equation*}
where $P_1(\eta)=\eta^q$ and $P_2(\eta)=\eta^l$ are positive homogeneous on $\mathbb{R}=\mathbb{R}^1$ and $Q(\zeta)=\zeta^p$ is evidently a nondegenerate multivariate homogeneous $C^1$ function from $\mathbb{R}$ to itself. For $P_1$, we have $E_1=I/q\in\Exp(P_1)$ and $\mu_{P_1}=1/q$ and, for $P_2$, we have $E_2=I/l$ and $\mu_{P_2}=1/l$. Further, we observe that
\begin{equation*}
t^{E_1}Q(\zeta)=t^{1/q}\zeta^{p}=(t^{1/qp}\zeta)^p=Q(t^{1/qp}\zeta)
\end{equation*}
for $t>0$ and $\zeta\in\mathbb{R}$ and from this we see that $Q$ is homogeneous with respect to the pair $(E_1,F_1)$ where $F_1=I/qp$. Similarly,$Q$ is also homogeneous with respect to the pair $(E_2,F_2)$ where $F_2=1/lp$. Finally, we observe that  $E_1$ and $E_2$ commute and, since $q\leq l$, $t^{E_1-E_2}=t^{(1/q-1/l)I}$ is non-expanding. Thus, an application of the theorem is valid and we conclude that
\begin{equation*}
\varphi(t)=H_P^t(0)\asymp \begin{cases}
t^{-(1/l+1/qp)} & t\leq 1\\
t^{-(1/q+1/lp)} & t\geq 1
\end{cases} 
\end{equation*}
because $\mu_0=\mu_{P_2}+\mu_{P_1\circ Q}=\tr E_2+\tr F_1=1/l+1/qp$ and $\mu_\infty=\mu_{P_1}+\mu_{P_2\circ (-Q)}=\tr E_1+\tr F_2=1/q+1/lp$. 

We recognize that the motivating example in the introduction is of the above form where $p=q=2$ and $l=4$, i.e.,
\begin{equation*}
    P(\eta,\zeta)=(\eta+\zeta^2)^2+\eta^4
\end{equation*}
for $(\eta,\zeta)\in\mathbb{R}^2$. From this, we conclude that
\begin{equation*}
    \varphi(t)\asymp \begin{cases}
    t^{-1/2} & 0<t\leq 1\\
    t^{-5/8} & 0<t\leq 1
    \end{cases}
\end{equation*}
for $t>0$, as was asserted. Following Theorem \ref{thm:TrueAsymptotic}, we revisit this example to obtain precise values of $\lim_{t\to 0}t^{1/2}\varphi(t)$ and $\lim_{t\to\infty}t^{5/8}\varphi(t)$. Necessarily, the $t\to 0$ limit is precisely that obtained by the scaling argument presented in the introduction.
\end{example}

\begin{example}\label{ex:CanonicalNMHFollowUp}
In view of Example \ref{ex:CanonicalNMH}, let $Q:\mathbb{R}^{a}\to\mathbb{R}^a$ be a nondegenerate multivariate homogeneous function of the form
\begin{equation*}
Q(\zeta)=Q(\zeta_1,\zeta_2,\dots,\zeta_{a})=(\zeta_{\sigma(1)}^{\alpha_1},\zeta_{\sigma(2)}^{\alpha_2},\dots,\zeta_{\sigma(a)}^{\alpha_a})
\end{equation*}
where $\alpha_1,\alpha_2,\dots,\alpha_a\in\mathbb{N}_+$ and $\sigma$ is a permutation of $\{1,2,\dots,a\}$; in particular, $\{\sigma(1),\sigma(2),\dots,\sigma(a)\}=\{1,2,\dots,a\}.$ We remark that $Q$ is clearly smooth. Now, let $P_1$ and $P_2$ be positive homogeneous functions on $\mathbb{R}^a$ and suppose that, for $k=1,2$, $\Exp(P_k)$ contains $E_k\in\End(\mathbb{R}^a)$ having standard matrix representation $\diag(\lambda_{k,1},\lambda_{k,2},\dots,\lambda_{k,a})$ where $\lambda_{k,j}>0$ for $j=1,2,\dots, a$. If, for $k=1,2$, we consider $F_k$ with standard matrix representation
\begin{equation*}
\diag\left(\frac{\lambda_{k,\sigma^{-1}(1)}}{\alpha_{\sigma^{-1}(1)}},\frac{\lambda_{k,\sigma^{-1}(2)}}{\alpha_{\sigma^{-1}(2)}},\dots,\frac{\lambda_{k,\sigma^{-1}(a)}}{\alpha_{\sigma^{-1}(a)}}\right),
\end{equation*}
we see that
\begin{equation*}
t^{E_k}Q(\zeta)=\left(t^{\lambda_{k,1}}\zeta_{\sigma(1)}^{\alpha_1},t^{\lambda_{k,2}}\zeta_{\sigma(2)}^{\alpha_2},\dots,t^{\lambda_{k,2}}\zeta_{\sigma(a)}^{\alpha_a}\right)=Q(t^{F_k}\zeta)
\end{equation*}
for all $t>0$ and $\zeta\in\mathbb{R}^a$. Correspondingly, $Q$ is homogeneous with respect to the pair $(E_k,F_k)$ for $k=1,2$. We have the following result.
\begin{proposition}
If $\lambda_{1,j}\geq\lambda_{2,j}$ for $j=1,2,\dots, a$, then for $P:\mathbb{R}^{2a}\to\mathbb{R}$ defined by
\begin{equation*}
P(\eta,\zeta)=P_1(\eta+Q(\zeta))+P_2(\eta)
\end{equation*}
for $(\eta,\zeta)\in\mathbb{R}^{2a}$, we have
\begin{equation*}
\varphi(t)\asymp\begin{cases}
t^{-\mu_{0}} & t\leq 1\\
t^{-\mu_{\infty}} & t\geq 1
\end{cases}
\end{equation*}
for $t>0$ where
\begin{equation*}
\mu_0=\sum_{j=1}^a\left(\lambda_{2,j}+\frac{\lambda_{1,j}}{\alpha_j}\right)
\end{equation*}
and
\begin{equation*}
\mu_\infty=\sum_{j=1}^a\left(\lambda_{1,j}+\frac{\lambda_{2,j}}{\alpha_j}\right).
\end{equation*}
\end{proposition}
\begin{exproof}
The hypothesis guarantees that $t^{E_1-E_2}$ is non-expanding and so, in view of Remark \ref{rmk:OrderTrace}, we only need to verify that the exponents $\mu_0$ and $\mu_\infty$ are as stated. Of course, since $\sigma$ is a permutation of $\{1,2,\dots,a\}$, 
\begin{equation*}
\mu_0=\tr E_2+\tr F_1=\sum_{j=1}^a \lambda_{2,j}+\sum_{j'=1}^a\frac{\lambda_{1,\sigma^{-1}(j')}}{\alpha_{\sigma^{-1}(j')}}=\sum_{j=1}^a\lambda_{2,j}+\sum_{j=1}^a\frac{\lambda_{1,j}}{\alpha_j}.
\end{equation*}
The computation for $\mu_\infty$ is done analogously.
\end{exproof}
In this example, we have assumed that $a=b$. We remark that the results found above can be extended to arbitrary dimensions $a$ and $b$ by consider $Q$'s of the form found in Examples \ref{ex:CanonicalNMH} and \ref{ex:AGeneralQ}. We leave these details to the reader.
\end{example}

Our proof of Theorem \ref{thm:OnDiagonal} makes use of the following lemma, whose proof can be found in Appendix \ref{sec:TechnicalEstimates}. 

\begin{lemma}\label{lem:P2P1Estimates}
Let $P_1$ and $P_2$ be positive homogeneous functions on $\mathbb{R}^{{a}}$. If there exist $E_1\in\Exp(P_1)$ and $E_2\in\Exp(P_2)$ for which $[E_1,E_2]=0$ and $\{t^{E_1-E_2}\}$ is non-expanding, then the following statements hold:
\begin{enumerate}
\item There are positive constants $C,C',M,M'$ for which
\begin{equation}\label{eq:EstimateForSmallTime}
C(P_1(\xi)+P_2(\eta))-M\leq P_1(t^{E_1-E_2}\eta+\xi)+P_2(\eta)\leq C'(P_1(\xi)+P_2(\eta))+M'
\end{equation}
for all $\xi,\eta\in\mathbb{R}^{{a}}$ and $0<t\leq 1$.
\item There are positive constants $C,C',M,M'$ for which
\begin{equation}\label{eq:EstimateForLargeTime}
C(P_1(\eta)+P_1(\xi))-M\leq P_1(\eta)+P_2(t^{E_1-E_2}\eta+\xi)\leq C'(P_2(\eta)+P_2(\xi))+M'
\end{equation}
for all $\eta,\xi\in\mathbb{R}^{{a}}$ and $0<t\leq 1$.
\end{enumerate}
\end{lemma}

As the reader will see, this lemma is fundamental to our work throughout the article and, in addition to its appearance in the proof of Theorem \ref{thm:OnDiagonal}, it is used essentially in the proofs of Theorems \ref{thm:Perturbation} and \ref{thm:LLT}. For the aid of the reader, we find it useful to state this result in the context of our introductory example where $P(\eta,\zeta)=(\eta+\zeta^2)^2+\eta^4$. In this case (using $\xi=Q(\zeta)=\zeta^2$), the lemma gives the following estimates.
\begin{enumerate}
    \item There are positive constants $C,\, C',\, M,\, M'$ for which
    \begin{equation*}
        C(\zeta^4+\eta^4)-M\leq (t^{1/4}\eta+\zeta^2)^2+\eta^4\leq C'(\zeta^4+\eta^4)+M'
    \end{equation*}
    for all $\eta,\zeta\in\mathbb{R}$ and $0<t\leq 1$.
    \item There are positive constants $C,\, C',\, M,\, M'$ for which
    \begin{equation*}
        C(\eta^2+\zeta^4)-M\leq \eta^2+(t^{1/4}\eta-\zeta^2)^4\leq C'(\eta^4+\zeta^8)+M'
    \end{equation*}
    for all $\eta,\zeta\in\mathbb{R}$ and $0<t\leq 1$.
\end{enumerate}

\begin{proof}[Proof of Theorem \ref{thm:OnDiagonal}]
We first treat the $t\leq 1$ behavior. Using the homogeneity of $P_1$ and $P_2$, we observe that
\begin{equation*}
tP(\eta,\zeta)=P_1\left(t^{E_1}\eta+t^{E_1}Q(\zeta)\right)+P_2(t^{E_2}\eta)=P_1\left(t^{E_1}\eta+Q(t^{F_1}\zeta)\right)+P_2(t^{E_2}\eta)
\end{equation*}
for all $t>0$, $\eta\in\mathbb{R}^{{a}}$ and $\zeta\in\mathbb{R}^{{b}}$. By making the change of variables $(\eta,\zeta)\mapsto (t^{-E_2}\eta,t^{-F_1}\zeta)$, we see that
\begin{equation}\label{eq:SmallTScaled}
\varphi(t)=H_P^t(0)=\frac{t^{-(\tr E_2+\tr F_1)}}{(2\pi)^{d}}\int_{\mathbb{R}^{d}}e^{-\left(P_1\left(t^{E_1-E_2}\eta+Q(\zeta)\right)+P_2(\eta)\right) }\,d\eta\,d\zeta
\end{equation}
for $t>0$. By virtue of Lemma \ref{lem:P2P1Estimates}, we find that there are positive constants $C,C',M,M'$ for which
\begin{equation}\label{eq:SmallTEst}
-M'-C'(P_1(Q(\zeta))+P_2(\eta))\leq -P_1(t^{E_1-E_2}\eta+Q(\zeta))-P_2(\eta)\leq M-C(P_1(Q(\zeta))+P_2(\eta))
\end{equation}
for all $t\leq 1$, $\eta\in\mathbb{R}^{{a}}$, and $\zeta\in\mathbb{R}^{{b}}$. Upon noting that $(\eta,\zeta)\mapsto P_1(Q(\zeta))+ P_2(\eta)$ is a positive homogeneous function on $\mathbb{R}^{d}$ in view of Proposition \ref{prop:CompIsPosHom}, it follows from Lemma \ref{lem:ExpIntegrability} that
\begin{equation*}
\int_{\mathbb{R}^{d}}e^{-\epsilon (P_1(Q(\zeta))+P_2(\eta))}\,d\eta\,d\zeta
\end{equation*} 
is a positive finite number for each $\epsilon>0$. With this in mind, the inequality \eqref{eq:SmallTEst} guarantees constants $C,C'>0$ for which
\begin{equation*}
C\leq\frac{1}{(2\pi)^{d}}\int_{\mathbb{R}^{d}}e^{-P_1\left(t^{E_1-E_2}\eta+Q(\zeta)\right)}e^{-P_2(\eta)}\,d\eta\,d\zeta\leq C'
\end{equation*}
for all $0<t\leq 1$, and from \eqref{eq:SmallTScaled} we conclude that
\begin{equation*}
\varphi(t)\asymp t^{-(\tr E_2+\tr F_1)}=t^{-(\mu_{P_2}+\mu_{P_1\circ Q})}
\end{equation*}
for $t\leq 1$. 

On the other hand, establishing the $t\geq 1$ asymptotics is more difficult. A first hope would be to introduce the change of variables $(\eta,\zeta)\mapsto (t^{-E_1}\eta,t^{-F_1}\zeta)$ and find that
\begin{equation*}
\varphi(t)=\frac{t^{-(\tr E_1+\tr F_1)}}{(2\pi)^{d}}\int_{\mathbb{R}^{d}}e^{-P_1(\eta+Q(\zeta))-P_2(t^{E_2-E_1}\eta)}\,d\eta\,d\zeta
\end{equation*}
for $t>0$. This is problematic, however. To see this, let us assume momentarily that $\{t^{E_1-E_2}\}$ is contracting and observe that
\begin{equation*}
\lim_{t\to\infty}\int_{\mathbb{R}^{d}}e^{-P_1(\eta+Q(\zeta))-P_2(t^{E_2-E_1}\eta)}\,d\eta\,d\zeta=\int_{\mathbb{R}^{d}}e^{-P_1(\eta+Q(\eta))}\,d\eta\,d\zeta=\infty
\end{equation*}
by virtue of Fatou's lemma and the fact that
 \begin{equation*}
\int_{\mathbb{R}^{d}}e^{-P_1(\eta+Q(\zeta))}\,d\eta\,d\zeta=\int_{\mathbb{R}^{d}}e^{-P_1(\eta')}\,d\eta'\,d\zeta'=\infty.
\end{equation*}
Consequently, in the case that $\{t^{E_1-E_2}\}$ is contracting, the only information found by this argument is that $\varphi(t)$ decays more slowly than $t^{-(\tr E_1+\tr F_1)}=t^{-(\mu_{P_1}+\mu_{P_1\circ Q})}$ as $t\to\infty$. 

Instead, we return to the assumption that $\{t^{E_1-E_2}\}$ is non-expanding and make the non-linear change of variables $(\eta,\zeta)\mapsto (\eta-Q(\zeta),\zeta)$. Denoting this transformation by $T$, we find that
\begin{equation*}
DT(\eta,\zeta)=\begin{pmatrix}
I_{{a}} & -D_{\zeta}Q\\
0 & I_{{b}}
\end{pmatrix}
\end{equation*}
where $DT$ and $D_\zeta Q$ are the Jacobian matrices for $T$ and $Q$, respectively, and $I_{{a}}$ and $I_{{b}}$ are the identity matrices on $\mathbb{R}^{{a}}$ and $\mathbb{R}^{{b}}$, respectively. With this, it is easy to see that $T$ is measure preserving and consequently
\begin{eqnarray*}
\varphi(t)&=&\frac{1}{(2\pi)^{d}}\int_{\mathbb{R}^{d}}e^{-t(P\circ T)(\eta,\zeta)}\det(DT(\eta,\zeta))\,d\eta \,d\zeta\\
&=&\frac{1}{(2\pi)^{d}}\int_{\mathbb{R}^{d}}e^{-t\widetilde{P}(\eta,\zeta)}\,d\eta \,d\zeta
\end{eqnarray*}
for $t>0$, where
\begin{equation*}
\widetilde P(\eta,\zeta)=(P\circ T)(\eta,\zeta)=P_1(\eta)+P_2(\eta-Q(\zeta))
\end{equation*}
for $(\eta,\zeta)\in\mathbb{R}^{d}$. Observe that
\begin{equation*}
t\widetilde{P}(\eta,\zeta)=P_1(t^{E_1}\eta)+P_2(t^{E_2}\eta-Q(t^{F_2}\zeta))
\end{equation*}
for $t>0$ and $(\eta,\zeta)\in\mathbb{R}^{d}$ and, upon making the change of variables $(\eta,\zeta)\mapsto (t^{-E_1}\eta,t^{-F_2}\zeta)$, it follows that
\begin{equation}\label{eq:LargeTimeAsymptoticRepresentation}
\varphi(t)=\frac{t^{-(\tr E_1+\tr F_2)}}{(2\pi)^{d}}\int_{\mathbb{R}^{d}}e^{-P_1(\eta)-P_2(t^{E_2-E_1}\eta-Q(\zeta))}\,d\eta\,d\zeta
\end{equation}
for $t>0$. By virtue of Lemma \ref{lem:P2P1Estimates}, there are positive constants $C,C',M,M'$ for which
\begin{equation*}
M'-C'(P_2(\eta)+P_2(-Q(\zeta))\leq -(P_1(\eta)+P_2(t^{E_2-E_1}\eta-Q(\zeta))\leq M-C(P_1(\eta)+P_1(-Q(\zeta)))
\end{equation*}
for all $\eta\in\mathbb{R}^{{a}}$, $\zeta\in\mathbb{R}^{{b}}$, and $1\leq t<\infty$ since $\{t^{E_1-E_2}\}$ is non-expanding and $t^{E_2-E_1}=(1/t)^{E_1-E_2}$. With the observation that $(\eta,\zeta)\mapsto P_1(\zeta)+P_1(-Q(\eta))$  and $(\eta,\zeta)\mapsto P_2(\zeta)+P_2(-Q(\eta))$ are positive homogeneous functions on $\mathbb{R}^{d}$, it follows from \eqref{eq:EstimateForLargeTime} and Lemma \ref{lem:ExpIntegrability} that there are positive numbers $C,C'$ for which
\begin{equation*}
C\leq \int_{\mathbb{R}^{d}}e^{-P_1(\eta)-P_2(t^{E_2-E_1}\eta-Q(\zeta))}\,d\eta\,d\zeta\leq C'
\end{equation*}
for all $1\leq 1<\infty$. Hence,
\begin{equation*}
\varphi(t)\asymp t^{-(\tr E_1+\tr F_2)}=t^{-\left(\mu_{P_1}+\mu_{P_2\circ(-Q)}\right)}
\end{equation*}
for $t\geq 1$. 
\end{proof}

Our final result in this section shows that, in the case that $\{t^{E_1-E_2}\}$ is contracting, the asymptotics of Theorem \ref{thm:OnDiagonal} are ``true asymptotics''.
\begin{theorem}\label{thm:TrueAsymptotic}
Let $P:\mathbb{R}^2\to\mathbb{R}$ satisfy the hypotheses of Theorem \ref{thm:OnDiagonal} and, for $P_1$, $P_2$, and $Q$ as in the theorem's statement, put
\begin{equation*}
   P_0(\xi)=P_1(Q(\zeta))+P_2(\eta)\hspace{1cm}\mbox{and}\hspace{1cm} P_\infty(\xi)=P_1(\eta)+P_2(-Q(\zeta))
\end{equation*}
for $\xi=(\eta,\zeta)\in\mathbb{R}^d$. Also, let $\varphi(t)$, $\mu_0$, $\mu_{\infty}$, $E_1$, and $E_2$ be as they appear in the statement of Theorem \ref{thm:OnDiagonal}. If $\{t^{E_1-E_2}\}$ is contracting, then
\begin{equation*}
    \lim_{t\to 0}t^{\mu_0}\varphi(t)=H_{P_0}^1(0)=\frac{1}{(2\pi)^d}\int_{\mathbb{R}^d}e^{-P_0(\xi)}\,d\xi
\end{equation*}
and
\begin{equation*}
    \lim_{t\to\infty}t^{\mu_\infty}\varphi(t)=H_{P_\infty}^1(0)=\frac{1}{(2\pi)^d}\int_{\mathbb{R}^d}e^{-P_\infty(\xi)}\,d\xi.
\end{equation*}
In particular, $\lim_{t\to\infty}t^{\mu_\infty}\varphi(t)$ exists, is a positive number, and can be computed using Lemma \ref{lem:ExpIntegrability}.
\end{theorem}
\begin{proof}
We shall prove the statement involving the limit as $t\to\infty$; the $t\to 0$ statement is proved analogously. In view of \eqref{eq:LargeTimeAsymptoticRepresentation}, we have
\begin{equation*}
    t^{\mu_\infty}\varphi(t)=\frac{1}{(2\pi)^d}\int_{\mathbb{R}^d}e^{-P_1(\eta)-P_2(t^{E_1-E_2}\eta-Q(\zeta))}\,d\xi
\end{equation*}
for $t>0$. Now, given that $\{t^{E_1-E_2}\}$ is contracting, we have
\begin{equation*}
    \lim_{t\to\infty}e^{-P_1(\eta)-P_2(t^{E_2-E_1}\eta-Q(\zeta))}=e^{-P_\infty(\xi)}
\end{equation*}
for each $\xi=(\eta,\zeta)\in\mathbb{R}^d$. As noted in the proof of Theorem \ref{thm:OnDiagonal} in the paragraph following \eqref{eq:LargeTimeAsymptoticRepresentation}, the integrands $\xi\mapsto \exp(-P_1(\eta)-P_2(t^{E_2-E_1}\eta-Q(\zeta)))$ are uniformly dominated, for $t\geq 1$, by the integrable function $\xi\mapsto \exp(M-C(P_1(\eta)+P_1(-Q(\zeta))))$ and so our desired result follows by an appeal to the dominated convergence theorem.
\end{proof}

\begin{example}\label{ex:OnDiagonalMotivatingExampleTrueAsymptotics}
In Example \ref{ex:OnDiagonalMotivatingExample}, we found that
\begin{equation*}
    \varphi(t)=H_P^t(0)\asymp 
    \begin{cases} t^{-(1/l+1/qp)} &0< t\leq 1\\
    t^{-(1/q+1/lp)} & t\geq 1
    \end{cases}
\end{equation*}
for $t>0$ where $P(\eta,\zeta)=P_1(\eta+Q(\zeta))+P_2(\eta)=(\eta+\zeta^p)^q+\eta^l$ for $p\in\mathbb{N}_+$ and positive even integers $q$ and $l$ with $q\leq l$. As we noted in the example, $E_1-E_2=(1/q-1/l)I$ where $I$ is the identity transformation on $\mathbb{R}$ and therefore $\{t^{E_1-E_2}\}$ is contracting whenever $q<l$. Upon noting that $P_0(\eta,\zeta)=\zeta^{qp}+\eta^l$ and $P_\infty(\eta,\zeta)=\eta^q+\zeta^{lp}$ for $(\eta,\zeta)\in\mathbb{R}^2$, Theorem \ref{thm:TrueAsymptotic} guarantees that
\begin{eqnarray*}
\lim_{t\to 0}t^{(1/l+1/qp)}\varphi(t)=H_{P_0}^1(0)&=&\frac{1}{4\pi^2}\left(\int_{\mathbb{R}}e^{-\zeta^{qp}}\,d\zeta\right)\left(\int_{\mathbb{R}}e^{-\eta^l}\,d\eta\right)\\
&=&\frac{1}{\pi^2}\Gamma\left(1+\frac{1}{qp}\right)\Gamma\left(1+\frac{1}{l}\right)
\end{eqnarray*}
and similarly
\begin{equation*}
\lim_{t\to\infty}t^{(1/q+1/lp)}\varphi(t)=H_{P_\infty}^1(0)=\frac{1}{\pi^2}\Gamma\left(1+\frac{1}{q}\right)\Gamma\left(1+\frac{1}{lp}\right)
\end{equation*}
provided $q<l$. In particular, for our introductory example in which $P(\eta,\zeta)=(\eta+\zeta^2)^2+\eta^4$, i.e., where $2=p=q<l=4$, this gives the (expected) limit \eqref{eq:IntroSmallTimeTrueAsymptotic} and, more interestingly,
\begin{equation*}
    \lim_{t\to\infty}t^{5/8}\varphi(t)=H_{P_\infty}^1(0)=\frac{1}{2\pi^{3/2}}\Gamma(9/8)\approx 0.0845624.
\end{equation*}
\end{example}

\section{A perturbation theory}

Let us take the classical viewpoint that the theory of elliptic/semi-elliptic operators is a ``perturbation theory'' in which a sufficiently well-behaved partial differential operator is perturbed by adding operators whose order is lower than that of the given operator. In that setting, one may investigate properties of solutions to related partial differential equations which are preserved under such perturbations. For example, in the theory of elliptic operators, short-time heat kernel estimates for a uniformly elliptic operator are determined by the operator's principal symbol. In this way, perturbation by lower-order operators -- provided they are sufficiently well behaved -- will not essentially change the short-time behavior of heat kernels. In this short section, we explore perturbation by higher-order operators/symbols. In particular, we show that, under certain conditions, the large-time decay of $\varphi(t)=H_P^t(0)$ is essentially unchanged when $P$ is replaced with $P+R$ where $R(\xi)=o(P(\xi))$ as $\xi\to 0$. Given that our analysis is done exclusively in the frequency domain, our results amount, essentially, to a perturbation theory of constant-coefficient operators. We suspect that a successful variable-coefficient theory is possible, however, we do not pursue that here. Our main result is as follows.

\begin{theorem}\label{thm:Perturbation}
Let $P$ satisfy the hypotheses of Theorem \ref{thm:OnDiagonal} and let $R:\mathbb{R}^d\to\mathbb{C}$ be a continuous function for which $R(\xi)=o(P(\xi))$ as $\xi\to 0$ and $\xi\mapsto \Re(R(\xi))$ non-negative. Then
\begin{equation*}
H_{P+R}^t(x)=\frac{1}{(2\pi)^d}\int_{\mathbb{R}^d}e^{-t(P(\xi)+R(\xi))}e^{-ix\cdot\xi}\,d\xi
\end{equation*}
exists for all $t>0$ and $x\in\mathbb{R}^d$. Further, for $H_P=H_P^{(\cdot)}(\cdot)$ and $\mu_\infty$ as given in Theorem \ref{thm:OnDiagonal},
\begin{equation}\label{eq:PerturbationUniform}
    H_{P+R}^t(x)=H_P^t(x)+o(t^{-\mu_\infty})
\end{equation}
uniformly for $x\in\mathbb{R}^d$ as $t\to\infty$. In particular, we have the following large-time on-diagonal asymptotics:
\begin{enumerate}
\item \begin{equation}\label{eq:PerturbationOnDiagonal}
    \abs{H_{P+R}^t(0)}\asymp t^{-\mu_\infty}
\end{equation}
for $t\geq 1$.
\item If $\{t^{E_1-E_2}\}$ is contracting (where $E_1$ and $E_2$ are as given in the statement of Theorem \ref{thm:OnDiagonal}), then
\begin{equation}\label{eq:PerturbationOnDiagonalTrue}
    \lim_{t\to\infty}t^{\mu_\infty}H_{P+R}^t(0)=H_{P_\infty}^1(0)=\frac{1}{(2\pi)^d}\int_{\mathbb{R}^d}e^{-P_\infty(\xi)}\,d\xi
\end{equation}
where $P_\infty$ is as defined in Theorem \ref{thm:TrueAsymptotic}.
\end{enumerate}
\end{theorem}

Before proving the theorem, we shall first treat a technical lemma which will also be found useful in our application to convolution powers of complex-valued functions presented in Section \ref{sec:ConvPower}. The lemma introduces the useful notion of subhomogeneity on an ad hoc basis; for a more complete treatment, we refer the reader to Section 2 of \cite{BR22}.

\begin{lemma}\label{lem:RTildeNice}
Let $P$ satisfy the hypotheses of Theorem \ref{thm:OnDiagonal} and take $E_1$ and $F_2$ as in the statement of the theorem. For convenience of notation, we set $G=E_1\oplus F_2$ and $\widetilde{P}=P\circ T$ where $T:\mathbb{R}^d\to\mathbb{R}^d$ is the measure-preserving transformation defined by $T(\eta,\zeta)=(\eta-Q(\zeta),\zeta)$ for $(\eta,\zeta)\in\mathbb{R}^a\times\mathbb{R}^b=\mathbb{R}^d$. Finally, given an open neighborhood $\mathcal{O}\subseteq\mathbb{R}^d$ of $0$, let $R:\mathcal{O}\to\mathbb{C}$ be a continuous function and set $\widetilde{R}=R\circ T$. Then the following statements hold.
\begin{enumerate}
\item\label{item:RTildeNice1} $R(\xi)=o(P(\xi))$ as $\xi\to 0$ if and only if $\widetilde{R}(\xi)=o(\widetilde{P}(\xi))$ as $\xi\to 0$. 
\item\label{item:RTildeNice2} If either of the preceding equivalent conditions are satisfied, then $\widetilde{R}$ is subhomogeneous with respect to $G$ in the sense that, for each $\epsilon>0$ and compact set $K\subseteq\mathbb{R}^d$, there exists $t_0>0$ for which
\begin{equation*}
\abs{\widetilde{R}(t^G\xi)}\leq \epsilon t
\end{equation*}
whenever $0<t\leq t_0$ and $\xi\in K$.
\end{enumerate}
\end{lemma}
\begin{proof}
Because $T$ is a homeomorphism with $T(0)=0$, the first assertion is immediate. For the second assertion, let us fix a compact set $K$ and a positive number $\epsilon$. Also, given that $\widetilde{R}(\xi)=o(\widetilde{P}(\xi))$ as $\xi\to 0$, let $\delta>0$ be such that $\abs{\widetilde{R}(\xi)}\leq (\epsilon/M)\widetilde{P}(\xi)$ for $\abs{\xi}\leq \delta$ where
\begin{equation*}
M=\sup_{\xi=(\eta,\zeta)\in K}C'(P_2(\eta)+P_2(-Q(\zeta))+M'
\end{equation*}
where $C'$ and $M'$ are those positive constants appearing in \eqref{eq:EstimateForLargeTime} of Lemma \ref{lem:P2P1Estimates}. Using the fact that $\{t^G\}$ is a contracting group, there exists $0<t_0\leq 1$ for which $\abs{t^{G}\xi}\leq \delta$ for all $\xi\in K$ and $0<t\leq t_0$. Consequently, for $\xi=(\eta,\zeta)\in K$ and $0<t\leq t_0$, we have
\begin{eqnarray*}
\abs{\widetilde{R}(t^{G}\xi)}&\leq & (\epsilon/M)\widetilde{P}(t^{G}\xi)\\
&=&(\epsilon/M)\left(P_1(t^{E_1}\eta)+P_2(t^{E_1}\eta-t^{E_2}Q(\zeta))\right)\\
&=&(\epsilon/M)\left(tP_1(\eta)+tP_2(r^{E_1-E_2}\eta-Q(\zeta))\right)\\
& \leq &(\epsilon/M)t\left(C'(P_2(\eta)+P_2(-Q(\zeta))+M'\right)\\
&\leq &\epsilon t
\end{eqnarray*}
thanks to Lemma \ref{lem:P2P1Estimates}.
\end{proof}

\begin{proof}[Proof of Theorem \ref{thm:Perturbation}]
Given that 
\begin{equation*}
\abs{e^{-t(P(\xi)+R(\xi)}}=e^{-tP(\xi)}e^{-t\Re(R(\xi))}\leq e^{-tP(\xi)}
\end{equation*}
for $\xi\in\mathbb{R}^d$, the first assertion follows immediately from Theorem \ref{thm:OnDiagonal}. For the second assertion, observe that
\begin{eqnarray*}
H_{P+R}^t(x)-H_P^t(x)&=&\frac{1}{(2\pi)^d}\int_{\mathbb{R}^d}e^{-tP(\xi)}\left(e^{-tR(\xi)}-1\right)e^{-ix\cdot\xi}\,d\xi\\
&=&\frac{1}{(2\pi)^d}\int_{\mathbb{R}^d}e^{-t\widetilde{P}(\xi)}\left(e^{-t\widetilde{R}(\xi)}-1\right)e^{-ix\cdot T(\xi)}\,d\xi\\
&=&\frac{t^{-\mu_\infty}}{(2\pi)^d}\int_{\mathbb{R}^d}e^{-t\widetilde{P}(t^{-G}\xi)}\left(e^{-t\widetilde{R}(t^{-G}\xi)}-1\right)e^{-ix\cdot T(t^{-G}\xi)}\,d\xi
\end{eqnarray*} 
for $t>0$ and $x\in\mathbb{R}^d$ where $G=E_1\oplus F_2$ and $\mu_\infty=\tr E_1+\tr F_2=\tr G$. By an appeal to Lemma \ref{lem:P2P1Estimates}, we find that
\begin{equation*}
-t\widetilde{P}(t^{-G}\xi)=-(P_1(\eta)+P_2(t^{E_2-E_1}\eta-Q(\zeta))\leq M-C P_{*}(\xi)
\end{equation*}
for all $t\geq 1$ where $M$ and $C$ are positive constants and we have set $P_*(\xi)=P_1(\eta)+P_1(-Q(\zeta))$ for $\xi=(\eta,\zeta)\in\mathbb{R}^d$. Consequently,
\begin{equation}\label{eq:Perturbation1}
t^{\mu_\infty}\abs{H_{P+R}^t(x)-H_P^t(x)}\leq e^M\int_{\mathbb{R}^d}e^{-C P_*(\xi)}\abs{e^{-t\widetilde{R}(t^{-G}\xi)}-1}\,d\xi
\end{equation}
for $t\geq 1$ and $x\in\mathbb{R}^d$. Given that $R(\xi)=o(P(\xi))$ as $\xi\to 0$, Lemma \ref{lem:RTildeNice} guarantees that
\begin{equation*}
\lim_{t\to \infty}e^{-\rho P_*(\xi)}\abs{e^{-t\widetilde{R}(t^{-G}\xi)}-1}=\lim_{s\to 0}e^{-\rho P_*(\xi)}\abs{e^{-s\widetilde{R}(s^{-G}\xi)}-1}=0
\end{equation*}
for each $\xi\in\mathbb{R}^d$. Upon noting that the integrand in \eqref{eq:Perturbation1} is dominated by the integrable function $\xi\mapsto 2e^{-C P_*(\xi)}$ (see Lemma \ref{lem:ExpIntegrability}), an appeal to the dominated convergence theorem guarantees that, for each $\epsilon>0$, there exists $t_0\geq 1$ for which
\begin{equation*}
    t^{\mu_\infty}\abs{H_{P+R}^t(x)-H_P^t(x)}<\epsilon
\end{equation*}
for all $x\in\mathbb{R}^d$ and $t\geq t_0$; this is precisely the uniform limit \eqref{eq:PerturbationUniform}. Applying this result at $x=0$, we immediately obtain \eqref{eq:PerturbationOnDiagonal} from Theorem \ref{thm:OnDiagonal} and \eqref{eq:PerturbationOnDiagonalTrue} from Theorem \ref{thm:TrueAsymptotic}.
\end{proof}

\begin{example}
For the operator 
\begin{equation*}
    \Lambda=\partial_{x_1}^4+\partial_{x_2}^4+2i\partial_{x_1}\partial_{x_2}^2-\partial_{x_1}^2,
\end{equation*}
we consider the perturbation $\Lambda+\Lambda^2$ with symbol $P(\xi)+R(\xi)$ where $P(\xi)=(\eta+\zeta^2)^2+\eta^4$ and $R(\xi)=P(\xi)^2$ for $\xi=(\eta,\zeta)\in\mathbb{R}^2$. As shown in Example \ref{ex:OnDiagonalMotivatingExample}, $P$ satisfies the hypotheses of Theorem \ref{thm:OnDiagonal} and, from the theorem, we obtain the large-time asymptotic: $\varphi(t)=t^{-5/8}$ for $t\geq 1$. Since $P$ is continuous at $0$ and $R(\xi)=P(\xi)^2\geq 0$, it is evident that $R$ satisfies the hypotheses of Theorem \ref{thm:Perturbation}. Consequently, the heat kernel
\begin{equation*}
    H_{P+R}^t(x)=\frac{1}{(2\pi)^2}\int_{\mathbb{R}^2}e^{-t(P(\xi)+R(\xi))}e^{-ix\cdot\xi}\,d\xi
\end{equation*}
associated to the operator $\Lambda+\Lambda^2$ has
\begin{equation*}
    H_{P+R}^t(x)=H_P^t(x)+o(t^{-5/8})
\end{equation*}
uniformly for $x\in\mathbb{R}^2$ as $t\to\infty$; here, $H_P$ is that given in \eqref{eq:HeatKernelIntro} and illustrated in Figure \ref{fig:IntroAttractor} for $t=10$. Also, in view of our analysis in Example \ref{ex:OnDiagonalMotivatingExampleTrueAsymptotics}, Theorem \ref{thm:Perturbation} gives us the large-time asymptotics, $\abs{H_{P+R}^t(0)}\asymp t^{-5/8}$ for $t\geq 1$ and
\begin{equation*}
    \lim_{t\to\infty}t^{5/8}H_{P+R}^t(0)=\frac{1}{2\pi^{3/2}}\Gamma(9/8).
\end{equation*}
We note that, by contrast, $H_{P+R}^t(0)$ does not obey the $t^{-1/2}$ small-time on-diagonal asymptotic of $H_P$. Indeed, $\Lambda+\Lambda^2$ is an eighth-order elliptic operator and necessarily decays as $t^{-1/4}$ in small time.
\end{example}

The following example generalizes the previous one and places the result in the context semigroups and ultracontractivity.

\begin{example}
Let $\Lambda$ be a constant-coefficient partial differential operator on $\mathbb{R}^d=\mathbb{R}^a\times\mathbb{R}^b$ with polynomial symbol
\begin{equation*}
P(\xi)=P_1(\eta+Q(\zeta))+P_2(\eta)
\end{equation*}
satisfying the hypotheses of Theorem \ref{thm:OnDiagonal}. Let $q(\lambda)$ be a real-valued polynomial of a single real variable for which $q(0)=0$, $q'(0)=1$ and $q(\lambda)\geq \lambda$ for all $\lambda\geq 0$. Using the Fourier transform or the spectral calculus, it is easy to see that $q(\Lambda)$ is a positive self-adjoint operator on $L^2(\mathbb{R}^d)$ and therefore $-q(\Lambda)$ generates a continuous semigroup $\{e^{-tq(\Lambda)}\}$ on $L^2(\mathbb{R}^d)$. Denoting by $E$ the spectral resolution of $\Lambda$, observe that, for $f\in L^2(\mathbb{R}^d)$,
\begin{equation*}
\|e^{-t(q(\Lambda)-\Lambda)}f\|_2^2=\int_{0}^\infty e^{-2t(q(\lambda)-\lambda)}\,dE_{f,f}(\lambda)\leq \|f\|_2^2
\end{equation*} 
and therefore, for each $t>0$, $\{e^{-t(q(\Lambda)-\Lambda)}\}$ is a contraction on $L^2(\mathbb{R}^d)$. Consequently,
\begin{equation*}
\|e^{-tq(\Lambda)}\|_{2\to\infty}=\|e^{-t\Lambda}e^{-t(q(\Lambda)-\Lambda)}\|_{2\to\infty}\leq \|e^{-t(q(\Lambda)-\Lambda)}\|_{2\to 2}\|e^{-t\Lambda }\|_{2\to \infty}\leq \sqrt{\varphi(2t)}\leq C't^{-\mu_\infty}
\end{equation*}
for $t\geq 1$ where $C'$ is a positive constant and
\begin{equation*}
    \mu_\infty=\mu_{P_1}+\mu_{P_2\circ(-Q)}=\tr E_1+\tr F_2
\end{equation*}
as given in the statement of Theorem \ref{thm:OnDiagonal}. By duality, we find that $\|e^{-tq(\Lambda)}\|_{1\to\infty}\leq Ct^{-\mu_\infty}$ for $t\geq 1$ for some positive constant $C$. It follows (see Lemma 2.1.2 of \cite{Da89}) that $\{e^{-tq(\Lambda)}\}$ has integral representation
\begin{equation*}
\left(e^{-tq(\Lambda)}f\right)(x)=\int_{\mathbb{R}^d}H^t(x,y)f(y)\,dy
\end{equation*}
where
\begin{equation}\label{eq:FirstPerturbExample}
\|e^{-tq(\Lambda)}\|_{1\to\infty}=\sup_{x,y}\abs{H^t(x,y)}\leq C t^{-\mu_\infty}
\end{equation}
for $t\geq 1$. Of course, given that $q$ is a polynomial, it is easy to verify that, in fact, $H^t(x,y)=H_{q\circ P}^t(x-y)$ for $x,y\in\mathbb{R}^d$ and $t>0$ where
\begin{equation*}
H_{q\circ P}^t(x)=\frac{1}{(2\pi)^d}\int_{\mathbb{R}^d}e^{-t(q\circ P)(\xi)}e^{-ix\cdot\xi}\,d\xi
\end{equation*}
for $\xi=(\eta,\zeta)\in\mathbb{R}^d$. Thus,
\begin{equation*}
\|e^{-tq(\Lambda)}\|_{1\to\infty}=\sup_{x,y}\abs{H^t(x,y)}=H_{q\circ P}^t(0).
\end{equation*}
We claim that
\begin{equation*}
H_{q\circ P}^t(0)\asymp t^{-\mu_\infty}
\end{equation*}
for $t\geq 1$ and so the upper bound in the ultracontractive estimate \eqref{eq:FirstPerturbExample} is optimal. Indeed, under the given hypotheses concerning $q$, we may write
\begin{equation*}
q(\Lambda)=\Lambda+r(\Lambda)
\end{equation*}
where $r$ is a polynomial having $r(\lambda)\to 0$ as $\lambda\to 0$ and $r(\lambda)\geq 0$ for $\lambda\geq 0$. From this it follows that 
\begin{equation*}
(q\circ P)(\xi)=P(\xi)+R(\xi)
\end{equation*}
where $R(\xi)=r(P(\xi))$ is continuous, non-negative, and has $R(\xi)=o(P(\xi))$ as $\xi\to 0$. With this, our claim follows by an application of Theorem \ref{thm:Perturbation}. If we additionally assume that, for $E_1$ and $E_2$ as they appear in the statement of Theorem \ref{thm:OnDiagonal}, $\{t^{E_1-E_2}\}$ is contracting, then Theorem \ref{thm:Perturbation} also guarantees that
\begin{equation*}
\lim_{t\to\infty}t^{\mu_\infty}H_{q\circ P}^t(0)=H_{P_\infty}^1(0).
\end{equation*}
\end{example}

In contrast to the preceding examples, we now consider a perturbation of an operator $\Lambda$ by one which is not easily comparable to $\Lambda$. 

\begin{example}\label{ex:PerturbByLaplace}
Consider
\begin{equation}\label{eq:PerturbByLaplace1}
    \Lambda+(-\Delta)^5
\end{equation}
where $\Delta=\partial_{x_1}^2+\partial_{x_2}^2$ is the Laplacian on $\mathbb{R}^2$ and $\Lambda=\partial_{x_1}^4+\partial_{x_2}^4+2i\partial_{x_1}\partial_{x_2}^2-\partial_{x_1}^2.$ Associated to the operator \eqref{eq:PerturbByLaplace1} is the heat kernel
\begin{equation*}
 H_{P+R}^t(x)=\frac{1}{(2\pi)^d}\int_{\mathbb{R}^2}e^{-t(P(\xi)+R(\xi))}e^{-ix\cdot\xi}\,d\xi
\end{equation*}
where
\begin{equation*}
    P(\xi)=(\eta+\zeta^2)^2+\eta^4\hspace{1cm}\mbox{and}\hspace{1cm}R(\xi)=(\eta^2+\zeta^2)^5
\end{equation*}
for $\xi=(\eta,\zeta)\in\mathbb{R}^2$. We claim that $R(\xi)=o(P(\xi))$ as $\xi\to 0$. To see this, we consider the open neighborhood $\mathcal{O}=\{\xi=(\eta,\zeta)\in\mathbb{R}^2:P(\xi)<1\}$ of $0$ and write $\mathcal{O}=\mathcal{R}_1\cup\mathcal{R}_2$ where
\begin{equation}\label{eq:Region1}
   \mathcal{R}_1=\left\{\xi=(\eta,\zeta)\in\mathcal{O}:\abs{\eta+\zeta^2}\leq(1-1/\sqrt{2})\zeta^2\right\}
   \end{equation}
 and
 \begin{equation}\label{eq:Region2}
 \mathcal{R}_2= \left\{\xi=(\eta,\zeta)\in\mathcal{O}:\abs{\eta+\zeta^2}\geq (1-1/\sqrt{2})\zeta^2\right\}.
\end{equation}
For $\xi=(\eta,\zeta)\in\mathcal{R}_1$, observe that $0\leq \zeta^2\leq -\sqrt{2}\eta$ and therefore 
\begin{equation*}
    4 P(\xi)=4\left((\eta+\zeta^2)^2+\eta^4\right)\geq 4\eta^4\geq \zeta^8. 
\end{equation*}
Thus,
\begin{equation}\label{eq:PerturbByLaplace2}
    \eta^2+\zeta^2\leq P(\xi)^{1/2}+\sqrt{2}P(\xi)^{1/4}=\left(P(\xi)^{1/4}+\sqrt{2}\right)P(\xi)^{1/4}\leq 5P(\xi)^{1/4}
\end{equation}
for $\xi=(\eta,\zeta)\in\mathcal{R}_1$. On $\mathcal{R}_2$, we find that
\begin{equation*}
    P(\xi)\geq \max\left\{\eta^4,\left(1-\frac{1}{\sqrt{2}}\right)^2\zeta^4\right\}
\end{equation*}
so that
\begin{equation}\label{eq:PerturbByLaplace3}
    \eta^2+\zeta^2\leq P(\xi)^{1/2}+(2+\sqrt{2})P(\xi)^{1/2}=(3+\sqrt{2})P(\xi)^{1/2}\leq 5P(\xi)^{1/4}
\end{equation}
for $\xi=(\eta,\zeta)\in\mathcal{R}_2$. Since $\mathcal{O}=\mathcal{R}_1\cup\mathcal{R}_2$, the estimates \eqref{eq:PerturbByLaplace2} and \eqref{eq:PerturbByLaplace2} guarantee that, for each $\xi=(\eta,\zeta)\in\mathcal{O}$,
\begin{equation*}
    \abs{R(\xi)}=(\eta^2+\zeta^2)(\eta^2+\zeta^2)^4\leq 625(\eta^2+\zeta^2)P(\xi)
\end{equation*}
and, from this, our claim follows immediately. Since $R(\xi)$ is non-negative, an appeal to Theorem \ref{thm:Perturbation} is valid and we conclude that $\abs{H_{P+R}^t(0)}\asymp t^{-5/8}$ for $t\geq 1$ and
\begin{equation*}
   \lim_{t\to\infty}t^{5/8}H_{P+R}^t(0)=\frac{1}{2\pi^{3/2}}\Gamma(9/8)
\end{equation*}
in view of Example \ref{ex:OnDiagonalMotivatingExampleTrueAsymptotics}.
\end{example}

As evidenced by the preceding example, it isn't straightforward to show that $R(\xi)=o(P(\xi))$ as $\xi\to 0$. This is connected to the fact that $P$ is generally inhomogeneous and so examining a polynomial $R$ along the coordinate axes or by comparing the order of its terms against those of $P$ is not often helpful. For $P(\xi)=(\eta+\zeta^2)^2+\eta^4$, a careful study of the example shows that $R(\xi)=o(P(\xi))$ as $\xi\to 0$ provided that $R$ is a polynomial comprised of terms whose (multivariate) order is at least nine. Still, polynomials $R$ with terms of lower order can decay as ``little-o'' of $P$, e.g., $(\eta,\zeta)\mapsto \eta^4$, however, in general, it is difficult to tell. For example, consider the polynomials
\begin{equation*}
R_1(\xi)=\eta^2\zeta^2+2\eta\zeta^4+\zeta^6\hspace{1cm}\mbox{and}\hspace{1cm}R_2(\xi)=\eta^2\zeta^4
\end{equation*}
defined for $\xi=(\eta,\zeta)\in\mathbb{R}^2$. Though the polynomial $R_1$ contains terms of lower order than $R_2$ and thus decays more slowly than $R_2$ as $\xi\to 0$ (at least, along rays), $R_1(\xi)= o(P(\xi))$ as $\xi\to 0$ while, by contrast, $R_2(\xi)\neq o(P(\xi))$ as $\xi\to 0$. For sorting out these somewhat unintuitive statements, the following refinement of Lemma \ref{lem:RTildeNice} is helpful; its proof can be found in Appendix \ref{sec:TechnicalEstimates}.
\begin{proposition}\label{prop:RTildeVeryNice}
Let $P$ satisfy the hypotheses of Theorem \ref{thm:OnDiagonal} and let $P_1$, $P_2$, $Q$, $E_1$, $E_2$, $F_1$ and $F_2$ be as in the statement of the theorem. Set $G=E_1\oplus F_2$ and $\widetilde{P}=P\circ T$ where $T:\mathbb{R}^d\to\mathbb{R}^d$ is the measure-preserving transformation defined by $T(\eta,\zeta)=(\eta-Q(\zeta),\zeta)$ for $(\eta,\zeta)\in\mathbb{R}^d$. Finally, given an open neighborhood $\mathcal{O}\subseteq\mathbb{R}^d$ of $0$, let $R:\mathcal{O}\to\mathbb{C}$ be a continuous function and set $\widetilde{R}=R\circ T$. If $\{t^{E_1-E_2}\}$ is contracting, then the following statements are equivalent:
\begin{enumerate}
\item $R(\xi)=o(P(\xi))$ as $\xi\to 0$. 
\item $\widetilde{R}(\xi)=o(\widetilde{P}(\xi))$ as $\xi\to 0$. 
\item $\widetilde{R}$ is subhomogeneous with respect to $G$ in the sense that, for each $\epsilon>0$ and compact set $K\subseteq\mathbb{R}^d$, there exists $t_0>0$ for which
\begin{equation*}
\abs{\widetilde{R}(t^G\xi)}\leq \epsilon t
\end{equation*}
whenever $0<t\leq t_0$ and $\xi\in K$.
\end{enumerate}
\end{proposition}

We now use the proposition to prove the assertions made right before it.
\begin{example}\label{ex:R1R2Details}
Let $R_1$ and $R_2$ be as in the paragraph preceding the proposition and let $P(\xi)=(\eta+\zeta^2)^2+\eta^4$. As shown in Example \ref{ex:OnDiagonalMotivatingExample}, we have $P_1(\eta)=\eta^2$, $P_2(\eta)=\eta^4$, $Q(\zeta)=\zeta^2$, $E_1=I/2$, $E_2=I/4$, and $F_2=1/8$. Observe that, since $E_1-E_2=I/4$, $\{t^{E_1-E_2}\}$ is contracting and so an application of proposition is justified for $R_1$ and $R_2$. Focusing first on $R_1$, we compute
\begin{equation*}
\widetilde{R}_1(\xi)=R_1(\eta-\zeta^2,\zeta)=\left((\eta-\zeta^2)^2+2(\eta-\zeta^2)\zeta^2+\zeta^4\right)\zeta^2=\eta^2\zeta^2
\end{equation*}
and, because $G=E_1\oplus F_2$ has standard matrix representation $\diag(1/2,1/8)$,
\begin{equation*}
\widetilde{R}_1(t^G\xi)=(t^{1/2}\eta)^2(t^{1/8}\zeta)^2=t^{5/4}\eta^2\zeta^2  
\end{equation*}
for $t>0$ and $\xi=(\eta,\zeta)\in\mathbb{R}^2$. Consequently, given $\epsilon>0$ and a compact set $K\subseteq\mathbb{R}^2$, we observe that
\begin{equation*}
   \abs{ \widetilde{R}_1(t^G\xi)}\leq t^{5/4}\eta^2\zeta^2\leq t t_0^{1/4}\eta^2\zeta^2\leq \epsilon t
\end{equation*}
whenever $\xi=(\eta,\zeta)\in K$ and $0<t\leq t_0:=\epsilon^4/(1+\sup_{(\eta,\zeta)\in K}\eta^2\zeta^2)^4$. Thus $\widetilde{R}_1$ is subhomogeneous with respect to $G$ and from Proposition \ref{prop:RTildeVeryNice} we conclude that $R_1(\xi)=o(P(\xi))$ as $\xi\to 0$.

For $R_2$, we have
\begin{equation*}
    \widetilde{R}_2(\xi)=(\eta-\zeta^2)^2\zeta^4
\end{equation*}
and therefore
\begin{equation*}
\widetilde{R}_2(t^G\xi)=\left(t^{1/2}\eta-(t^{1/8}\zeta)^2\right)^2\left(t^{1/8}\zeta\right)^4=t(t^{1/4}\eta-\zeta^2)^2\zeta^4
\end{equation*}
for $t>0$ and $\xi=(\eta,\zeta)\in\mathbb{R}^2$. Thus, for any $\xi=(\eta,\zeta)\in\mathbb{R}^2$ for which $\zeta\neq 0$,
\begin{equation*}
\lim_{t\to 0}t^{-1}\widetilde{R}_2(t^G\xi)=\zeta^8\neq 0.
\end{equation*}
Consequently, $\widetilde{R}_2$ is not subhomogeneous with respect to $G$ and, by virtue of the proposition, we conclude that $R_2(\xi)\neq o(P(\xi))$ as $\xi\to 0$.
\end{example}

\begin{example}\label{ex:PerturbByLaplaceFollowUp}
We return to the set-up of Example \ref{ex:CanonicalNMHFollowUp} and let $Q:\mathbb{R}^a\to\mathbb{R}^a$ be a nondegenerate multivariate homogeneous function given by
\begin{equation}\label{eq:CanonicalQFollowUp}
    Q(\zeta)=\left(\zeta_{\sigma(1)}^{\alpha_1},\zeta_{\sigma(2)}^{\alpha_2},\dots,\zeta_{\sigma(a)}^{\alpha_a}\right)
\end{equation}
for $\zeta=(\zeta_1,\zeta_2,\dots,\zeta_a)\in\mathbb{R}^a$ where $\alpha_1,\alpha_2,\dots,\alpha_a\in\mathbb{N}_+$ and $\sigma$ is a permutation of the set $\{1,2,\dots,a\}$. As in that example we shall take positive homogeneous functions $P_1$ and $P_2$ on $\mathbb{R}^a$ for which, for $k=1,2$, $\Exp(P_k)$ contains $E_k\in\End(\mathbb{R}^a)$ with standard matrix representation $\diag(\lambda_{k,1},\lambda_{k,2},\dots,\lambda_{k,a})$ where $\lambda_{k,j}>0$ for $j=1,2,\dots,a$. As we showed in Example \ref{ex:CanonicalNMHFollowUp}, the function $P:\mathbb{R}^{2a}\to\mathbb{R}$ defined by
\begin{equation*}
    P(\xi)=P_1(\eta+Q(\zeta))+P_2(\eta)
\end{equation*}
for $\xi=(\eta,\zeta)\in\mathbb{R}^{2a}$ satisfies the hypotheses of Theorem \ref{thm:OnDiagonal} provided that $\lambda_{1,j}\geq \lambda_{2,j}$ for all $j=1,2,\dots,a$. In that case, we proved that
\begin{equation*}
    \varphi(t)=H_{P}^t(0)\asymp t^{-\mu_{\infty}}
\end{equation*}
for $t\geq 1$ where
\begin{equation*}
\mu_\infty=\sum_{j=1}^a\left(\lambda_{1,j}+\frac{\lambda_{2,j}}{\alpha_j}\right).
\end{equation*}
Motivated by the example preceding Proposition \ref{prop:RTildeVeryNice}, we shall perturb $P:\mathbb{R}^{2a}\to\mathbb{R}$ by \begin{equation*}
    L_k(\xi)=\abs{\xi}^{2k}=(\abs{\eta}^2+\abs{\zeta}^2)^k
\end{equation*}
for $\xi=(\eta,\zeta)\in\mathbb{R}^{2a}$ where $k>0$. We have the following result.
\begin{proposition}\label{prop:CanonicalPerturbation}
Suppose that $\lambda_{1,j}>\lambda_{2,j}$ for $j=1,2,\dots,a$. If
\begin{equation*}
   k>\frac{1}{2 \min_{j=1,2,\dots,a}(\lambda_{2,j}/\alpha_j)},
\end{equation*}
then $L_k(\xi)=o(P(\xi))$ as $\xi\to 0$.
\end{proposition}
\begin{exproof}
We have
\begin{equation*}
    \widetilde{L}_k(\xi)=\left(\abs{\eta-Q(\zeta)}^2+\abs{\zeta}^2\right)^k
\end{equation*}
for $\xi=(\eta,\zeta)\in\mathbb{R}^{2a}$. Using our analysis in Example \ref{ex:CanonicalNMH}, we see that $F_2$ has standard matrix representation
\begin{equation*}
    \diag\left(\frac{\lambda_{2,\sigma^{-1}}(1)}{\alpha_{\sigma^{-1}(1)}},\frac{\lambda_{2,\sigma^{-1}}(2)}{\alpha_{\sigma^{-1}(2)}},\dots,\frac{\lambda_{2,\sigma^{-1}}(a)}{\alpha_{\sigma^{-1}(a)}}\right).
\end{equation*}
Thus, for $G=E_1\oplus F_2$, we have
\begin{eqnarray*}
\widetilde{L}_k(t^G\xi)&=&\left(\abs{t^{E_1}\eta-Q(t^{F_2}\zeta)}^2+\abs{t^{F_2}\zeta}^2\right)^k\\
&=&\left(\abs{t^{E_2}(t^{E_1-E_2}\eta-Q(\zeta))}^2+\abs{t^{F_2}\zeta}^2\right)^k\\
&\leq &\left(\|t^{E_2}\|^2\abs{t^{E_1-E_2}\eta-Q(\zeta)}^2+\|t^{F_2}\|^2\abs{\zeta}^2\right)^k\\
&\leq &\left(\|t^{E_2}\|^2\left(\|t^{E_1-E_2}\|\abs{\eta}+\abs{Q(\zeta)}\right)^2+\|t^{F_2}\|^2\abs{\zeta}^2\right)
\end{eqnarray*}
for $t>0$ and $\xi=(\eta,\zeta)\in\mathbb{R}^{2a}$. Given that $E_1$, $E_2$, and $F_2$ are diagonal, our hypotheses guarantee that $
\|t^{E_1-E_2}\|\leq 1$, $\|t^{E_2}\|\leq t^{\omega}$, and $\|t^{F_2}\|\leq t^{\omega}$ for $0<t\leq 1$ where
\begin{equation*}
\omega=\min\left(\Spec(F_2)\right)=\min_{j=1,2,\dots,a}(\lambda_{2,j}/\alpha_j)\leq \min_{j=1,2\dots,a}\lambda_{2,j}=\min\left(\Spec(E_2)\right).
\end{equation*}
Therefore,
\begin{equation*}
    \widetilde{L}_k(t^G\xi)\leq \left(t^{2\omega}\left(\abs{\eta}+\abs{Q(\zeta)}\right)^2+t^{2\omega}\abs{\zeta}^2\right)^k\leq t^{2k\omega}\left(\left(\abs{\eta}+\abs{Q(\zeta)}\right)^2+\abs{\zeta}^2\right)^k
\end{equation*}
for $0<t\leq 1$ and $\xi=(\eta,\zeta)\in\mathbb{R}^{2a}$. Since $2k\omega=2k\left(\min_{j=1,2,\dots,a}\lambda_{2,j}/\alpha_j\right)>1$, it follows that, for each $\epsilon>0$ and compact set $K\subseteq\mathbb{R}^{2a}$,
\begin{equation*}
    \abs{\widetilde{L}_k(t^G\xi)}\leq \epsilon t
\end{equation*}
for $\xi\in K$ and $0<t\leq t_0$ where $t_0>0$ is chosen so that
\begin{equation*}
    t_0^{2k\omega-1}\sup_{\xi=(\eta,\zeta)\in K}\left(\left(\abs{\eta}+\abs{Q(\zeta)}\right)^2+\abs{\zeta}^2\right)^k<\epsilon.
\end{equation*}
Upon noting that $\{t^{E_1-E_2}\}$ is contracting because $\lambda_{1,j}>\lambda_{2,j}$ for $j=1,2,\dots,a$, the desired result follows immediately from Proposition \ref{prop:RTildeVeryNice}.
\end{exproof}

Upon noting that the hypotheses of Proposition \ref{prop:CanonicalPerturbation} guarantee that $\{t^{E_1-E_2}\}$ is contracting, by an application of Theorem \ref{thm:Perturbation} and Proposition \ref{prop:CanonicalPerturbation}, we immediately obtain the following corollary.

\begin{corollary}\label{cor:LaplaceCompareSemiElliptic}
Let $\Lambda$ be a constant-coefficient operator on $\mathbb{R}^{2a}$ with symbol $P(\xi)=P_1(\eta+Q(\zeta))+P_2(\eta)$ where $Q$ is given by \eqref{eq:CanonicalQFollowUp} and $P_1$ and $P_2$ are positive-definite semi-elliptic polynomials given by
\begin{equation*}
    P_1(\eta)=\sum_{|\beta:\mathbf{m}^{(1)}|=1}b_{1,\beta}\eta^\beta\hspace{1cm}\mbox{and}\hspace{1cm}P_2(\eta)=\sum_{|\beta:\mathbf{m}^{(2)}|=1}b_{2,\beta}\eta^\beta
\end{equation*}
for $\eta\in\mathbb{R}^a$, respectively, where $\mathbf{m}^{(1)}=\left(m_1^{(1)},m_2^{(1)},\dots,m_a^{(1)}\right)$ and $\mathbf{m}^{(2)}=\left(m_1^{(2)},m_2^{(2)},\dots,m_a^{(2)}\right)$ are $a$-tuples of even positive integers. Also, let $\Delta=\partial_{x_1}^2+\partial_{x_2}^2+\cdots+\partial_{x_{2a}}^2$ denote the Laplacian on $\mathbb{R}^{2a}$. If $m_j^{(2)}>m_j^{(1)}$ for all $j=1,2,\dots,a$, then, for any integer
\begin{equation*}
    k>\left(\max_{j=1,2,\dots,a}\alpha_jm_j^{(2)}\right)/2,
\end{equation*}
the heat kernel $H_{P+L_k}^t(x)$ associated to the operator $\Lambda+(-\Delta)^k$ satisfies the on-diagonal large-time asymptotics, $H_{P+L_k}^t(0)\asymp t^{-\mu_\infty}$ for $t\geq 1$ and
\begin{equation*}
    \lim_{t\to\infty}t^{\mu_\infty}H_{P+L_k}^t(0)=H_{P_\infty}^1(0)=\frac{1}{(2\pi)^d}\int_{\mathbb{R}^d}e^{-P_\infty(\xi)}\,d\xi
\end{equation*}
where
\begin{equation*}
    \mu_\infty=\sum_{j=1}^a\left(\frac{1}{m_j^{(1)}}+\frac{1}{\alpha_jm_j^{(2)}}\right)
\end{equation*}
and
\begin{equation*}
    P_\infty(\xi)=\sum_{|\beta:\mathbf{m}^{(1)}|=1}b_{1,\beta}\eta^\beta+\sum_{|\beta:\mathbf{m}^{(2)}|=1}b_{2,\beta}(-Q(\zeta))^\beta=\sum_{|\beta:\mathbf{m}^{(1)}|=1}b_{1,\beta}\eta^\beta+\sum_{|\beta:\mathbf{m}^{(2)}|=1}b_{2,\beta}(-1)^{|\beta|}\zeta_\sigma^{\alpha\odot\beta}
\end{equation*}
for $\xi=(\eta,\zeta)\in\mathbb{R}^{2a}$ where $\zeta_\sigma$ denotes the action of the permutation $\sigma$ on $\zeta=(\zeta_1,\zeta_2,\dots,\zeta_a)$ defined by $\zeta_\sigma\\
=(\zeta_{\sigma(1)},\zeta_{\sigma(2)},\dots,\zeta_{\sigma(a)})$ and for each multi-index $\beta$, we have set $|\beta|=\beta_1+\beta_2+\cdots+\beta_a$, $\alpha\odot\beta=(\alpha_1\beta_2,\alpha_2\beta_2,\dots,\alpha_a\beta_a)$.
\end{corollary}

For an easy application of this corollary, consider our motivating example in which $P(\xi)=P_1(\eta+Q(\zeta))+P_2(\eta)=(\eta+\zeta^2)^2+\eta^4$ for $\xi=(\eta,\zeta)\in\mathbb{R}^2$. In this case, $P_1$ and $P_2$ are positive-definite and semi-elliptic with $\mathbf{m}^{(1)}=(m_1^{(1)})=(2)$ and $\mathbf{m}^{(2)}=(m_1^{(2)})=(4)$, respectively, and $Q(\zeta)=\zeta^2$ is a nondegenerate multivariate homogeneous of the form \eqref{eq:CanonicalQFollowUp} with $\alpha_1=2$. Thus, for any integer
\begin{equation*}
    k>\left(\max_j \alpha_jm_j^{(2)}\right)/2=8/2=4,
\end{equation*}
the heat kernel $H_{P+L_k}$ associated to
\begin{equation*}
\Lambda+(-\Delta)^k=\partial_{x_1}^4+\partial_{x_2}^4+2i\partial_{x_1}\partial_{x_2}^2-\partial_{x_1}^2+(-\Delta)^k
\end{equation*}
has $H_{P+L_k}^t(0)\asymp t^{-5/8}$ for $t\geq 1$ and $\lim_{t\to\infty}t^{5/8}H_{P+L_k}^t(0)=\Gamma(9/8)/2\pi^{3/2}$. In particular, this holds for $k=5$ and so we have recaptured the result of Example \ref{ex:PerturbByLaplace}.
\end{example}

\section{An application to the study of convolution powers of complex-valued functions on \texorpdfstring{$\mathbb{Z}^d$}{Lg}}\label{sec:ConvPower}

Given a finitely-supported\footnote{We work with this condition for simplicity. One can assume, more generally, that $\phi$ and all of its multivariate moments are absolutely summable, c.f., \cite{RSC17}.} function $\phi:\mathbb{Z}^{d}\to\mathbb{C}$, we are interested in the behavior of its convolution powers $\phi^{(n)}:\mathbb{Z}^{d}\to\mathbb{C}$ defined iteratively by putting $\phi^{(1)}=\phi$ and, for $n\geq 2$, 
\begin{equation*}
\phi^{(n)}(x)=\sum_{y\in\mathbb{Z}^{d}}\phi^{(n-1)}(x-y)\phi(y)
\end{equation*}
for $x\in\mathbb{Z}^{d}$. In the case that $\phi$ is non-negative and $\sum_{x}\phi(x)=1$, the behavior of $\phi^{(n)}$ is well-known and is the subject of the local central limit theorem \cite{Sp64}. Beyond the probabilistic setting, the study of convolution powers of complex-valued functions dates back to the late nineteenth century and was initially investigated by E. L. de Forest through its applications to data smoothing. During the explosion of scientific computing in the mid-twentieth century, the study was reinvigorated by applications to numerical solution algorithms to partial differential equations. Early on, these studies focused almost entirely in one spatial dimension, i.e., $d=1$, and, for an account of these results and a thorough discussion of the early history, we refer the reader to the article \cite{DSC14}. Recent developments in the context of one dimension can be found in \cite{RSC15}, \cite{CF22}, and \cite{Coe22}. Moving beyond one spatial dimension, the articles \cite{RSC17}, \cite{BR22}, and \cite{Ra22} develop a theory for convolution powers of complex-valued functions on $\mathbb{Z}^d$ and, in particular, the article \cite{RSC17} establishes local limit theorems, sup-norm estimates, off-diagonal estimates, and stability results in that context. As we only briefly discuss the local limit theorems of \cite{RSC17} and \cite{Ra22} below, we refer the reader to these articles, both of which provide history and a more thorough presentation of the theory than is given here.

For our finitely-supported function $\phi:\mathbb{Z}^{d}\to\mathbb{C}$, we define its Fourier transform $\widehat{\phi}$ by
\begin{equation*}
\widehat{\phi}(\xi)=\sum_{x\in\mathbb{Z}^{d}}\phi(x)e^{ix\cdot\xi}
\end{equation*}
for $\xi\in\mathbb{R}^{d}$. With this, we obtain the representation
\begin{equation}\label{eq:ConvolutionFTRepresentation}
\phi^{(n)}(x)=\frac{1}{(2\pi)^{d}}\int_{\mathbb{T}^{d}}\widehat{\phi}(\xi)^ne^{-ix\cdot\xi}\,d\xi
\end{equation}
for $x\in\mathbb{Z}^{d}$ and $n\in\mathbb{N}_+$ where $\mathbb{T}^{d}=(-\pi,\pi]^{d}$ or, equivalently, any representation of the $d$ dimensional torus of the form $(-\pi,\pi]^{d}+\xi'$ for $\xi'\in\mathbb{R}^{d}$. Beyond the assumption of finite support, we shall assume that $\phi$ has been normalized so that
\begin{equation*}
\max_{\xi}\abs{\widehat{\phi}(\xi)}=\sup_{\xi}\abs{\widehat{\phi}(\xi)}=1.
\end{equation*}
Also, and though this assumption can be significantly weakened (c.f., \cite{RSC17,Ra22}), we shall assume that this maximum is attained at exactly one point $\xi_0\in\mathbb{T}^{d}$ and that
\begin{equation}\label{eq:GammaExpansion}
\Gamma_{\xi_0}(\xi):=\log\left(\frac{\widehat{\phi}(\xi+\xi_0)}{\widehat{\phi}(\xi_0)}\right)=i\alpha\cdot\xi-P(\xi)+R(\xi)
\end{equation}
where $\alpha\in\mathbb{R}^{d}$, $P$ is a positive-definite polynomial and $R$ is a smooth complex-valued function for which $R(\xi)=o(P(\xi))$ as $\xi\to 0$. Extending the one-dimensional results of \cite{DSC14} and \cite{RSC15}, the article \cite{RSC17} establishes local limit theorems in the case that the polynomial $P$ is a positive homogeneous function on $\mathbb{R}^d$. Under the assumptions above, if $P$ is a positive homogeneous polynomial on $\mathbb{R}^d$, Theorem 1.6 of \cite{RSC17} says that
\begin{equation}\label{eq:LLTGood}
    \phi^{(n)}(x)=\widehat{\phi}(\xi_0)^n e^{-ix\cdot\xi_0}H_P^n(x-n\alpha)+o(n^{-\mu_P})
\end{equation}
uniformly for $x\in\mathbb{Z}^d$ as $n\to\infty$ where $\mu_P$ is the homogeneous order of $P$ and, for each $x\in\mathbb{R}^d$ and $n\in\mathbb{N}_+$,
\begin{equation*}
H_P^n(x)=\frac{1}{(2\pi)^d}\int_{\mathbb{R}^d}e^{-nP(\xi)}e^{-ix\cdot\xi}\,d\xi.
\end{equation*}
Thanks to the homogeneity of $P$, it is easy to see that
\begin{equation}\label{eq:HomogeneityOfLLTKernel}
    H_P^n(x)=n^{-\mu_P}H_P^1(n^{-E^*}x)
\end{equation}
for all $x\in\mathbb{R}^d$, $n\in\mathbb{N}_+$, and $E\in\Exp(P)$ where $E^*$ denotes the adjoint of $E$. In this way, the local limit theorem \eqref{eq:LLTGood} can be written in terms of the single rescaled attractor, $H_P^1$. We note that homogeneity and the resultant space-time rescaling in \eqref{eq:HomogeneityOfLLTKernel} are central to the proof of Theorem 1.6 of \cite{RSC17}. Using homogeneity as a main ingredient and making use of the generalized polar-coordinate integration formula developed in \cite{BR22}, the recent article \cite{Ra22} extends Theorem 1.6 of \cite{RSC17} further to include the case in which the positive homogeneous polynomial $P$ is replaced by $iQ(\xi)$ where $Q$ is a real-valued function on $\mathbb{R}^d$ for which $\xi\mapsto \abs{Q(\xi)}$ is positive homogeneous; see Theorems 1.9 and Theorem 3.8 of \cite{Ra22}. At present and to our knowledge, all known local limit theorem on $\mathbb{Z}^d$ assume, in one way or another, that the expansion \eqref{eq:GammaExpansion} is dominated, at low order, by a homogeneous polynomial. The following treats an example in which this assumption is not satisfied.

\begin{theorem}\label{thm:LLT}
For positive integers $a$ and $b$, set $d=a+b$ and consider $\phi:\mathbb{Z}^{d}\to\mathbb{C}$ as above, i.e., $\phi$ is finitely supported on $\mathbb{Z}^d$, $\widehat{\phi}$ is maximized in absolute value at a single point $\xi_0\in\mathbb{T}^{d}$ and the local approximation $\Gamma_{\xi_0}(\xi)=i\alpha\cdot \xi -P(\xi)+R(\xi)$ is valid on the domain of $\Gamma_{\xi_0}$ where $\alpha\in\mathbb{R}^{d}$, $P$ is a positive-definite polynomial, and $R(\xi)=o(P(\xi))$ as $\xi\to 0$. Suppose, additionally, that $P$ satisfies the hypotheses of Theorem \ref{thm:OnDiagonal}. That is, we assume that
\begin{equation*}
P(\xi)=P_1(\eta+Q(\zeta))+P_2(\eta)
\end{equation*}
for $\xi=(\eta,\zeta)\in\mathbb{R}^{{a}}\times\mathbb{R}^{{b}}=\mathbb{R}^{d}$ where: 
\begin{enumerate}
\item $P_1$ and $P_2$ are positive homogeneous functions on $\mathbb{R}^{{a}}$ with homogeneous orders $\mu_{P_1}$ and $\mu_{P_2}$, respectively.
\item $Q:\mathbb{R}^{{b}}\to\mathbb{R}^{{a}}$ is nondegenerate multivariate homogeneous.
\item There exist $E_1\in\Exp(P_1)$ and $E_2\in\Exp(P_2)$, and $F_1,F_2\in\End(\mathbb{R}^{{b}})$ for which
\begin{enumerate}
\item For $k=1,2,$ $Q$ is homogeneous with respect to the pair $(E_k,F_k)$. 
\item We have $[E_1,E_2]=E_1E_2-E_2E_1$ and $\{t^{E_1-E_2}\}$ is non-expanding.
\end{enumerate}
\end{enumerate}
Upon setting $\mu_\phi=\mu_{P_1}+\mu_{P_2\circ(-Q)}=\tr E_1+\tr F_2$ and taking $H_P$ as in Theorem \ref{thm:OnDiagonal}, we have the following local limit theorem: For each $\epsilon>0$, there exists $N\in\mathbb{N}_+$ for which
\begin{equation*}
\abs{\phi^{(n)}(x)-\widehat{\phi}(\xi_0)^ne^{-ix\cdot\xi_0}H_P^n(x-n\alpha)}<\epsilon n^{-\mu_\phi}
\end{equation*}
for all $n\geq N$ and $x\in\mathbb{Z}^{d}$. In other words,
\begin{equation*}
\phi^{(n)}(x)=\widehat{\phi}(\xi_0)^ne^{-ix\cdot\xi_0}H_P^n(x-n\alpha)+o(n^{-\mu_\phi})
\end{equation*}
uniformly for $x\in\mathbb{Z}^d$ as $n\to\infty$.
\end{theorem}

Combining this result with Theorem \ref{thm:OnDiagonal} and Theorem \ref{thm:TrueAsymptotic}, we obtain the following corollary. 
\begin{corollary}\label{cor:ConvSupNorm}
Let $\phi$ satisfy the hypotheses of the above theorem. If $\alpha=0$, then
\begin{equation}\label{eq:ConvSupNormAsymptotic}
\|\phi^{(n)}\|_\infty\asymp n^{-\mu_\phi}
\end{equation}
for $n\geq 1$. If, additionally, $\{t^{E_1-E_2}\}$ is contracting, then
\begin{equation}\label{eq:ConvSupNormTrueAsymptotic}
    \lim_{n\to\infty}n^{\mu_\phi}\|\phi^{(n)}\|_\infty=H_{P_\infty}^1(0)=\frac{1}{(2\pi)^d}\int_{\mathbb{R}^d}e^{-P_\infty(\xi)}\,d\xi
\end{equation}
where $P_\infty(\xi)=P_1(\eta)+P_2(-Q(\zeta))$ for $\xi=(\eta,\zeta)\in\mathbb{R}^d$.
\end{corollary}
\begin{proof}
For each $\epsilon>0$, an application of Theorem \ref{thm:LLT} guarantees that
\begin{equation}\label{eq:ConvSupNorm1}
\abs{\phi^{(n)}(x)}\leq \epsilon n^{-\mu_\phi}+\abs{H_P^n(x)}\leq  \epsilon n^{-\mu_\phi}+H_P^n(0)
\end{equation}
for all $x\in\mathbb{Z}^{d}$ and sufficiently large values of $n$; here, we have noted that $\abs{\widehat{\phi}(\xi_0)^ne^{-ix\cdot\xi_0}}=1$ and $\abs{H_P^n(x)}\leq H_P^n(0)$ whenever $n\in\mathbb{N}_+$ and $x\in\mathbb{Z}^d$. By virtue of Theorem \ref{thm:OnDiagonal} and in view of the fact that $\mu_\phi=\mu_\infty$, we have
\begin{equation*}
\|\phi^{(n)}\|_\infty=\sup_{x\in\mathbb{Z}^{d}}\abs{\phi^{(n)}(x)}\leq C n^{-\mu_\phi}
\end{equation*}
for sufficiently large $n$ where $C$ is some positive constant. We now obtain a matching lower bound. Using Theorem \ref{thm:OnDiagonal}, $H_P^n(0)\geq 2C' n^{-\mu_\phi}$ for all $n\in\mathbb{N}_+$ for some $C'>0$. An appeal to Theorem \ref{thm:LLT} (with $\epsilon=C'$) now guarantees that
\begin{equation*}
2C'n^{-\mu_\phi}\leq\abs{\widehat{\phi}(\xi_0)^nH_P^n(0)}
\leq\abs{\widehat{\phi}(\xi_0)^nH_P^n(0)-\phi^{(n)}(0)}+\abs{\phi^{(n)}(0)}\leq C'n^{-\mu_\phi}+\abs{\phi^{(n)}(0)}
\end{equation*}
and consequently
\begin{equation*}
C'n^{-\mu_\phi}\leq \abs{\phi^{(n)}(0)}
\end{equation*}
for all sufficiently large $n$. Thus, for some $N\in\mathbb{N}_+$,
\begin{equation*}
C'n^{-\mu_\phi}\leq \abs{\phi^{(n)}(0)}\leq \|\phi^{(n)}\|_\infty\leq Cn^{-\mu_\phi}
\end{equation*}
for $n\geq N$. With this, the asymptotic \eqref{eq:ConvSupNormAsymptotic} follows by, if necessary, adjusting $C$ and $C'$ to account for those $n$ from $1$ to $N-1$ (and while noting that $\|\phi^{(n)}\|_\infty$ cannot vanish for any such $n$ for otherwise all subsequent convolution powers would vanish identically). 

We now prove prove \eqref{eq:ConvSupNormTrueAsymptotic}. In view of \eqref{eq:ConvSupNorm1}, for each $\epsilon>0$,
\begin{equation*}
n^{\mu_\phi}\|\phi^{(n)}\|_\infty =\sup_{x\in\mathbb{Z^d}}n^{\mu_\phi}\abs{\phi^{(n)}(x)}\leq \epsilon +n^{\mu_\phi}H_P^n(0)
\end{equation*}
for all sufficiently large $n$. Upon noting that $\mu_\phi=\mu_\infty$, an appeal to Theorem \ref{thm:TrueAsymptotic} guarantees that
\begin{equation*}
   \limsup_{n\to\infty} n^{\mu_\phi}\|\phi^{(n)}\|_\infty\leq \epsilon+\limsup_{n\to\infty} n^{\mu_\phi}H_P^n(0)=\epsilon+H_{P_\infty}^1(0)
\end{equation*}
for each $\epsilon>0$ and therefore
\begin{equation*}
    \limsup_{n\to\infty} n^{\mu_\phi}\|\phi^{(n)}\|_\infty\leq H_{P_\infty}^1(0).
\end{equation*}
By virtue of Theorem \ref{thm:TrueAsymptotic} and Theorem \ref{thm:LLT}, observe that
\begin{equation*}
    H_{P_\infty}^1(0)=\lim_{n\to\infty}n^{\mu_\phi}H_P^n(0)=\lim_{n\to\infty}n^{\mu_\phi}\abs{\widehat{\phi}(\xi_0)^nH_P^n(0)}=\lim_{n\to\infty} n^{\mu_\phi}\abs{\phi^{(n)}(0)}.
\end{equation*}
Since $\abs{\phi^{(n)}(0)}\leq\|\phi^{(n)}\|_\infty$ for all $n$, it follows that
\begin{equation*}
H_{P_\infty}^1(0)= \liminf_{n\to\infty}n^{\mu_\phi}\abs{\phi^{(n)}(0)}\leq\liminf_{n\to\infty} n^{\mu_\phi}\|\phi^{(n)}\|_\infty\leq\limsup_{n\to\infty} n^{\mu_\phi}\|\phi^{(n)}\|_\infty\leq H_{P_\infty}^1(0)
\end{equation*}
and, from this, \eqref{eq:ConvSupNormTrueAsymptotic} follows at once.
\end{proof}

As we did in the proofs of Theorem \ref{thm:OnDiagonal} and Theorem \ref{thm:Perturbation}, our proof of Theorem \ref{thm:LLT} makes use of the measure-preserving transformation $T(\eta,\zeta)=(\eta-Q(\zeta),\zeta)$ from $\mathbb{R}^{d}=\mathbb{R}^{{a}}\times\mathbb{R}^{{b}}$ to itself. It is easy to verify that $T:\mathbb{R}^d\to\mathbb{R}^d$ is a homeomorphism and $T(0)=0$. With this transformation, we define $\widetilde{P}=P\circ T$ and $\widetilde{R}=R\circ T$ and note that the domain of $\widetilde{R}$ coincides with the preimage of $\Gamma_{\xi_0}$'s domain under $T$ and necessarily contains an open neighborhood of $0$. Finally, for convenience of notation, we shall set $G=E_1\oplus F_2$ and recall that $\{t^G=t^{E_1}\oplus t^{F_2}\}$ is a contracting group. Before the proof of the theorem, we present two lemmas; the first follows immediately from Lemma \ref{lem:RTildeNice}.

\begin{lemma}\label{lem:RTildeExponentLimit}
For any compact set $K\subseteq\mathbb{R}^{d}$ and $\epsilon>0$, there is a natural number $N$ for which
\begin{equation*}
\abs{e^{n\widetilde{R}\left(n^{-G}\xi\right)}-1}<\epsilon
\end{equation*}
for all $n\geq N$ and $\xi\in K$.
\end{lemma}

The follows lemma asserts that the collection $\left\{\xi\mapsto\exp(-n\widetilde{P}(n^{-G}\xi)/2)\right\}_n$ is uniformly integrable on $\mathbb{R}^d$.

\begin{lemma}\label{lem:CompactSetFromIntegrability}
For any $\epsilon>0$, there exists a compact set $K$ for which
\begin{equation*}
\int_{\mathbb{R}^{d}\setminus K}e^{-n \widetilde{P}(n^{-G}\xi)/2}\,d\xi<\epsilon
\end{equation*}
for all $n\in\mathbb{N}_+$.
\end{lemma}
\begin{proof}
Fix $\epsilon>0$ and observe that, for $n\in\mathbb{N}_+$ and $\xi=(\eta,\zeta)\in\mathbb{R}^{d}$,
\begin{eqnarray*}
n\widetilde{P}(n^{-G}\xi)&=&nP_1(n^{-E_1}\eta)+nP_2(n^{-E_1}\eta-Q(n^{-F_2}\zeta))\\
&=&P_1(\eta)+P_2(n^{E_2-E_1}\eta-Q(\zeta))
\end{eqnarray*}
where we have used the fact that $Q$ is homogeneous with respect to the pair $(E_2,F_2)$. By virtue of Lemma \ref{lem:P2P1Estimates}, we can find positive constants $C$ and $M$ for which
\begin{equation*}
2C (P_1(\eta)+P_1(-Q(\zeta))-2M\leq P_1(\eta)+P_2(n^{E_2-E_1}\eta-Q(\zeta))
\end{equation*}
for all $(\eta,\zeta)\in\mathbb{R}^{d}$ and $n\in\mathbb{N}_+$. Consequently,
\begin{equation*}
-n\widetilde{P}(n^{-G}\xi)/2=-(1/2)(P_1(\eta)+P_2(n^{E_2-E_2}\eta-Q(\zeta))\leq M-C(P_1(\eta)+P_1(-Q(\zeta))
\end{equation*}
for all $\xi=(\eta,\zeta)\in\mathbb{R}^{d}$ and $n\in\mathbb{N}_+$. Upon noting that $\xi=(\eta,\zeta)\mapsto C(P_1(\eta)+P_1(-Q(\zeta))$ is a positive homogeneous function, an appeal to Lemma \ref{lem:ExpIntegrability} guarantees a compact set $K\subseteq\mathbb{R}^{d}$ for which
\begin{equation*}
\int_{\mathbb{R}^{d}\setminus K}e^{-n\widetilde{P}(n^{-G}\xi)/2}\,d\xi\leq \int_{\mathbb{R}^{d}\setminus K}e^{M-C(P_1(\eta)+P_1(-Q(\eta)))}\,d\xi<\epsilon
\end{equation*} 
for all $n\in\mathbb{N}_+$.
\end{proof}
\begin{proof}[Proof of Theorem \ref{thm:LLT}]
Fix $\epsilon>0$. Without loss of generality, we shall assume that $\xi_0$ lives on the interior of $\mathbb{T}^{d}$ for, otherwise, another representation of the $d$-dimensional torus can be used as the domain of integration in \eqref{eq:ConvolutionFTRepresentation} and the proof proceeds without change. Let us select a sufficiently small open neighborhood $\mathcal{O}_{\xi_0}\subseteq\mathbb{T}^{d}$ of $\xi_0$ for which the following properties hold:
\begin{enumerate}[label=P\arabic*]
\item\label{property:OpenNeighborhoodProperties1}The functions $\Gamma_{\xi_0}$ and $R$ are defined and smooth on the open neighborhood $\mathcal{O}:=\mathcal{O}_{\xi_0}-\xi_0$ of $0$.
\item\label{property:OpenNeighborhoodProperties2} The open neighborhood $\mathcal{U}:=T^{-1}(\mathcal{O})=T^{-1}(\mathcal{O}_{\xi_0}-\xi_0)$ of $0$ has
\begin{equation*}
\abs{\widetilde{R}(\xi)}\leq \widetilde{P}(\xi)/2
\end{equation*}
for all $\xi\in\mathcal{U}$.
\end{enumerate}
The fact that $\mathcal{O}_{\xi_0}$ can be chosen so that \ref{property:OpenNeighborhoodProperties1} holds is clear from the definitions of $\Gamma_{\xi_0}$ and $R$. The ability to choose $\mathcal{O}_{\xi_0}$ so that Property \ref{property:OpenNeighborhoodProperties2} also holds is a consequence of Item \ref{item:RTildeNice1} of Lemma \ref{lem:RTildeNice}. Using Lemma \ref{lem:CompactSetFromIntegrability}, let $K\subseteq\mathbb{R}^{d}$ be a compact set for which
\begin{equation*}
\int_{\mathbb{R}^{d}\setminus K}e^{-n\widetilde{P}(n^{-G}\xi)/2}\,d\xi<\epsilon/4
\end{equation*}
for all $n\in\mathbb{N}_+$.

With the sets $\mathcal{O}_{\xi_0}$, $\mathcal{O}$, $\mathcal{U}$, and $K$ in hand, we shall now go about selecting $N$. First, because $\{t^{G}\}$ is a contracting  group and $\mathcal{U}$ is an open neighborhood of $0$, there exists $N_1\in\mathbb{N}_+$ for which $n^{-G}(K)\subseteq\mathcal{U}$ (equivalently, $K\subseteq n^G(\mathcal{U})$) for all $n\geq N_1$. By an appeal to Lemma \ref{lem:RTildeExponentLimit}, let $N_2\in\mathbb{N}_+$ be such that
\begin{equation*}
\abs{e^{n\widetilde{R}(n^{-G}\xi)}-1}<\frac{\epsilon}{4(m(K)+1)}
\end{equation*}
for all $\xi\in K$ and $n\geq N_1$ where $m(K)$ is the $d$-dimensional Lebesgue measure of $K$. Finally, given that $\abs{\widehat{\phi}(\xi)}$ is maximized only at $\xi_0$,
\begin{equation*}
\rho:=\sup_{\xi\in \mathbb{T}^{d}\setminus\mathcal{O}_{\xi_0}}\abs{\widehat{\phi}(\xi)}<1
\end{equation*}
and so we may choose a natural number $N\geq \max\{N_1,N_2\}$ for which $n^{\mu_\phi}\rho^n<\epsilon/4$ for all $n\geq N$. 

By virtue of the identity \eqref{eq:ConvolutionFTRepresentation}, we compute
\begin{eqnarray}\label{eq:LLT1}\nonumber
\lefteqn{n^{\mu_\phi}\abs{\phi^{(n)}(x)-\widehat{\phi}(\xi_0)^ne^{-ix\cdot\xi_0}H_P^n(x-n\alpha)}}\\\nonumber
&=&\frac{n^{\mu_\phi}}{(2\pi)^{d}}\abs{\int_{\mathbb{T}^{d}}\widehat{\phi}(\xi)^ne^{-ix\cdot\xi}\,d\xi-\widehat{\phi}(\xi_0)^n e^{-ix\cdot\xi_0}\int_{\mathbb{R}^{d}}e^{-nP(\xi)}e^{-i(x-n\alpha)\cdot\xi}\,d\xi}\\\nonumber
&\leq & \frac{n^{\mu_\phi}}{(2\pi)^{d}}\int_{\mathbb{T}^{d}\setminus\mathcal{O}_{\xi_0} }\abs{\widehat{\phi}(\xi)}^n\,d\xi+\\\nonumber
& &\hspace{2cm} n^{\mu_\phi}\abs{\int_{\mathcal{O}_{\xi_0}}\widehat{\phi}(\xi)^n e^{-ix\cdot\xi}\,d\xi-\widehat{\phi}(\xi_0)^ne^{-ix\cdot\xi_0}\int_{\mathbb{R}^{d}}e^{-nP(\xi)}e^{-i(x-n\alpha)\cdot\xi}\,d\xi}\\\nonumber
&\leq & n^{\mu_\phi}\rho^n+n^{\mu_\phi}\abs{\int_{\mathcal{O}}\widehat{\phi}(\xi+\xi_0)^ne^{-ix\cdot(\xi+\xi_0)}\,d\xi-\widehat{\phi}(\xi_0)^ne^{-ix\cdot\xi_0}\int_{\mathbb{R}^{d}}e^{-nP(\xi)}e^{-i(x-n\alpha)\cdot\xi}\,d\xi}\\
&\leq &\frac{\epsilon}{4}+n^{\mu_\phi}\abs{\int_{\mathcal{O}}e^{n\Gamma_{\xi_0}(\xi)}e^{-ix\cdot\xi}\,d\xi-\int_{\mathbb{R}^{d}}e^{-nP(\xi)}e^{-i(x-n\alpha)\cdot\xi}\,d\xi}
\end{eqnarray}
for all $n\geq N$ and $x\in\mathbb{Z}^{d}$ where we have made of variables $\xi\mapsto \xi+\xi_0$ and made use of the fact that $\abs{\widehat{\phi}(\xi_0)^ne^{-ix\cdot\xi_0}}=1$. Let us now make the measure-preserving change of variables $\xi\mapsto T(\xi)$ to see that
\begin{eqnarray*}
\int_{\mathcal{O}}e^{n\Gamma_{\xi_0}(\xi)}e^{-ix\cdot\xi}\,d\xi&=&\int_{\mathcal{U}}e^{n(\Gamma_{\xi_0}\circ T)(\xi)}e^{-ix\cdot T(\xi)}\,d\xi\\
&=&\int_{n^{-G}(K)}e^{n(\Gamma_{\xi_0}\circ T)(\xi)}e^{-ix\cdot T(\xi)}\,d\xi+\int_{\mathcal{U}\setminus n^{-G}(K)}e^{n(\Gamma_{\xi_0}\circ T)(\xi)}e^{-ix\cdot T(\xi)}\,d\xi.
\end{eqnarray*}
for each $n\geq N\geq N_1$ and $x\in\mathbb{Z}^d$. Similarly,
\begin{equation*}
\int_{\mathbb{R}^{d}}e^{-nP(\xi)}e^{-i(x-n\alpha)\cdot\xi}\,d\xi=\int_{n^{-G}(K)}e^{-n\widetilde{P}(\xi)}e^{-i(x-n\alpha)\cdot T(\xi)}\,d\xi+\int_{\mathbb{R}^{d}\setminus n^{-G}(K)}e^{-n\widetilde{P}(\xi)}e^{-i(x-n\alpha)\cdot T(\xi)}\,d\xi
\end{equation*}
for all $n\in\mathbb{N}_+$ and $x\in\mathbb{Z}^{d}$. Consequently, for $x\in\mathbb{Z}^{d}$ and $n\geq N$,
\begin{eqnarray}\label{eq:LLT2}\nonumber
\lefteqn{\hspace{-1cm}\abs{\int_{\mathcal{O}}e^{n\Gamma_{\xi_0}(\xi)}e^{-ix\cdot\xi}\,d\xi-\int_{\mathbb{R}^{d}}e^{-nP(\xi)}e^{-i(x-n\alpha)\cdot\xi}\,d\xi}}\\\nonumber
\hspace{2cm}&\leq& \abs{\int_{n^{-G}(K)}e^{n(\Gamma_{\xi_0}\circ T)(\xi)}e^{-ix\cdot T(\xi)}\,d\xi-\int_{n^{-G}(K)}e^{-n\widetilde{P}(\xi)}e^{-i(x-n\alpha)\cdot T(\xi)}\,d\xi}\\\nonumber
&&\hspace{1cm}+\abs{\int_{\mathcal{U}\setminus n^{-G}(K)}e^{n(\Gamma_{\xi_0}\circ T)(\xi)}e^{-ix\cdot T(\xi)}\,d\xi}+\abs{\int_{\mathbb{R}^{d}\setminus n^{-G}(K)}e^{-n\widetilde{P}(\xi)}e^{-i(x-n\alpha)\cdot T(\xi)}\,d\xi}\\\nonumber
&\leq &\abs{\int_{n^{-G}(K)}\left(e^{n\widetilde{R}(\xi)}-1\right)e^{-n\widetilde{P}(\xi)}e^{-i(x-n\alpha)\cdot T(\xi)}\,d\xi}\\\nonumber
&&\hspace{1cm}+\int_{\mathcal{U}\setminus n^{-G}(K)}\abs{e^{n(\Gamma_{\xi_0}\circ T)(\xi)}e^{-ix\cdot T(\xi)}}\,d\xi+\int_{\mathbb{R}^{d}\setminus n^{-G}(K)}\abs{e^{-n\widetilde{P}(\xi)}e^{-i(x-n\alpha)\cdot T(\xi)}}\,d\xi\\\nonumber
&\leq & \int_{n^{-G}(K)}\abs{e^{-n\widetilde{R}(\xi)}-1}\,d\xi\\
&&\hspace{1cm}+\int_{\mathcal{U}\setminus n^{-G}(K)}e^{n\Re[(\Gamma_{\xi_0}\circ T)(\xi)]}\,d\xi+\int_{\mathbb{R}^{d}\setminus n^{-G}(K)}e^{-n\widetilde{P}(\xi)}\,d\xi
\end{eqnarray}
where we have used the fact that $(\Gamma_{\xi_0}\circ T)(\xi)=i\alpha\cdot T(\xi)-\widetilde{P}(\xi)+\widetilde{R}(\xi)$ and $\widetilde{P}$ is non-negative. Upon making the change of variables $\xi\mapsto n^{-G}\xi$ and recalling how $N_2$ was chosen, observe that
\begin{equation}\label{eq:LLT3}
\int_{n^{-G}(K)}\abs{e^{-n\widetilde{R}(\xi)}-1}\,d\xi=n^{-\tr G}\int_K\abs{e^{-n\widetilde{R}(n^{-G}\xi)}-1}\,d\xi\leq n^{-\mu_\phi}\frac{\epsilon}{4(m(K)+1)}m(K)\leq\frac{\epsilon n^{-\mu_\phi}}{4}
\end{equation}
whenever $n\geq N\geq N_2$; here, we have recalled that $\tr G=\tr (E_1\oplus F_2)=\tr E_1+\tr E_2=\mu_\phi$. Recalling our choice of $K$ and upon noting that
\begin{equation*}
\Re[(\Gamma_{\xi_0}\circ T)(\xi)]=-\widetilde{P}(\xi)+\Re[\widetilde{R}(\xi)]\leq -\widetilde{P}(\xi)+\widetilde{P}(\xi)/2=-\widetilde{P}(\xi)/2
\end{equation*}
for all $\xi\in \mathcal{U}$ thanks to \ref{property:OpenNeighborhoodProperties2}, we find that
\begin{eqnarray}\label{eq:LLT4}\nonumber
\int_{\mathcal{U}\setminus n^{-G}(K)}e^{n\Re[(\Gamma_{\xi_0}\circ T)(\xi)]}\,d\xi&\leq & \int_{\mathcal{U}\setminus n^{-G}(K)}e^{-n\widetilde{P}(\xi)/2}\,d\xi\\\nonumber
&\leq & \int_{\mathbb{R}^{d}\setminus n^{-G}(K)}e^{-n\widetilde{P}(\xi)/2}\,d\xi\\
&=&n^{-\mu_\phi}\int_{\mathbb{R}^{d}\setminus K}e^{-n\widetilde{P}(n^{-G}\xi)/2}\,d\xi<\frac{\epsilon n^{-\mu_\phi}}{4}
\end{eqnarray}
for all $n\geq N$ where we have made the change of variables $\xi\mapsto n^{-G}\xi$ between the final and penultimate lines. Finally, using the fact that $\widetilde{P}$ is non-negative, a similar computation shows that
\begin{equation}\label{eq:LLT5}
\int_{\mathbb{R}^{d}\setminus n^{-G}(K)}e^{-n\widetilde{P}(\xi)}\,d\xi\leq n^{-\mu_\phi} \int_{\mathbb{R}^{d}\setminus K}e^{-n\widetilde{P}(n^{-G}\xi)/2}\,d\xi<\frac{\epsilon n^{-\mu_\phi}}{4}
\end{equation}
for all $n\in\mathbb{N}_+$. Combining \eqref{eq:LLT2}, \eqref{eq:LLT3}, \eqref{eq:LLT4}, and \eqref{eq:LLT5} yields
\begin{equation*}
n^{\mu_\phi}\abs{\int_{\mathcal{O}}e^{n\Gamma_{\xi_0}(\xi)}e^{-ix\cdot\xi}\,d\xi-\int_{\mathbb{R}^{d}}e^{-nP(\xi)}e^{-i(x-n\alpha)\cdot\xi}\,d\xi}<\frac{3\epsilon}{4}
\end{equation*}
for all $n\geq N$ and $x\in\mathbb{Z}^{d}$. Finally, substituting the above estimate into \eqref{eq:LLT1} gives
\begin{equation*}
n^{\mu_\phi}\abs{\phi^{(n)}(x)-\widehat{\phi}(\xi_0)^n e^{-ix\cdot\xi}H_P^n(x-n\alpha)}<\epsilon
\end{equation*}
for all $n\geq N$ and $x\in\mathbb{Z}^{d}$, as was asserted.
\end{proof}

\begin{example}[The introductory Example]
For $\phi:\mathbb{Z}^2\to\mathbb{C}$ given by \eqref{eq:IntroPhi}, we have
\begin{eqnarray*}
\widehat{\phi}(\xi)&=&\frac{1}{100} \left(100-\left(\sin (\eta)+4 \sin ^2(\zeta/2)\right)^2-\frac{797}{600} \sin
   ^4(\eta)-10 \sin ^6(\eta/2)-\frac{1}{6}\sin ^6(\zeta)-\frac{179}{1200} \sin ^8(\zeta)\right.\\
&&\left.-\frac{1}{6} \sin (\eta) \sin^4(\zeta)-\frac{77}{900} \sin (\eta) \sin ^6(\zeta)-\frac{47}{150}\sin(\zeta)^2\sin(\eta)^3+\frac{3}{100} \sin
   ^2(\eta) \sin ^4(\zeta)\right)
\end{eqnarray*} 
for $\xi=(\eta,\zeta)\in\mathbb{R}^2$. A straightforward computation shows that $\max_{\xi}|\widehat{\phi}(\xi)|=1$ and, within $\mathbb{T}^2$, this maximum is attained only at the origin, $\xi_0=(0,0)$. To give a reader a sense of $\widehat{\phi}$, we have illustrated its absolute value in Figure \ref{fig:PhiHat}.
\begin{figure}[!htb]
    \centering
    \includegraphics[width=0.6\textwidth]{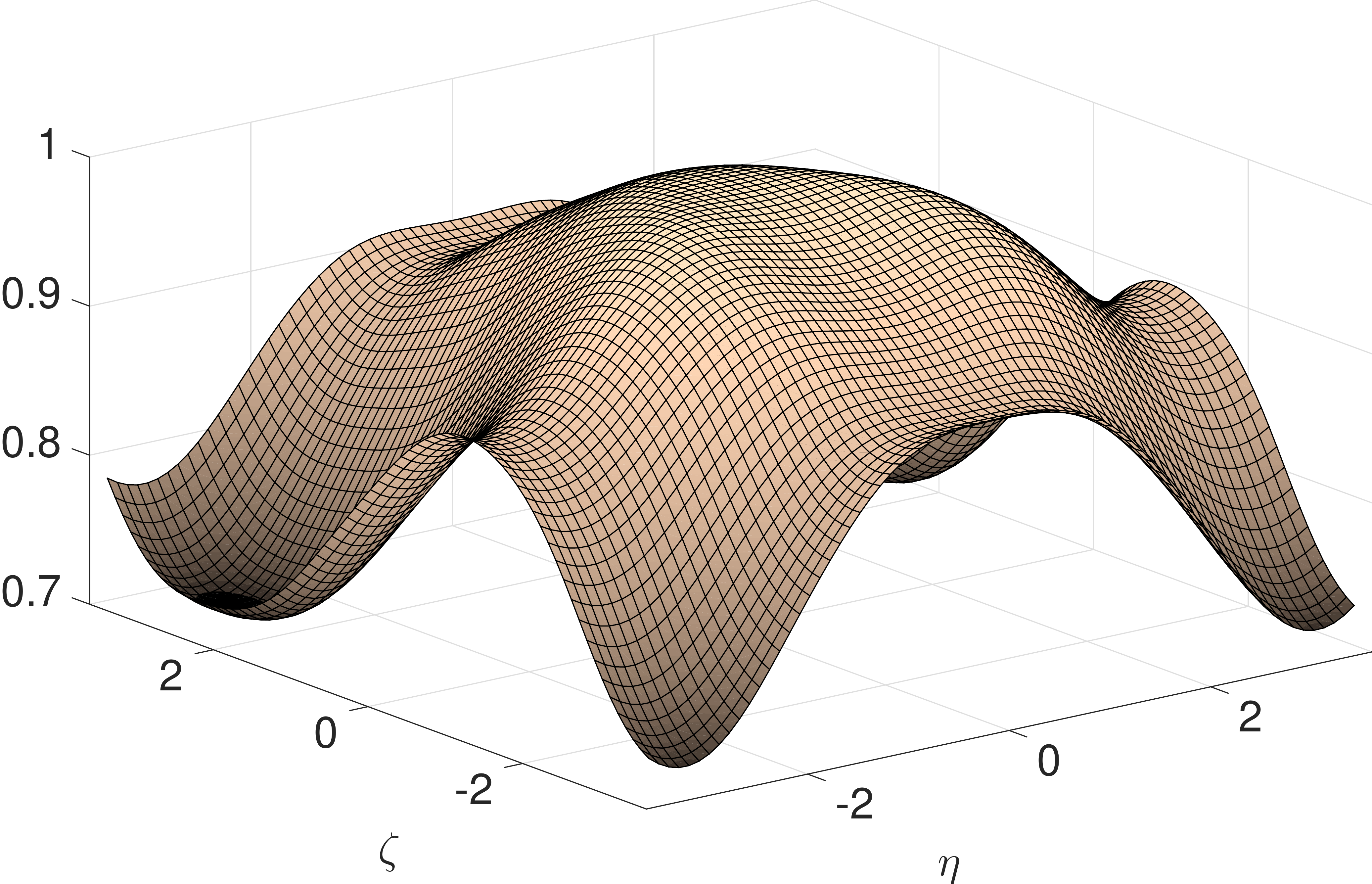}
    \caption{The graph of $|\widehat{\phi}(\xi)|$ for $\xi=(\eta,\zeta)\in\mathbb{T}^2$.}
    \label{fig:PhiHat}
\end{figure}
It is easy to see that $\widehat{\phi}(\xi_0)=1$ and we compute
\begin{equation*}
\Gamma_{\xi_0}(\xi)=\log(\widehat{\phi}(\xi))=-\frac{1}{100}P(\xi)+R(\xi)
\end{equation*}
where $P(\xi)=(\eta+\zeta^2)^2+\eta^4$ is the symbol discussed throughout the introduction and
\begin{equation*}
R(\xi)=-\frac{1}{1000}\left(\frac{1}{5}\eta^2\zeta^6-\frac{47}{45}\eta^3\zeta^4-\frac{79}{400}\zeta^4\zeta^4-\frac{1201}{1000}\eta^5\zeta^2-\frac{60739}{9000}\eta^6-\frac{37233979}{16800000}\eta^8\right)+O(\abs{\xi}^9)
\end{equation*}
as $\xi\to 0$. Observe that $\xi\mapsto \abs{\xi}^9=L_{4.5}(\xi)$ where $L_k$ is defined in Example \ref{ex:PerturbByLaplace} and corresponds to the $k$-th power of the Laplacian. As shown in that example (following Corollary \ref{cor:LaplaceCompareSemiElliptic}), $L_{4.5}(\xi)=o(P(\xi))$ as $\xi\to 0$ and so it follows that $O(\abs{\xi}^9)=o(P(\xi)/100)$ as $\xi\to 0$. Thus, to show that $R(\xi)=o(P(\xi)/100)$ as $\xi\to 0$ so that we may apply Theorem \ref{thm:LLT} and Corollary \ref{cor:ConvSupNorm}, we must show that the monomials $R_1(\xi):=\eta^2\zeta^6$, $R_2(\xi):=\eta^3\zeta^4$, $R_3(\xi):=\eta^4\zeta^4$, $R_4(\xi):=\eta^5\zeta^2$, $R_5(\xi):=\eta^6$, and $R_6(\xi)=\eta^8$ are ``little-o'' of $P(\xi)$ as $\xi\to 0$ and, to this end, our approach will  employ Proposition \ref{prop:RTildeVeryNice} just as we did in Example \ref{ex:R1R2Details}. Indeed, we have
\begin{equation*}
\widetilde{R_1}(\xi)=(\rho_1\circ T)(\eta,\zeta)=(\eta-\zeta^2)^2\zeta^6
\end{equation*}
for $\xi=(\eta,\zeta)\in\mathbb{R}^2$ where $T(\eta,\zeta)=(\eta-\zeta^2,\zeta)$. Therefore,
\begin{equation*}
\widetilde{R_1}(t^G\xi)=(t^{1/2}\eta-(t^{1/8}\zeta)^2)^2(t^{1/8}\zeta)^6=t^{5/4}(t^{1/4}\eta-\zeta^2)^2\zeta^6
\end{equation*}
for $t>0$ and $\xi=(\eta,\zeta)\in\mathbb{R}^2$ where $G=E_1\oplus F_2$ has standard matrix representation $\diag(1/2,1/8)$. From this it follows (by the same argument used in Example \ref{ex:R1R2Details}) that $\widetilde{R_1}$ is subhomogeneous with respect to $G$ and therefore $R_1(\xi)=o(P(\xi))=o(P(\xi)/100)$ as $\xi\to 0$ on account of Proposition \ref{prop:RTildeVeryNice}. We leave it to the reader the verify this conclusion for $R_k$ for $k=2,3,\dots,6$. Consequently, $R(\xi)=o(P(\xi)/100)$ as $\xi\to 0$.

Since $\xi\mapsto P(\xi)/100$ meets the hypotheses of Theorem \ref{thm:LLT} (as well as Theorem \ref{thm:OnDiagonal}) with $\mu_\infty=5/8$ and $R(\xi)=o(P(\xi)/100)$ as $\xi\to 0$, an appeal to Theorem \ref{thm:LLT} guarantees that
\begin{equation*}
\phi^{(n)}(x)=\widehat{\phi}(\xi_0)^ne^{-ix\cdot\xi_0}H_{P/100}^n(x-n\alpha)+o(n^{-5/8})=H_{P}^{n/100}(x)+o(n^{-5/8})
\end{equation*}
uniformly for $x\in\mathbb{Z}^2$ as $n\to\infty$; here, we have noted that $\mu_\phi=\mu_\infty=5/8$, $\xi_0=(0,0)$, $\widehat{\phi}(\xi_0)=1$, and $\alpha=(0,0)$. We recall that this conclusion was discussed in the introduction and illustrated in Figure \ref{fig:Intro}. Finally, we appeal to Corollary \ref{cor:ConvSupNorm} to conclude that $\|\phi^{(n)}\|_\infty\asymp n^{-5/8}$ as $n\geq 1$ and
\begin{equation*}
    \lim_{n\to\infty}n^{5/8}\|\phi^{(n)}\|_\infty=H_{(P/100)_\infty}^1(0).
\end{equation*}
Of course, $(P/100)_\infty=P_\infty/100$ where $P_\infty(\xi)=\eta^2+\zeta^8$ and therefore
\begin{equation*}
    H_{(P/100)_\infty}^1(0)=H_{P_\infty}^{1/100}(0)=(100)^{5/8}H_{P_\infty}(0)
\end{equation*}
where we have used the homogeneity of $P_\infty$ with respect to $G$. In view of the calculations done in Example \ref{ex:OnDiagonalMotivatingExampleTrueAsymptotics}, we conclude that
\begin{equation*}
\lim_{n\to\infty}n^{5/8}\|\phi^{(n)}\|_\infty=(100)^{5/8}H_{P_\infty}^1(0)=\frac{(100)^{5/8}}{2\pi^{3/2}}\Gamma(9/8)\approx 1.50376.
\end{equation*}
This is illustrated in Figure \ref{fig:ExampleLogLog} wherein $\|\phi^{(n)}\|_\infty$ is compared with $1.5 \times n^{-5/8}$ on a log-log scale for $10^1\leq n\leq10^6$. 
\begin{figure}[!htb]
    \centering
    \includegraphics[width=0.6\textwidth]{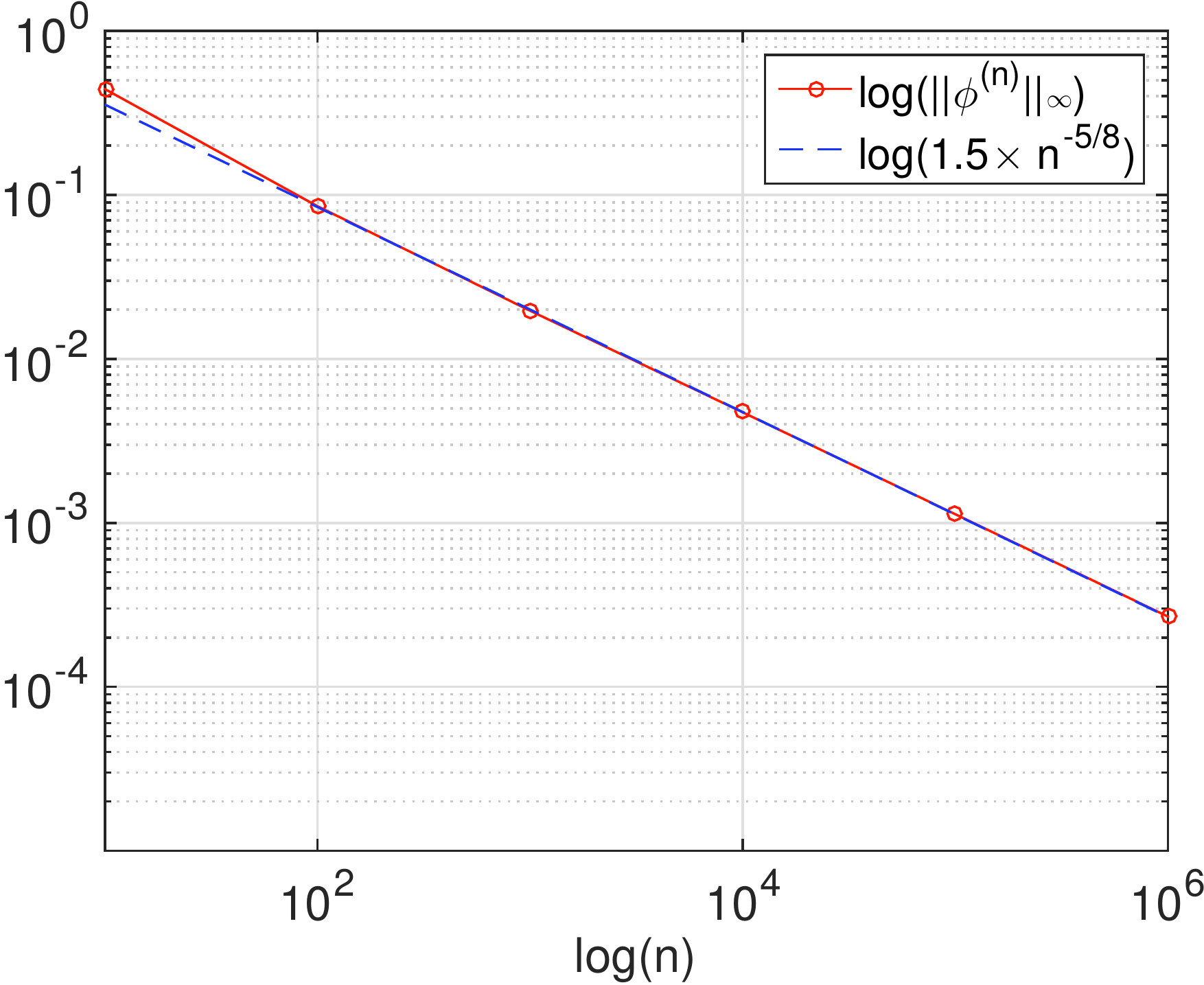}
    \caption{A log-log plot of $\|\phi^{(n)}\|_\infty$ and $1.5\times n^{-5/8}$ for $10^1\leq n\leq 10^6$.}
    \label{fig:ExampleLogLog}
\end{figure}
\end{example}

\begin{example}[A nearby non-example]
In this example, we consider a complex-valued function $\psi$ on $\mathbb{Z}^2$ which is a close approximation to the function $\phi$ of the preceding example, however, its convolution powers $\psi^{(n)}$ behave distinctly from those of $\phi$. In particular, this example shows that the hypotheses concerning $R(\xi)$ in the expansion \eqref{eq:GammaExpansion} of Theorem \ref{thm:LLT} are necessary and, further, it demonstrates the delicate nature of ``higher order'' terms present in these expansions insofar as they affect the behavior of convolution powers. In what follows, we shall assume the notation of the preceding example and consider $\psi:\mathbb{Z}^d\to\mathbb{C}$ defined by
\begin{equation*}
\psi=\frac{1}{960000}\left(\psi_1+\psi_2+\psi_3\right)
\end{equation*}
where
\begin{equation*}
\psi_1(x)=
\begin{cases}
862318 & (x_1,x_2)=(0,0)\\
22500\pm 19200 i & (x_1,x_2)=\pm(1,0)\\
    -3412&(x_1,x_2)=\pm(2,0)\\
    1500 & (x_1,x_2)=\pm(3,0)\\
    -797 &(x_1,x_2)=\pm (4,0)\\
    0 &\mbox{else}
\end{cases},
\hspace{1cm}
\psi_2(x)=
\begin{cases}
38400 & (x_1,x_2)=(0,\pm 1)\\
    -9225 & (x_1,x_2)=(0,\pm 2)\\
    -150 & (x_1,x_2)=(0,\pm 4)\\
    25 &(x_1,x_2)=(0,\pm 6)\\
    0 &\mbox{else}
\end{cases}
\end{equation*}
and
\begin{equation*}
    \psi_3(x)=
    \begin{cases}
    \pm9600 i & (x_1,x_2)=(\mp 1,1), (\mp 1,-1)\\
    0 &\mbox{else}
    \end{cases}
\end{equation*}
for $x=(x_1,x_2)\in\mathbb{Z}^2$. We remark that $\psi$ is close to $\phi$ in several senses. For example, we have $\|\phi-\psi\|_\infty=51/20935\approx 0.0024\approx 0.0024 \times \|\phi\|_\infty$ and $\|\phi-\psi\|_2=1/\sqrt{247293}\approx 0.0020\approx 0.0022\times\|\phi\|_2$. Of course, given that the behavior of convolution powers $\psi^{(n)}$ is determined by the nature of the Fourier transform of $\psi$ near points at which is it maximized in absolute value (c.f., \cite{RSC15,RSC17,BR22,Ra22}), we analyze $\widehat{\psi}$. We have
\begin{equation*}
\widehat{\psi}(\xi)=\frac{1}{100} \left(100-\left(\sin (\eta)+4 \sin ^2(\zeta/2)\right)^2-\frac{1}{6}\sin ^6(\zeta)\right)
\end{equation*}
for $\xi=(\eta,\zeta)\in\mathbb{R}^2$. As illustrated in Figure \ref{fig:FTCompare}, the graph of $|\widehat{\psi}(\xi)|$ is extremely similar to that of $|\widehat{\phi}(\xi)|$ on $\mathbb{T}^2$.
\begin{figure}[!htb]
    \begin{subfigure}{0.49\textwidth}
    \centering
    \includegraphics[width=0.8\textwidth]{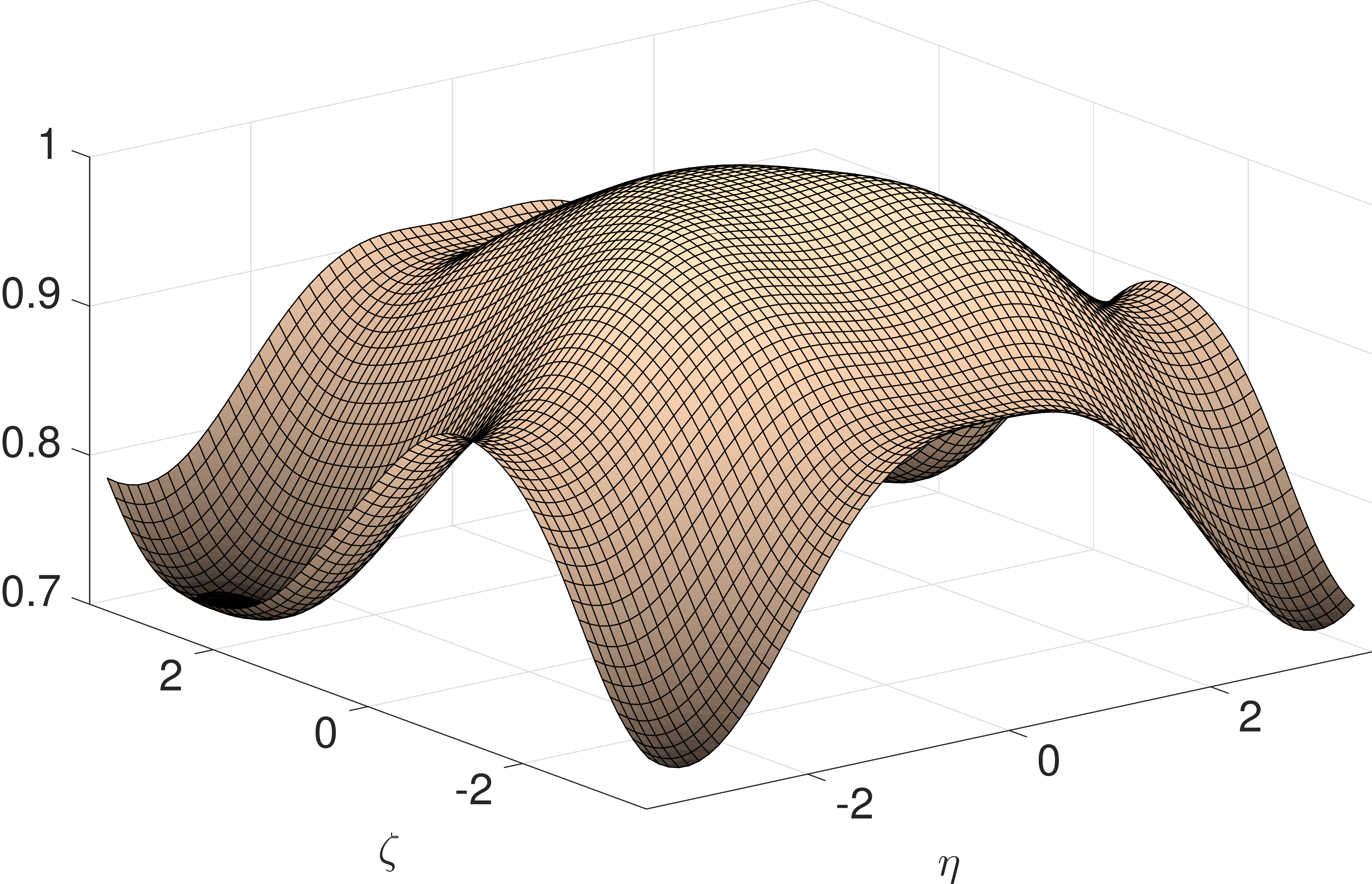}
    \caption{$|\widehat{\psi}(\xi)|$ for $\xi=(\eta,\zeta)\in\mathbb{T}^2$.}
    \end{subfigure}
    \begin{subfigure}{0.49\textwidth}
    \centering
    \includegraphics[width=0.8\textwidth]{PhiHat.pdf}
    \caption{$|\widehat{\phi}(\xi)|$ for $\xi=(\eta,\zeta)\in\mathbb{T}^2$.}
    \end{subfigure}
    \caption{Illustration of the Fourier transforms of $\psi$ and $\phi$ on $\mathbb{T}^2$.}
    \label{fig:FTCompare}
\end{figure}
Akin to the previous example, $\max_\xi|\widehat{\psi}(\xi)|=1$ and this maximum is attained only at $\xi_0=(0,0)$ in $\mathbb{T}^2$ where we have $\widehat{\psi}(\xi_0)=1$. We compute
\begin{equation*}
\Gamma_{\xi_0}(\xi)=\log(\widehat{\psi}(\xi))=-\frac{1}{100}P(\xi)+R(\xi)
\end{equation*}
where $P(\xi)=(\eta+\zeta^2)^2+\eta^4$ and
\begin{eqnarray}\label{eq:RBadDecomposition}\nonumber
R(\xi)&=&\frac{1}{100}\left(\frac{179}{1200}\zeta^8+\frac{1}{6}\eta \zeta^4-\frac{23}{900}\eta\zeta^6-\frac{3}{100}\eta^2\zeta^4+\frac{47}{150}\eta^3\zeta^2\right)\\\nonumber
&&+\frac{1}{10000}\left(\frac{1}{3}\eta^2\zeta^6-\frac{47}{18}\eta^3\zeta^4-\frac{227}{600}\eta^4\zeta^4-\frac{1003}{300}\eta^5\zeta^2+\frac{60739}{900}\eta^6-\frac{37233979}{1680000}\eta^8+O(\abs{\xi}^9)\right)\\
&=:&R_1(\xi)+R_2(\xi)
\end{eqnarray}
as $\xi=(\eta,\zeta)\to 0$. In other words, akin to the analogous expansion for $\widehat{\phi}$, $\Gamma_{\xi_0}$ is made up of the polynomial $-P(\xi)/100$ and a series $R(\xi)$ consistent of ``higher order terms'' relative to those of $P$. With the initial aim of applying Theorem \ref{thm:LLT} or Corollary \ref{cor:ConvSupNorm} to this example, we ask if $R(\xi)=o(P(\xi))$ as $\xi\to 0$. As we did in the preceding example, we investigate this by composing by the non-linear transformation $T$ and then checking if $R_1(\xi)$ and $R_2(\xi)$ are subhomogeneous with respect to $G=E_1\oplus F_2$. We find that
\begin{equation*}
\widetilde{R}_1(\xi)=\frac{1}{100}\left(\frac{179}{1200}\zeta^8+\frac{1}{6}(\eta-\zeta^2)\\
\zeta^4-\frac{23}{900}(\eta-\zeta^2)\zeta^6-\frac{3}{100}(\eta-\zeta^2)^2\zeta^4+\frac{47}{150}(\eta-\zeta^2)^3\zeta^2\right)
\end{equation*}
for $\xi=(\eta,\zeta)$. From this it follows that
\begin{equation}\label{eq:R1TildeBad}
\lim_{t\to 0}t^{-1}\widetilde{R}_1(t^G\xi)=\lim_{t\to 0} \frac{1}{6}  (\eta-t^{-1/4}\zeta^2)\zeta^4=-\infty
\end{equation}
whenever $\xi=(\eta,\zeta)$ is such that $\zeta\neq 0$.  Consequently, $\widetilde{R}_1$ is not subhomogeneous with respect to $G$. By contrast, we find that $\widetilde{R}_2$ is subhomogeneous with respect to $G$ by an argument analogous to that used in Example \ref{ex:R1R2Details}. In view of Proposition \ref{prop:RTildeVeryNice}, we conclude that $R(\xi)\neq o(P(\xi))$ as $\xi\to 0$. We note that, though our calculation above for $\widetilde{R}_1$ shows that $R_1$ is not well controlled by $P(\xi)$ for small $\xi$, this loss of control only happens from below and from this we will still be able to deduce something useful (see Lemma \ref{lem:PsiNotAymptotic} and Proposition \ref{prop:PsiNotAsymptitic} below). 

Though the expansion of $\Gamma_{\xi_0}$ for $\widehat{\psi}$ agrees with that for $\widehat{\phi}$ at the lowest order (both are $-P(\xi)/100$), we cannot apply Theorem \ref{thm:LLT} or Corollary \ref{cor:ConvSupNorm} to this example because $R(\xi)\neq o(P(\xi))$ as $\xi\to 0$. In particular, for $\|\psi^{(n)}\|_\infty$, we are not able to deduce from Corollary \ref{cor:ConvSupNorm} the decay of $n^{-5/8}$ which is characteristic of $H_{P/100}^n$ and $\|\phi^{(n)}\|_\infty$. Figure \ref{fig:NonExampleLogLog} presents strong numerical evidence that $\|\psi^{(n)}\|_\infty$ decays faster than $n^{-5/8}$ as $n\to\infty$.
\begin{figure}[!htb]
    \centering
    \includegraphics[width=0.6\textwidth]{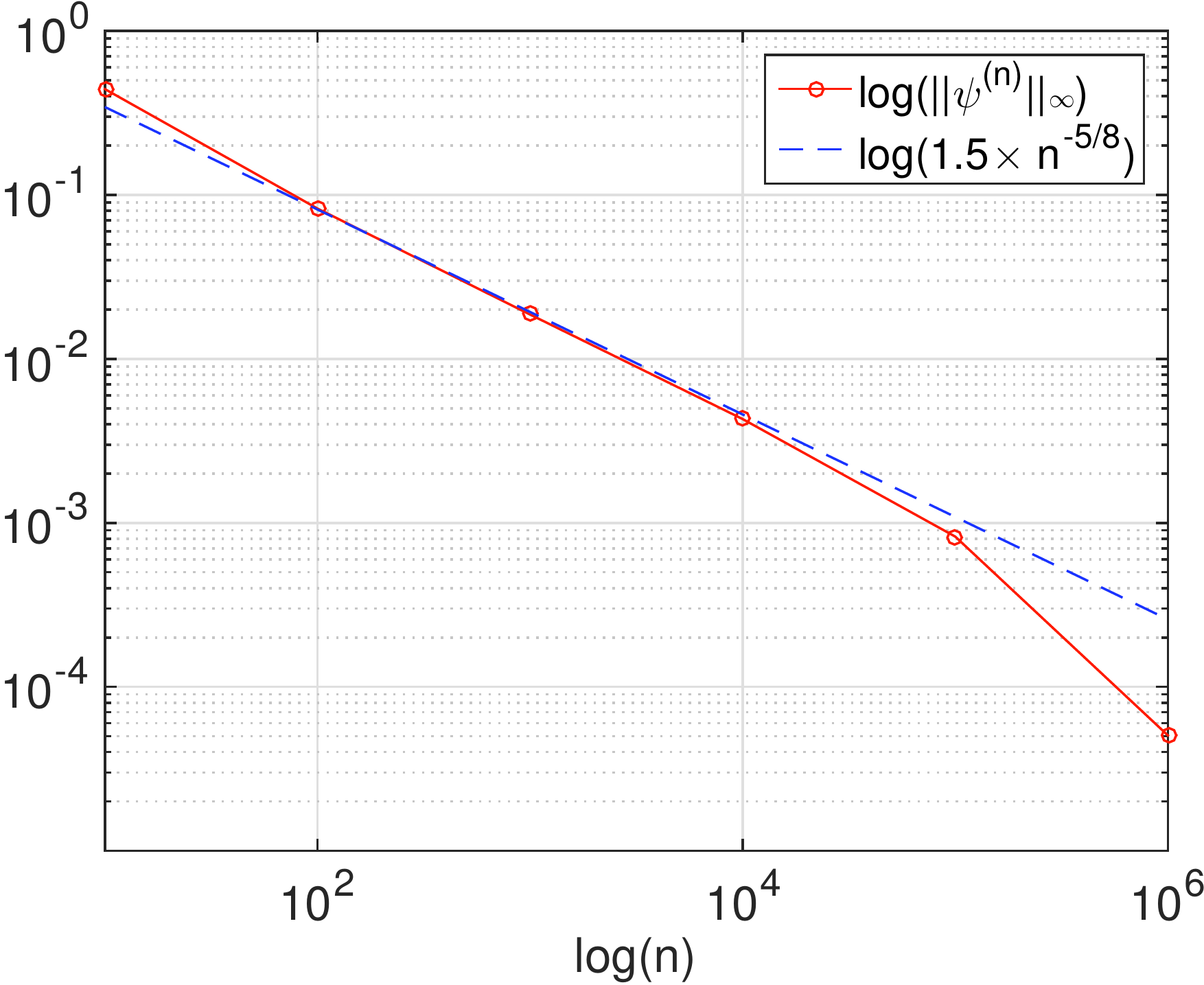}
    \caption{A log-log plot of $\|\psi^{(n)}\|_\infty$ and $1.5\times n^{-5/8}$ for $10^1\leq n\leq 10^6$.}
    \label{fig:NonExampleLogLog}
\end{figure}
In fact, the following proposition confirms that this is the case.

\begin{proposition}\label{prop:PsiNotAsymptitic}
We have
\begin{equation*}
\lim_{n\to\infty}n^{5/8}\|\psi^{(n)}\|_\infty=0.
\end{equation*}
\end{proposition}
\begin{lemma}\label{lem:PsiNotAymptotic}
There exists $\delta\in (0,1)$ and an open neighborhood $\mathcal{O}\subseteq\mathbb{T}^2$ of $0$ for which
\begin{equation*}
R(\xi)\leq \frac{1-\delta}{100}P(\xi)
\end{equation*}
for $\xi\in\mathcal{O}$.
\end{lemma}
\begin{exproof}
In \eqref{eq:RBadDecomposition}, we have written $R(\xi)=R_1(\xi)+R_2(\xi)$ and noted, on account of Proposition \ref{prop:RTildeVeryNice}, that $R_2(\xi)=o(P(\xi)/100)$ as $\xi\to 0$. Thus, to prove the statement at hand, it suffices to find an open neighborhood $\mathcal{O}\subseteq\mathbb{T}^2$ of $0$ and $\rho\in (0,1)$ for which
\begin{equation*}
\mathcal{E}(\xi):=\frac{179}{1200}\zeta^8+\frac{1}{6}\eta\zeta^4-\frac{23}{900}\eta\zeta^6-\frac{3}{100}\eta^2\zeta^4+\frac{47}{150}\eta^3\zeta^2=100R_1(\xi)\leq \rho P(\xi)
\end{equation*}
for all $\xi=(\eta,\zeta)\in\mathcal{O}$. As a first approximation, we set $\mathcal{O}=\{\xi=(\eta,\zeta):P(\xi)<1\}$ and decompose it as $\mathcal{O}=\mathcal{R}_1\cup\mathcal{R}_2$ where $\mathcal{R}_1$ and $\mathcal{R}_2$ are given by \eqref{eq:Region1} and \eqref{eq:Region2}, respectively. In view of our analysis in Example \ref{ex:PerturbByLaplace}, we observe that, for $\xi=(\eta,\zeta)\in\mathcal{R}_1$, $0\leq\zeta^2\leq -\sqrt{2}\eta$ and, in particular, $\eta$ is non-positive. Therefore,
\begin{equation*}
\mathcal{E}(\xi)\leq\frac{179}{1200}4\eta^4+0+\frac{23}{900}(\sqrt{2})^3\eta^4+0+0=\left(\frac{179}{300}+\frac{23\sqrt{2}}{450}\right)\eta^4\leq \frac{7}{9}P(\xi)
\end{equation*}
for $\xi=(\eta,\zeta)\in\mathcal{R}_1$. On $\mathcal{R}_2$, we have
\begin{equation*}
\max\left\{\eta^4,\left(1-\frac{1}{\sqrt{2}}\right)^2\zeta^4\right\}\leq P(\xi)
\end{equation*}
and therefore
\begin{eqnarray*}
\lefteqn{\mathcal{E}(\xi)\leq \frac{179}{1200}(6+4\sqrt{2})^2P(\xi)^2+\frac{1}{6}(6+4\sqrt{2})P(\xi)^{5/4}}\\
&&\hspace{2cm}+\frac{23}{900}(6+4\sqrt{2})^{3/2}P(\xi)^{7/4}+0+\frac{47}{150}(2+\sqrt{2})P(\xi)^{5/4}\leq 150 P(\xi)^{5/4}
\end{eqnarray*}
where we have used the fact that $P(\xi)<1$ for $\xi\in\mathcal{O}$. Thus, by further restricting the open set $\mathcal{O}$ so that $150 P(\xi)^{1/4}<7/9$, the preceding estimates ensure that $\mathcal{E}(\xi)\leq \rho P(\xi)$ for all $\xi\in\mathcal{O}$ where $\rho=7/9$.
\end{exproof}
\begin{exproof}[Proof of Proposition \ref{prop:PsiNotAsymptitic}.]
By an appeal to the preceding lemma, let $\mathcal{O}\subseteq\mathbb{T}^2$ be an open neighborhood of $0$ for which
\begin{equation}\label{eq:PsiNotAsymptotic1}
\abs{\widehat{\psi}(\xi)}=e^{-P(\xi)/100+R(\xi)}\leq e^{-\gamma P(\xi)}
\end{equation}
for $\xi\in\mathcal{O}$ where $\gamma=\delta/100>0$. Using similar arguments to those which appear in the proof of Theorem \ref{thm:LLT}, we find that
\begin{equation*}
\|\varphi^{(n)}\|_\infty=\sup_{x\in\mathbb{Z}^2}\abs{\psi^{(n)}(x)}\leq \rho^n+\int_{\mathcal{O}}\abs{\widehat{\psi}(\xi)}^n\,d\xi
\end{equation*}
for all $n\in\mathbb{N}_+$ where $0<\rho<1$. Focusing on the integral above, we make the change of variables $\xi\mapsto T(n^{-G}\xi)$ to see that
\begin{equation*}
n^{5/8}\int_{\mathcal{O}}\abs{\widehat{\psi}(\xi)}^n\,d\xi=\int_{n^{G}(\mathcal{U})}\exp\left(-n\left(\frac{\widetilde{P}(n^{-G}\xi)}{100}-\widetilde{R}(n^{-G}\xi)\right)\right)\,d\xi
\end{equation*}
for each $n\in\mathbb{N}_+$ where $\mathcal{U}=T^{-1}(\mathcal{O})$. Upon noting that $T(n^{-G}\xi)\in\mathcal{O}$ whenever $\xi\in n^G(\mathcal{U})$, it follows from Lemma \ref{lem:P2P1Estimates} and \eqref{eq:PsiNotAsymptotic1} that
\begin{equation*}
\exp\left(-n\left(\frac{\widetilde{P}(n^{-G}\xi)}{100}-\widetilde{R}(n^{-G}\xi)\right)\right)\chi_{_{n^G(\mathcal{U})}}(\xi)\leq e^{-n\gamma\widetilde{P}(n^{-G}\xi)}\chi_{_{n^G(\mathcal{U})}}(\xi)\leq e^Me^{-C( \eta^2+\zeta^4)}
\end{equation*}
for all $n\in\mathbb{N}_+$ and $\xi=(\eta,\zeta)\in\mathbb{R}^2$; here $M$ and $C$ are positive constants and, for a measurable set $\mathcal{A}\subseteq\mathbb{R}^2$, $\chi_{_\mathcal{A}}$ denotes the indicator function of $\mathcal{A}$. Also, thanks to the computation \eqref{eq:R1TildeBad} and the fact that $\widetilde{R}_2(\xi)=o(\widetilde{P}(\xi))$ as $\xi\to 0$, we have
\begin{equation*}
\lim_{n\to\infty}\exp\left(-n\left(\frac{\widetilde{P}(n^{-G}\xi)}{100}-\widetilde{R}(n^{-G}\xi)\right)\right)\chi_{_{n^G(\mathcal{U})}}(\xi)=0,
\end{equation*}
for almost every $\xi\in\mathbb{R}^d$. Given that $\xi\mapsto e^Me^{-C(\eta^2+\zeta^4)}\in L^1(\mathbb{R}^2)$, an appeal to the dominated convergence theorem is justified and we conclude that
\begin{eqnarray*}
\lim_{n\to\infty}n^{5/8}\|\psi^{(n)}\|_\infty&=&\lim_{n\to\infty}n^{5/8}\rho^n+\lim_{n\to\infty}n^{-5/8}\int_{\mathcal{O}}\abs{\widehat{\psi}(\xi)}^n\,d\xi\\
&=&0+\lim_{n\to\infty}\int_{n^{G}(\mathcal{U})}\exp\left(-n\left(\frac{\widetilde{P}(n^{-G}\xi)}{100}-\widetilde{R}(n^{-G}\xi)\right)\right)\,d\xi\\
&=&0.
\end{eqnarray*}
\end{exproof}
To our knowledge, there is no known theory that is able to treat this example and, in particular, the asymptotic behavior of $\|\psi^{(n)}\|_\infty$ remains unknown.
\end{example}
\section{Discussion}\label{sec:Discussion}

\subsection{Another perspective} 
As we noted in the introduction and demonstrated in the proof of Theorem \ref{thm:OnDiagonal}, key to the analysis in this paper is the observation that
\begin{equation*}
\varphi(t)=H_P^t(0)=H_{\widetilde{P}}^t(0)=\widetilde{\varphi}(t)
\end{equation*}
for all $t>0$ where $\widetilde{P}=P\circ T$ can be seen, in some sense, as a ``dual'' symbol to $P$. In terms of our motivating example in which $P(\xi)=(\eta+\zeta^2)^2+\eta^4$ and
\begin{equation*}
\Lambda=\partial_{x_1}^4+\partial_{x_2}^4+2i\partial_{x_1}\partial_{x_2}^2-\partial_{x_1}^2,
\end{equation*}
$\widetilde{P}(\xi)=\eta^2+(\eta-\zeta^2)^4$ for $\xi=(\eta,\zeta)\in\mathbb{R}^2$ and corresponds to the constant coefficient operator
\begin{equation*}
\widetilde{\Lambda}=-\partial_{x_1}^2+\partial_{x_1}^4+4i\partial_{x_1}^3\partial_{x_2}^2-6\partial_{x_1}^2\partial_{x_2}^4-4i\partial_{x_1}\partial_{x_2}^6+\partial_{x_2}^8.
\end{equation*}
Though the operator $\Lambda$ is a self-adjoint fourth-order elliptic operator and is therefore well understood, $\widetilde{\Lambda}$ is mysterious and somewhat poorly behaved as the following proposition shows.
\begin{proposition}
The operator $\widetilde{\Lambda}$ with symbol $\widetilde{P}=\eta^2+(\eta-\zeta^2)^4$ is not hypoelliptic. Still, $-\widetilde{\Lambda}$ generates a semigroup $\{e^{-t\widetilde{\Lambda}}\}_{t>0}$ with heat kernel given by
\begin{equation*}
H_{\widetilde{P}}^t(x)=\frac{1}{(2\pi)^2}\int_{\mathbb{R}^2}e^{-t\widetilde{P}(\xi)}e^{-ix\cdot\xi}\,d\xi
\end{equation*}
defined for $t>0$ and $x\in\mathbb{R}^2$. This heat kernel has
\begin{equation*}
\widetilde{\varphi}(t)=H_{\widetilde{P}}^t(0)\asymp 
\begin{cases}
 t^{-1/2} & 0<t\leq 1 \\
t^{-5/8} & t\geq 1
\end{cases}
\end{equation*}
for $t>0$ and, further,
\begin{equation*}
\lim_{t\to 0}t^{1/2}\widetilde{\varphi}(t)=\frac{1}{\pi^2}\Gamma(5/4)^2\hspace{1cm}\mbox{and}\hspace{1cm}\lim_{t\to\infty}t^{5/8}\widetilde{\varphi}(t)=\frac{1}{2\pi^{3/2}}\Gamma(9/8).
\end{equation*}
\end{proposition}
\begin{proof}
The statements concerning $H_{\widetilde{P}}$ and $\widetilde{\varphi}$ follow directly from our analysis in Examples \ref{ex:OnDiagonalMotivatingExample} and \ref{ex:OnDiagonalMotivatingExampleTrueAsymptotics} through the identity $\varphi=\widetilde{\varphi}$. Thus, it remains to prove that $\widetilde{\Lambda}$ is not hypoelliptic. To see this, we compute
\begin{equation*}
\partial^4_{\zeta}\widetilde{P}(\eta,\zeta)=384\zeta^4-1152\zeta^2(\eta-\zeta^2)+144(\zeta-\zeta^2)^2
\end{equation*}
to see that
\begin{equation*}
\lim_{\zeta\to\infty}\frac{\partial^4_{\zeta}\widetilde{P}(\zeta^2,\zeta)}{\widetilde{P}(\zeta^2,\zeta)}=384.
\end{equation*}
Thus, for the multi-index $\alpha=(0,4)$, 
\begin{equation*}
\lim_{|\xi|\to\infty}\frac{\widetilde{P}^{(\alpha)}(\xi)}{\widetilde{P}(\xi)}\neq 0;
\end{equation*}
here, for a multi-index $\alpha=(\alpha_1,\alpha_2)$, $\widetilde{P}^{(\alpha)}=\partial_{\eta}^{\alpha_1}\partial_{\zeta}^{\alpha_2}\widetilde{P}$ in the notation of L. H\"{o}rmander \cite{Ho83}. By virtue of Theorem 11.1.1 of \cite{Ho83}, we conclude that $\widetilde{\Lambda}$ is not hypoelliptic. 
\end{proof}

Of course, as we discussed in the introduction, the theory formulated in this article has a version in which one begins with the operator $\widetilde{\Lambda}$ whose heat kernel $H_{\widetilde{P}}$, along the diagonal, is well behaved in large time but whose small time behavior is elusive and is deduced more easily from $H_{P}$ via the correspondence $H_{P}^t(0)=H_{\widetilde{P}}^t(0)$. 

\subsection{Future Directions}

This article has focused on on-diagonal large-time asymptotics for the heat kernels $H_P$ of certain inhomogeneous operators $\Lambda$ and symbols $P$ and well-behaved associated perturbations. At present, we do not have a good understanding of the off-diagonal behavior of these kernels in large time. Knowledge of this behavior could greatly improve our understanding of the theory presented in this article and it would also inform our study of convolution powers. In particular, having a good handle on this off-diagonal behavior would inform on the stability of convolution powers (seen as numerical difference schemes) as it has in \cite{Th65}, \cite{CF22} (see also \cite{Th69} and Section 6 of \cite{RSC17}).  

The symbols $P(\xi)$ treated throughout this article all take the special form $P_1(\eta+Q(\zeta))+P_2(\eta)$ resemblant of our motivating example in which $P(\eta,\zeta)=(\eta+\zeta^2)^2+\eta^4$. Our methods appear to be more broadly applicable, however. Consider, for example, the symbol 
\begin{equation*}
P'(\xi)=(\eta+\zeta^2)+\eta^2\zeta^2
\end{equation*}
defined for $\xi=(\eta,\zeta)\in\mathbb{R}^2$. Akin to our motivating example, $P'$ lacks a tractable scaling from which large-time asymptotics for $\varphi'(t)=H_{P'}^t(0)$ may be computed directly.  Still, by composing $P'$ with the measure-preserving transformation $T(\eta,\zeta)=(\eta-\zeta^2,\zeta)$, we find that
\begin{equation*}
\lim_{t\to\infty}t\widetilde{P'}(t^{-1/2}\eta,t^{-1/6}\zeta)=\eta^2+\zeta^6
\end{equation*}
for $(\eta,\zeta)\in\mathbb{R}^2$ where $\widetilde{P'}=P'\circ T$. Accompanied by the fact that $t\widetilde{P'}(t^{-1/2}\eta,t^{-1/6}\zeta)\geq (\eta^2+\zeta^4-2)/8$ for $t\geq 1$ and $(\eta,\zeta)\in\mathbb{R}^2$, the change of variables $(\eta,\zeta)\mapsto (t^{-1/2}\eta,t^{-1/6}\zeta)$ followed by an application of the dominated convergence theorem shows that
\begin{equation*}
\lim_{t\to\infty}t^{2/3}\varphi'(t)=\frac{1}{(2\pi)^2}\int_{\mathbb{R}^2}e^{-\eta^2-\zeta^6}\,d\eta\,d\zeta=\frac{1}{2\pi^{3/2}}\Gamma(7/6).
\end{equation*} 
Another class of example can be produced easily by considering powers of symbols $P(\xi)=P_1(\eta+Q(\zeta))+P_2(\eta)$ on $\mathbb{R}^d$ satisfying the hypotheses of Theorem \ref{thm:OnDiagonal}. Specifically, given such a symbol $P(\xi)$ and $\kappa>0$, our methods show that the heat kernel $H_{P^{\kappa}}$ associated to $P(\xi)^{\kappa}$ satisfies the on-diagonal asymptotics
\begin{equation*}
H_{P^\kappa}^t(0)\asymp
\begin{cases}
t^{-\mu_0/\kappa} & 0<t\leq 1\\
t^{-\mu_\infty/\kappa} & t\geq 1
\end{cases}
\end{equation*}
for $t>0$ where $\mu_0$ and $\mu_\infty$ are those given by Theorem \ref{thm:OnDiagonal}. At this time, describing precisely the set of examples to which our methods apply is an open question. These questions will be explored in a forthcoming article.

\appendix
\section{Technical Estimates}\label{sec:TechnicalEstimates}

\begin{lemma}\label{lem:GroupEstForP}
Let $P$ be a positive homogeneous function and let $F\in\End(\mathbb{R}^a)$ be such that $\{t^F\}$ is non-expanding. If there exits $E\in\Exp(P)$ such that $[E,F]=0$, then, for any $\delta>0$, there exists $\rho>0$ for which
\begin{equation*}
\rho P(t^{F}\eta)\leq \delta P(\eta)
\end{equation*}
for all $\eta\in\mathbb{R}^a$ and $0<t\leq 1$.
\end{lemma}

\begin{proof}
Because $P(0)=0$, it suffices to prove that, for some $C>0$,
\begin{equation*}
P(t^F\eta)\leq CP(\eta)
\end{equation*}
for all $0<t\leq 1$ and non-zero $\eta\in\mathbb{R}^a$. Denote by $S$ the unital level set of $P$ and observe that
\begin{equation*}
K=\{t^F\eta:0<t\leq 1,\,\eta\in S\}
\end{equation*}
is a bounded set because $S$ is compact and $\{t^F\}_{0<t\leq 1}$ is uniformly bounded. By virtue of the continuity of $P$, it follows that, for every $0<t\leq 1$ and $\eta\in S$,
\begin{equation*}
P(t^F\eta)\leq \sup_{\xi\in K}P(\xi)=:C<\infty.
\end{equation*}
Given any non-zero $\eta\in\mathbb{R}^a$, Proposition 4.1 of \cite{BR22} guarantees that $\eta=r^E\eta_0$ where $r=P(\eta)>0$, $\eta_0\in S$, and $E$ is as in the statement of the lemma. Using the fact that $E$ and $F$ commute, we have
\begin{equation*}
P(t^F\eta)=P(r^Et^F\eta_0)=rP(t^F\eta_0)\leq rC=CP(\eta),
\end{equation*}
as desired.
\end{proof}

\begin{lemma}\label{lem:SumEst}
Let $P$ be a positive homogeneous function on $\mathbb{R}^a$. Then, for any positive constants $\rho_1,\rho_2$, there exists $\epsilon>0$ for which
\begin{equation*}
\epsilon P(\xi)\leq\rho_1 P(\zeta+\xi)+\rho_2P(\zeta)
\end{equation*}
for all $\xi,\zeta\in\mathbb{R}^a$.
\end{lemma}

\begin{proof}
Let $\rho_1,\rho_2>0$ and define $R:\mathbb{R}^{2a}\to\mathbb{R}$ by
\begin{equation*}
R(\xi,\zeta)=\rho_1 P(\zeta+\xi)+\rho_2 P(\zeta)
\end{equation*}
for $(\xi,\zeta)\in \mathbb{R}^{a}\times\mathbb{R}^a=\mathbb{R}^{2a}$. It is clear that $R$ is a continuous and positive-definite function on $\mathbb{R}^{2a}$. For any $E\in\Exp(P)$, observe that
\begin{equation*}
R(t^{E\oplus E}(\xi,\zeta))=R(t^E\xi,t^E\zeta)=\rho_1P(t^E\zeta+t^E\xi)+\rho_2P(t^E\zeta)=tR(\xi,\zeta)
\end{equation*}
for all $t>0$ and $(\xi,\zeta)\in\mathbb{R}^{2a}$. Thus, $R$ is homogeneous with respect to $E\oplus E$ and, because $\{t^E\}$ is a contracting group on $\mathbb{R}^a$, it is evident that $\{t^{E\oplus E}\}$ is a contracting group on $\mathbb{R}^{2a}$. In view of Proposition 1.1 of \cite{BR22}, we conclude that $R$ is a positive homogeneous function. Denote by $S$ the unital level set of $R$ and, because $S$ is compact and $(\xi,\zeta)\mapsto P(\xi)$ is continuous and does not identically vanish on $S$, we have
\begin{equation*}
\epsilon:=\left(\sup_{(\xi,\zeta)\in S}P(\xi)\right)^{-1}>0.
\end{equation*}
Given a non-zero $(\xi,\zeta)\in\mathbb{R}^{2a}$, Proposition 4.1 of \cite{BR22} ensures that $(\xi,\zeta)=r^{E\oplus E}(\xi_0,\zeta_0)=(r^E\xi_0,r^E\zeta_0)$ where $r=R(\xi,\zeta)>0$, $(\xi_0,\zeta_0)\in S$ and $E\in\Exp(P)$. With this, we observe that
\begin{equation*}
\epsilon P(\xi)=\epsilon r P(\xi_0)\leq r=R(\xi,\zeta)=\rho_1 P(\zeta+\xi)+\rho_2P(\zeta).
\end{equation*}
Since this inequality holds trivially when $(\xi,\zeta)=(0,0)$, the proof is complete.
\end{proof}

By making analogous arguments to those in the proof above, we easily obtain the following lemma.
\begin{lemma}\label{lem:EasySumEst}
Let $P$ be a positive homogeneous function on $\mathbb{R}^a$. Then, for any positive constants $\rho_1,\rho_2$, there exists $\epsilon>0$ for which
\begin{equation*}
\epsilon P(\zeta+\xi)\leq \rho_1P(\xi)+\rho_2P(\zeta)
\end{equation*}
for all $\xi,\zeta\in\mathbb{R}^a$.
\end{lemma}
\begin{lemma}\label{lem:CompareEst}
Let $P_1$ and $P_2$ be positive homogeneous functions on $\mathbb{R}^a$ and let $E_1\in\Exp(P_1)$ and $E_2\in\Exp(P_2)$. If $[E_1,E_2]=0$ and $\{t^{E_1-E_2}\}$ is non-expanding,
then, for any $\delta>0$, there are constants $\rho,M>0$ for which
\begin{equation*}
\rho P_1(t^{E_1-E_2}\eta)\leq M+\delta P_2(\eta)
\end{equation*}
for all $\eta\in\mathbb{R}^a$ and $0<t\leq 1$.
\end{lemma}

\begin{proof}
Let $\delta>0$. We shall first prove that, there are constants $\rho,M>0$ for which
\begin{equation*}
\rho P_1(\eta)\leq M+\delta P_2(\eta)
\end{equation*}
for all $\eta\in\mathbb{R}^a$. We note that this is the desired inequality evaluated at $t=1$; we shall obtain the full inequality by scaling. Now, denote by $S$ and $B$ the unital level set and unit ball associated to $P=\delta P_2$. Because $B$ is relatively compact and $\{t^{E_1-E_2}\}_{0<t\leq 1}$ is uniformly bounded, 
\begin{equation*}
K=\{r^{E_1-E_2}\eta:0<r\leq 1,\,\eta\in B\}
\end{equation*}
is bounded and hence relatively compact. Since $P_2$ is positive and does not vanish identically on $K$,
\begin{equation*}
\rho:=\left(\sup_{\xi\in K}P_1(\xi)\right)^{-1}>0. 
\end{equation*}
Now, for any $\eta\in \mathbb{R}^{a}\setminus B$, it follows from Proposition 4.1 of \cite{BR22} that $\eta=r^{-E_2}\eta_0$ for $r=1/P(\eta)=1/(\delta P_2(\eta))\leq 1$ and $\eta_0\in S$. Consequently,
\begin{equation*}
\rho P_1(\eta)=\rho P_1(r^{-E_2}\eta_0)=\frac{\rho}{r}P_1(r^{E_1-E_2}\eta_0)\leq \frac{1}{r}=\delta P_2(\eta).
\end{equation*}
Of course, $P_1(\eta)$ is bounded on $B$ and so it follows that, for some constant $M>0$,
\begin{equation*}
\rho P_1(\eta)\leq M+\delta P_2(\eta)
\end{equation*}
for all $\eta\in \mathbb{R}^a$, as claimed. Finally, for $0<t\leq 1$ and $\eta\in\mathbb{R}^a$, we apply the preceding inequality to $t^{-E_2}\eta$ to see that
\begin{equation*}
\rho P_1(t^{E_1-E_2}\eta)=t\rho P_1(t^{-E_2}\eta)\leq t\left(M+\delta P_2(t^{-E_2}\eta)\right)\leq tM+\delta P_2(\eta)\leq M+\delta P_2(\eta)
\end{equation*}
as desired.
\end{proof}

\begin{proof}[Proof of Lemma \ref{lem:P2P1Estimates}]
We shall first prove the estimate \eqref{eq:EstimateForLargeTime}. For the lower estimate in \eqref{eq:EstimateForLargeTime}, simultaneous appeals to Lemmas \ref{lem:GroupEstForP} and \ref{lem:CompareEst} give positive constants $\rho_1$, $\rho_2$, and $M$ for which
\begin{equation*}
\rho_1 P_1(\zeta)\leq M+P_2(\zeta)\hspace{1cm}\mbox{and}\hspace{1cm}\rho_2 P_1(t^{E_1-E_2}\eta)\leq P_1(\eta)/2
\end{equation*}
for $\zeta,\eta\in\mathbb{R}^a$ and $0<t\leq 1$. With the constants $\rho_1$ and $\rho_2$ in hand, we appeal to Lemma \ref{lem:SumEst} to find a positive constant $\epsilon>0$ for which 
\begin{equation*}
\epsilon P_1(\xi)\leq \rho_1 P_1(\zeta+\xi)+\rho_2 P_1(\zeta)
\end{equation*}
for all $\xi,\zeta\in\mathbb{R}^a$. Consequently, for $\xi,\eta\in\mathbb{R}^a$ and $0<t\leq 1$,
\begin{equation*}
\epsilon P_1(\xi)\leq \rho_1 P_1(t^{E_1-E_2}\eta+\xi)+\rho_2 P_1(t^{E_1-E_2}\eta)\leq M+P_2(t^{E_1-E_2}\eta+\xi)+P_1(\eta)/2.
\end{equation*}
Upon setting $C=\min\{\epsilon,1/2\}>0$, it follows that
\begin{equation*}
C(P_1(\eta)+P_1(\xi))-M\leq P_1(\eta)/2+\epsilon P_1(\xi)-M\leq P_1(\eta)+P_2(t^{E_1-E_2}\eta+\xi)
\end{equation*}
for all $\eta,\xi\in\mathbb{R}^a$ and $0<t\leq 1$. For the upper estimate in \eqref{eq:EstimateForLargeTime}, we first make an appeal to Lemma \ref{lem:CompareEst} to obtain positive constants $M'$ and $C_1$ for which
\begin{equation*}
P_1(\eta)\leq M'+C_1P_2(\eta)
\end{equation*}
for all $\eta\in\mathbb{R}^a$. By virtue of Lemma \ref{lem:EasySumEst}, there is a positive constant $C_2$ for which
\begin{equation*}
P_2(\zeta+\xi)\leq C_2 (P_2(\zeta)+P_2(\xi))
\end{equation*}
for all $\xi,\zeta\in\mathbb{R}^a$. Finally, an appeal to Lemma \ref{lem:GroupEstForP} guarantees a positive constant $C_3$ for which
\begin{equation*}
P_2(t^{E_1-E_2}\eta)\leq C_3 P_2(\eta)
\end{equation*} 
for all $\eta\in\mathbb{R}^a$ and $0<t\leq 1$. Upon combining these estimates, we obtain
\begin{eqnarray*}
P_1(\eta)+P_2(t^{E_1-E_2}\eta+\xi)&\leq& M'+C_1P_2(\eta)+C_2P_2(t^{E_1-E_2}\eta)+C_2P_2(\xi)\\
&\leq& M'+(C_1+C_2C_3)P_2(\eta)+C_2P_2(\xi)
\end{eqnarray*}
for all $\eta,\xi\in\mathbb{R}^a$ and $0<t\leq 1$. Our desired (upper) estimate in \eqref{eq:EstimateForLargeTime} follows immediately by taking $C'=\max\{C_2,C_1+C_2C_3\}$ and so the proof of \eqref{eq:EstimateForLargeTime} is complete.

To prove \eqref{eq:EstimateForSmallTime}, we first make an appeal to Lemma \ref{lem:CompareEst} to obtain positive constants $\rho$ and $M$ for which
\begin{equation*}
\rho P_1(t^{E_1-E_2}\eta)\leq M+P_2(\eta)/2
\end{equation*}
for all $\eta\in\mathbb{R}^a$ and $0<t\leq 1$. With this $\rho$, an appeal to Lemma \ref{lem:SumEst} yields $\epsilon>0$ for which
\begin{equation*}
\epsilon P_1(\xi)\leq P_1(\zeta+\xi)+\rho P_1(\zeta)
\end{equation*}
for all $\xi,\zeta\in\mathbb{R}^a$. Upon setting $C=\min\{\epsilon,1/2\}$, we obtain
\begin{eqnarray*}
C P_1(\xi)+P_2(\eta))-M&\leq& \epsilon P_1(\xi)+P_2(\eta)/2-M\\
&\leq &P_1(t^{E_1-E_2}\eta+\xi)+\rho P_1(t^{E_1-E_2}\eta)+P_2(\eta)/2-M\\
&\leq&P_1(t^{E_1-E_2}\eta+\xi)+P_2(\eta)
\end{eqnarray*}
for all $\eta,\xi\in\mathbb{R}^a$ and $0<t\leq 1$ and this is precisely the lower estimate in \eqref{eq:EstimateForLargeTime}. Making use of Lemmas \ref{lem:CompareEst} and \ref{lem:EasySumEst}, the upper estimate in \eqref{eq:EstimateForLargeTime} is established in a similar way to that for \eqref{eq:EstimateForSmallTime}; we leave the remaining details to the reader.
\end{proof}

Our final goal in this appendix is to prove Proposition \ref{prop:RTildeVeryNice}. Before the proof, we present two lemmas.

\begin{lemma}
Assume the notation and hypotheses of Proposition \ref{prop:RTildeVeryNice} and define
\begin{equation*}
\widetilde{P}_\infty(\xi)=P_1(\eta)+P_2(-Q(\zeta))\hspace{1cm}\mbox{and}\hspace{1cm}\widetilde{\mathcal{E}}_\infty(\xi)=\widetilde{P}(\xi)-\widetilde{P}_\infty(\xi)
\end{equation*}
for $\xi=(\eta,\zeta)\in\mathbb{R}^d$. Then $\widetilde{\mathcal{E}}_\infty(\xi)$ is subhomogeneous with respect to $G=E_1\oplus F_2$.
\end{lemma}
\begin{proof}
Let $\epsilon>0$ and $K\subseteq\mathbb{R}^d$ be a compact set. Since $\{t^{E_1-E_2}\}$ is contracting and $P_2$ and $Q$ are continuous, we can find $0<t_0\leq 1$ for which
\begin{equation*}
\abs{(P_1(\eta)+P_2(t^{E_1-E_2}\eta-Q(\zeta))-\widetilde{P}_\infty(\xi)}=\abs{P_2(t^{E_1-E_2}\eta-Q(\zeta))-P_2(-Q(\zeta))}<\epsilon
\end{equation*}
for all $0<t\leq t_0$ and $\xi=(\eta,\zeta)\in K$. Thus, for every $0<t\leq t_0$ and $\xi=(\eta,\zeta)\in K$, we have
\begin{eqnarray*}
\abs{\widetilde{\mathcal{E}}_\infty(t^{G}\xi)}&=&\abs{P_1(t^{E_1}\eta)+P_2(t^{E_1}\eta-Q(t^{F_2}\zeta))-P_1(t^{E_1}\eta)-P_2(-Q(t^{F_2}\zeta))}\\
&=&\abs{P_2(t^{E_1}\eta-t^{E_2}Q(\zeta))-P_2(-t^{E_2}Q(\zeta))}\\
&=&t\abs{P_2(t^{E_1-E_2}\eta-Q(\zeta)-P_2(-Q(\zeta))}\\
&\leq&\epsilon t,
\end{eqnarray*}
as desired.
\end{proof}

The following lemma asserts loosely that $\widetilde{P}$ is approximately homogeneous with respect to $G$ in small time whenever $\{t^{E_1-E_2}\}$ is contracting.

\begin{lemma}
Assume the notation and hypotheses of Proposition 4.3. Then, there exists a compact set $K\subseteq\mathbb{R}^d$ such that, for any $\tau>0$ and $0<\epsilon<1$, there exists $0<t_0\leq \tau$ for which
\begin{equation*}
\mathcal{O}:=\{\xi\in\mathbb{R}^d:\xi=0\,\,\mbox{or}\,\,\xi=t^G\xi'\,\,\mbox{for}\,\,0<t< t_0\,\,\mbox{and}\,\,\xi'\in K\}
\end{equation*}
is an open neighborhood of $0$ and, for each non-zero $\xi=t^G\xi'\in \mathcal{O}$,
\begin{equation*}
(1-\epsilon)t\leq \widetilde{P}(\xi)=\widetilde{P}(t^G\xi')\leq (1+\epsilon)t
\end{equation*}
\end{lemma}
\begin{proof}
Let $K=S_{\widetilde{P}_\infty}$ be the unital level set of $\widetilde{P}_\infty$ and, by an appeal to the preceding lemma, let $0<t_0\leq \tau$ be such that
\begin{equation*}
\abs{\widetilde{\mathcal{E}}_\infty(t^G\xi')}\leq \epsilon t
\end{equation*}
whenever $0<t\leq t_0$ and $\xi'\in K$. In this notation, we observe that $\mathcal{O}$, as defined in the statement, coincides with the $\widetilde{P}_\infty$-adapted open ball $B_{t_0}=\{\xi\in\mathbb{R}^d:\widetilde{P}_\infty(\xi)<t_0\}$ since $G\in\Exp(\widetilde{P}_\infty)$; in particular, $\mathcal{O}=B_{t_0}$ is necessarily an neighborhood of $0$ (see Proposition 4.1 of \cite{BR22}). Further, for each non-zero $\xi=t^G\xi'\in\mathcal{O}=B_{t_0}$, we have $\widetilde{P}_\infty(\xi)=t$ and therefore
\begin{equation*}
\abs{\widetilde{P}(\xi)-t}=\abs{\widetilde{\mathcal{E}}_\infty(\xi)}=\abs{\widetilde{\mathcal{E}}_\infty(t^G\xi')}\leq \epsilon t
\end{equation*}
so that $(1-\epsilon)t\leq \widetilde{P}(\xi)=\widetilde{P}(t^G\xi')\leq (1+\epsilon)t$.
\end{proof}

\begin{proof}[Proof of Proposition \ref{prop:RTildeVeryNice}]
In view of Lemma \ref{lem:RTildeNice}, we need only to prove the sufficiency of the subhomogeneity condition. Specifically, we prove that $\widetilde{R}(\xi)=o(\widetilde{P}(\xi))$ as $\xi\to 0$ when $\{t^{E_1-E_2}\}$ is contracting and $\widetilde{R}$ is subhomogeneous with respect to $G$. To this end, let $\epsilon>0$ and take $K$ as in the preceding lemma. With our assumption that $\widetilde{R}$ is subhomogeneous with respect to $G$, let $\tau>0$ be such that $\abs{\widetilde{R}(t^G\xi')}\leq \epsilon t/2$ whenever $\xi'\in K$ and $0<t\leq \tau$. By an appeal to the preceding lemma, let $0<t_0\leq \tau$ for which
\begin{equation*}
 \widetilde{P}(\xi)=\widetilde{P}(t^G\xi')\geq t/2
\end{equation*}
whenever $\xi=t^G\xi'$ is a non-zero member of the open neighborhood
\begin{equation*}
\mathcal{O}=\{\xi\in\mathbb{R}^d:\xi=0\,\,\mbox{or}\,\,\xi=t^G\xi'\,\,\mbox{for}\,\,0<t<t_0\,\,\mbox{and}\,\,\xi'\in K\}
\end{equation*}
of $0$ in $\mathbb{R}^d$. Thus, for any non-zero $\xi=t^G\xi'\in\mathcal{O}$,
\begin{equation*}
\abs{\widetilde{R}(\xi)}=\abs{\widetilde{R}(t^G\xi')}\leq\epsilon t/2\leq \epsilon \widetilde{P}(t^G\xi')=\epsilon\widetilde{P}(\xi).
\end{equation*}
By the continuity of $\widetilde{R}$ and $\widetilde{P}$, this estimate clearly holds when $\xi=0$ and so we have shown that $\widetilde{R}(\xi)=o(\widetilde{P}(\xi))$ as $\xi\to 0$ as was asserted.
\end{proof}

\noindent\textbf{\large Acknowlegements:} Evan Randles would like to thank Cornell University and Colby College for financial support and Cornell University for hosting him during his 2021-2022 sabbatical leave from Colby College. Laurent Saloff-Coste is partially supported by NSF grant DMS-1707589 and DMS-2054593.

\end{document}